\documentclass[11pt]{amsart}

\usepackage{amssymb,amsthm,amsmath,amsxtra}
\usepackage{color}
\usepackage[usenames,dvipsnames]{xcolor}
\usepackage[all]{xy}
\usepackage{fullpage}
\usepackage{comment}
\usepackage{enumerate}
\usepackage{bbm}
\usepackage{hyperref}
\hypersetup{colorlinks=true,urlcolor=blue,citecolor=blue,linkcolor=blue}
\usepackage{colonequals} 
\usepackage{mathrsfs}
\usepackage{multirow}
\allowdisplaybreaks 
\numberwithin{equation}{subsection}

\usepackage[T1]{fontenc}

\theoremstyle{plain}
\newtheorem{mainthm}[equation]{Main Theorem}
\newtheorem{thm}[equation]{Theorem}
\newtheorem{prop}[equation]{Proposition}
\newtheorem{lem}[equation]{Lemma} 
\newtheorem{cor}[equation]{Corollary}

\newtheorem{conj}[equation]{Conjecture}

\newtheorem*{cor*}{Corollary}
\newtheorem*{prob*}{Problem}
\newtheorem*{thm*}{Theorem}
\newtheorem*{thma*}{Theorem A}
\newtheorem*{thmb*}{Theorem B}

\theoremstyle{remark}
\newtheorem{exm}[equation]{Example}
\newtheorem{defn}[equation]{Definition}
\newtheorem{rmk}[equation]{Remark}

\newenvironment{enumalph}
{\begin{enumerate}}
{\end{enumerate}}

\DeclareMathOperator{\Frob}{Frob}
\DeclareMathOperator{\Gal}{Gal}
\DeclareMathOperator{\GL}{GL}

\DeclareMathOperator{\Nm}{Nm}

\DeclareMathOperator{\Tr}{Tr}
\DeclareMathOperator{\Res}{Res}

\newcommand{\dD}{\mathrm{d}}

\newcommand{\Fsf}{\mathsf F}
\newcommand{\Lsf}{\mathsf L}

\newcommand{\defi}[1]{\emph{\textsf{#1}}} 				

\newcommand{\tyS}{S}

\newcommand{\C}{\mathbb C}
\newcommand{\F}{\mathbb F}

\newcommand{\PP}{\mathbb P}
\newcommand{\Q}{\mathbb Q}

\newcommand{\Z}{\mathbb Z}

\newcommand{\eps}{\epsilon}

\usepackage{tipx}

\newcounter{DworkClusters}
\newcounter{L2L2Clusters}

\newcommand{\qstroke}{%
  \text{\ooalign{\hidewidth\raisebox{-0.1ex}{-\kern-.25em}\hidewidth\cr$q$\cr}\hspace{0.2ex}}%
}
\newcommand{\qq}{q^\times}
\newcommand{\ZZ}{\mathbb Z}

\newcommand{\calF}{\mathcal{F}}

\newcommand{\ii}{\sqrt{-1}}
\newcommand{\Fil}{\mathcal{F}}

\newcommand{\psmod}[1]{~(\textup{\text{mod}}~{#1})}

\usepackage{marginnote}

\newcommand{\fracpadding}{}
\newcommand{\setfracpadding}[1][2pt]{%
  \sbox0{$\frac{1}{2}$}%
  \dimen0=\ht0 \advance\dimen0 #1\relax
  \dimen2=\dp0 \advance\dimen2 #1\relax
  \edef\fracpadding{\vrule width 0pt height \the\dimen0 depth \the\dimen2\relax}%
}
\newcommand{\afrac}[2]{\fracpadding\frac{#1}{#2}}

\newcommand{\frakp}{\mathfrak{p}}
\newcommand{\frakn}{\mathfrak{n}}

\newcommand{\legen}[2]{\biggl(\displaystyle{\frac{#1}{#2}}\biggr)}

\newcommand{\qm}{\mkern-4mu}

\allowdisplaybreaks

\begin{document} 

\title[Hypergeometric K3 Quartic Pencils]{Hypergeometric decomposition of \\ symmetric K3 quartic pencils}

\author{Charles F. Doran}
\address{University of Alberta, Department of Mathematics, Edmonton, AB Canada}
\email{doran@math.ualberta.edu}

\author{Tyler L. Kelly}
\address{School of Mathematics, University of Birmingham, Edgbaston, Birmingham, UK, B15 2TT}
\email{t.kelly.1@bham.ac.uk}

\author{Adriana Salerno}
\address{Department of Mathematics, Bates College, 3 Andrews Rd., Lewiston, ME 04240, USA}
\email{asalerno@bates.edu}

\author{Steven Sperber}
\address{School of Mathematics, University of Minnesota, 206 Church Street SE, Minneapolis, MN 55455 USA}
\email{sperber@umn.edu}

\author{John Voight}
\address{Department of Mathematics, Dartmouth College, 6188 Kemeny Hall, Hanover, NH 03755, USA}
\email{jvoight@gmail.com}

\author{Ursula Whitcher}
\address{Mathematical Reviews, 416 Fourth St, Ann Arbor, MI 48103, USA}
\email{uaw@umich.edu}

\date{\today}
\setcounter{tocdepth}{1}

\begin{abstract}
We study the hypergeometric functions associated to five one-parameter deformations of Delsarte K3 quartic hypersurfaces in projective space. We compute all of their Picard--Fuchs differential equations; we count points using Gauss sums and rewrite this in terms of finite field hypergeometric sums; then we match up each differential equation to a factor of the zeta function, and we write this in terms of global $L$-functions.  This computation gives a complete, explicit description of the motives for these pencils in terms of hypergeometric motives.
\end{abstract}

\maketitle

\tableofcontents

\section{Introduction}

\subsection{Motivation}

There is a rich history of explicit computation of hypergeometric functions associated to certain pencils of algebraic varieties.  Famously, in the 1950s, Igusa \cite{Igusa} studied the Legendre family of elliptic curves and found a spectacular relation between the ${}_2F_1$-hypergeometric Picard--Fuchs differential equation satisfied by the holomorphic period and the trace of Frobenius.
More generally, the link between the study of Picard--Fuchs equations and point counts via hypergeometric functions has intrigued many mathematicians.  Clemens \cite{clemens} referred to this phenomenon as ``Manin's unity of mathematics.''  Dwork studied the now-eponymous Dwork pencil \cite[\S 6j, p.~73]{padic}, and Candelas--de la Ossa--Rodr\'{i}guez-Villegas considered the factorization of the zeta function for the Dwork pencil of Calabi--Yau threefolds in \cite{CORV,CORV2}, linking physical and mathematical approaches.  More recently, given a finite-field hypergeometric function defined over $\Q$, Beukers--Cohen--Mellit \cite{BCM} construct a variety whose trace of Frobenius is equal to the finite field hypergeometric sum up to certain trivial factors. 

\subsection{Our context}

In this paper, we provide a complete factorization of the zeta function and more generally a factorization of the $L$-series for some pencils of Calabi--Yau varieties, namely, families of K3 surfaces.  We study certain Delsarte quartic pencils in $\PP^3$ (also called \emph{invertible pencils}) which arise naturally in the context of mirror symmetry, listed in \eqref{table:5families}.  Associated to each family we have a discrete group of symmetries acting symplectically (i.e., fixing the holomorphic form).
Our main theorem (Theorem \ref{mainthm} below) shows that hypergeometric functions are naturally associated to this collection of Delsarte hypersurface pencils in two ways: as Picard--Fuchs differential equations and as traces of Frobenius yielding point counts over finite fields. 
\begin{equation} \label{table:5families}
\begin{tabular}{c|c|cc}
 Pencil & Equation	&  Symmetries & Bad primes \\
\hline	\hline
\rule{0pt}{2.5ex}   $\Fsf_4$ & $x_0^4+x_1^4 + x_2^4 + x_3^4 - 4\psi x_0x_1x_2x_3$ 	& $\mu_4 \times \mu_4$	& $2$ \\
 $\Fsf_1\Lsf_3$ & $x_0^4 + x_1^3x_2 + x_2^3x_3 + x_3^3x_1 - 4\psi x_0x_1x_2x_3$ & $\mu_7$ & $2,7$ \\
 $\Fsf_2\Lsf_2$ & $x_0^4 + x_1^4 + x_2^3x_3 + x_3^3x_2 - 4\psi x_0x_1x_2x_3$	& $\mu_8$ & $2$ \\
 $\Lsf_2\Lsf_2$ & $x_0^3x_1 + x_1^3x_0 + x_2^3x_3 + x_3^3x_2 - 4\psi x_0x_1x_2x_3$ & $\mu_4 \times \mu_2$ & $2$ \\
 $\Lsf_4$ & $x_0^3x_1 + x_1^3x_2 + x_2^3x_3 + x_3^3 x_0 - 4\psi x_0x_1x_2x_3$ & $\mu_5$ & $2,5$
\end{tabular}
\end{equation}
Here we write $\mu_n$ for the group of $n$th roots of unity. The labels $\Fsf$ and $\Lsf$ stand for ``Fermat" and ``loop", respectively, as in \cite{commonfactor}. 

In previous work \cite{commonfactor}, we showed that these five pencils share a common factor in their zeta functions, a polynomial of degree 3 associated to the hypergeometric Picard--Fuchs differential equation satisfied by the holomorphic form---see also recent work of Kloosterman \cite{Klo17}. Also of note is that the pencils are also related in that one can take a finite group quotient of each family and find that they are then birational to one another \cite{BvGK}.  However, these pencils (and their zeta functions) are \emph{not} the same!  
In this article, we investigate the remaining factors explicitly (again recovering the common factor). In fact, we show that each pencil is associated with a distinct and beautiful collection of auxiliary hypergeometric functions. 

\subsection{Notation}

We use the symbol $\diamond \in \calF=\{\Fsf_4,\Fsf_2\Lsf_2,\Fsf_1\Lsf_3,\Lsf_2\Lsf_2,\Lsf_4\}$ to signify one of the five \textup{K3}\/ pencils in \textup{\eqref{table:5families}}.  Let $\psi \in \Q \smallsetminus \{0,1\}$.  Let $S=S(\diamond,\psi)$ be the set of bad primes in \textup{\eqref{table:5families}} together with the primes dividing the numerator or denominator of either $\psi^4$ or $\psi^4-1$.  Then for $p \not \in S$, the K3 surface $X_{\diamond,\psi}$ has good reduction at $p$, and for $q=p^r$ we let
\begin{equation} 
P_{\diamond,\psi,q}(T) \colonequals \det(1 - \Frob_p^r T \,|\, H_{\textup{\'et},\textup{prim}}^2(X_{\diamond,\psi},\Q_\ell)) \in 1+T\Z[T] 
\end{equation}
be the characteristic polynomial of the $q$-power Frobenius acting on primitive second degree \'etale cohomology for $\ell \neq p$, which is independent of $\ell$. (Recall that the primitive cohomology of a hypersurface in $\PP^n$ is orthogonal to the hyperplane class.) Accordingly, the zeta function of $X_{\diamond,\psi}$ over $\F_q$ is 
\begin{equation} 
Z_q(X_{\diamond,\psi},T) = \frac{1}{(1-T)(1-qT)P_{\diamond,\psi,q}(T)(1-q^2 T)}.\end{equation}
The Hodge numbers of $X_{\diamond,\psi}$ imply that the polynomial $P_{\diamond,\psi,q}(T)$ has degree 21.  Packaging these together, we define the (incomplete) $L$-series
\begin{equation} 
L_S(X_{\diamond,\psi},s) \colonequals \prod_{p \not \in S} P_{\diamond,\psi,p}(p^{-s})^{-1}
\end{equation}
convergent for $s \in \C$ in a right half-plane.  

Our main theorem explicitly identifies the Dirichlet series $L_S(X_{\diamond,\psi},s)$ as a product of hypergeometric $L$-series. To state this precisely, we now introduce a bit more notation.  Let $\pmb{\alpha}=\{\alpha_1,\dots,\alpha_d\}$ and $\pmb{\beta}=\{\beta_1,\dots,\beta_d\}$ be multisets with $\alpha_i,\beta_i \in \Q_{\geq 0}$ that modulo $\Z$ are disjoint. We associate a field of definition $K_{\pmb{\alpha},\pmb{\beta}}$ to $\pmb{\alpha},\pmb{\beta}$, which is an explicitly given finite abelian extension of $\Q$.  For certain prime powers $q$ and $t \in \F_q$, there is a \emph{finite field hypergeometric sum} $H_q(\pmb{\alpha};\pmb{\beta}\,|\,t) \in K_{\pmb{\alpha},\pmb{\beta}}$ defined by Katz \cite{Katz} as a finite field analogue of the complex hypergeometric function, normalized by McCarthy \cite{DM}, extended by Beukers--Cohen--Mellit \cite{BCM}, and pursued by many authors: see section \ref{defn:hgmdef} for the definition and further discussion, and section \ref{defn:hgmdef-hybrid} for an extension of this definition.  We package together the exponential generating series associated to these hypergeometric sums into an $L$-series $L_S( H(\pmb{\alpha};\pmb{\beta}\,|\,t), s)$: see section~\ref{S:finalauto} for further notation.

\subsection{Results}

Our main theorem is as follows.

\begin{mainthm} \label{mainthm}
The following equalities hold with $t=\psi^{-4}$ and $S=S(\diamond,\psi)$.
\begin{enumalph}
\item For the Dwork pencil $\Fsf_4$, 
\begin{align*}
L_S(X_{\Fsf_4,\psi}, s) &= L_S( H(\tfrac{1}{4}, \tfrac{1}{2}, \tfrac{3}{4}; 0, 0, 0\,|\, t), s) \\
&\qquad \cdot L_S( H(\tfrac{1}{4}, \tfrac{3}{4}; 0, \tfrac{1}{2} \,|\, t), s-1, \phi_{-1})^3 \\
&\qquad \cdot L_S( H(\tfrac{1}{2}; 0 \,|\, t) , \Q(\sqrt{-1}), s-1, \phi_{\sqrt{-1}})^6 
\end{align*}
where
\begin{equation} \label{eqn:yupphichar0}
\begin{aligned}
\phi_{-1}(p) & \colonequals \legen{-1}{p} = (-1)^{(p-1)/2} &  & \text{ is associated to $\Q(\sqrt{-1}) \,|\, \Q$, and} \\
\phi_{\sqrt{-1}}(\frakp)& \colonequals \legen{\sqrt{-1}}{\frakp}=(-1)^{(\Nm(\frakp)-1)/4} & & \text{ is associated to $\Q(\zeta_8)\,|\,\Q(\sqrt{-1})$.}
\end{aligned}
\end{equation}
\item For the Klein--Mukai pencil $\Fsf_1 \Lsf_3$, 
\begin{align*}
L_S(X_{\Fsf_1\Lsf_3,\psi}, s) &= L_S( H(\tfrac{1}{4}, \tfrac{1}{2}, \tfrac{3}{4}; 0, 0, 0\,|\, t), s) \\
&\qquad \cdot L_S( H(\tfrac{1}{14}, \tfrac{9}{14}, \tfrac{11}{14}; 0, \tfrac{1}{4}, \tfrac{3}{4} \,|\, t^{-1}), \Q(\zeta_7), s-1) 
\end{align*}
where
\[ L_S( H(\tfrac{1}{14}, \tfrac{9}{14}, \tfrac{11}{14}; 0, \tfrac{1}{4}, \tfrac{3}{4} \,|\, t^{-1}), s) = L_S( H(\tfrac{3}{14}, \tfrac{5}{14}, \tfrac{13}{14}; 0, \tfrac{1}{4}, \tfrac{3}{4} \,|\, t^{-1}), s) \]
are defined over $K=\Q(\sqrt{-7})$.  
\item For the pencil $\Fsf_2 \Lsf_2$, 
\begin{align*}
L_S(X_{\Fsf_2\Lsf_2,\psi}, s) &= L_S( H(\tfrac{1}{4}, \tfrac{1}{2}, \tfrac{3}{4}; 0, 0, 0\,|\, t), s) \\
&\qquad \cdot L_S(\Q(\zeta_8)\,|\,\Q,s-1)^2 L_S( H(\tfrac{1}{4}, \tfrac{3}{4}; 0, \tfrac{1}{2} \,|\, t), s-1, \phi_{-1}) \\
&\qquad \cdot L_S( H(\tfrac{1}{2};0 \,|\, t) , \Q(\sqrt{-1}), s-1, \phi_{\sqrt{-1}}) \\
&\qquad \cdot L_S( H(\tfrac{1}{8}, \tfrac{5}{8}; 0, \tfrac{1}{4} \,|\, t^{-1}), \Q(\zeta_8), s-1, \phi_{\sqrt{2}})  
\end{align*}
where
\[
L_S( H(\tfrac{1}{8}, \tfrac{5}{8}; 0, \tfrac{1}{4} \,|\, t^{-1}), s) = L_S( H(\tfrac{3}{8}, \tfrac{7}{8}; 0, \tfrac{3}{4} \,|\, t^{-1}), s) \]
are defined over $K=\Q(\sqrt{-1})$, 
\begin{equation} \label{eqn:yupphichar23}
\begin{aligned}
\phi_{\sqrt{2}}(\frakp) & \colonequals \legen{\sqrt{2}}{\frakp} \equiv 2^{(\Nm(\frakp)-1)/4} \pmod{\frakp} &  & \text{ is associated to $\Q(\zeta_8,\sqrt[4]{2}) \,|\, \Q(\zeta_8)$,}
\end{aligned}
\end{equation}
and $L(\Q(\zeta_8)\,|\,\Q,s) \colonequals \zeta_{\Q(\zeta_8)}(s)/\zeta_\Q(s)$ is the ratio of the Dedekind zeta function of $\Q(\zeta_8)$ and the Riemann zeta function.
\item For the pencil $\Lsf_2 \Lsf_2$,
\begin{align*}
L_S(X_{\Lsf_2\Lsf_2,\psi}, s) &= L_S( H(\tfrac{1}{4}, \tfrac{1}{2}, \tfrac{3}{4}; 0, 0, 0\,|\, t), s) \\
&\qquad \cdot \zeta_{\Q(\sqrt{-1})}(s-1)^4  L_S( H(\tfrac{1}{4}, \tfrac{3}{4}; 0, \tfrac{1}{2} \,|\, t), s-1, \phi_{-1}) \\
&\qquad \cdot     L_S( H(\tfrac{1}{8}, \tfrac{3}{8}, \tfrac{5}{8}, \tfrac{7}{8}; 0, \tfrac{1}{4}, \tfrac{1}{2}, \tfrac{3}{4}\,|\, t), \Q(\sqrt{-1}), s-1, \phi_{\sqrt{-1}}\phi_{\psi})
\end{align*}
where
\begin{equation}
\begin{aligned}
\phi_{\psi}(p) & \colonequals \legen{\psi}{p} & & \text{ is associated to $\Q(\sqrt{\psi})\,|\,\Q$}.
\end{aligned}
\end{equation}
\item For the pencil $\Lsf_4$,
\begin{align*}
L_S(X_{\Lsf_4,\psi}, s) &= L_S( H(\tfrac{1}{4}, \tfrac{1}{2}, \tfrac{3}{4}; 0, 0, 0\,|\, t), s) \\
&\qquad \cdot \zeta(s-1)^2 L_S( H( \tfrac{1}{5}, \tfrac{2}{5}, \tfrac{3}{5}, \tfrac{4}{5}; 0, \tfrac{1}{4}, \tfrac{1}{2}, \tfrac{3}{4} \,|\, t^{-1}) , \Q(\zeta_5), s-1) .
\end{align*}
\end{enumalph}
\end{mainthm}

We summarize Theorem \ref{mainthm} for each of our five pencils in \eqref{table:alphabetaintro}: we list the degree of the $L$-factor, the hypergeometric parameters, and the base field indicating when it arises from base change.  A Dedekind (or Riemann) zeta function factor has factors denoted by -.
\begin{equation} \label{table:alphabetaintro}
{
\centering \setfracpadding
\begin{tabular}{c|c|c|c|c}
  Pencil   & Degree & $\pmb{\alpha}$ & $\pmb{\beta}$ & Base Field \\ \hline \hline
   & 3 & $\afrac{1}{4},\afrac{1}{2},\afrac{3}{4}$ & $0,0,0$ & $\Q$ \\ 
 $\Fsf_4$          & $2\cdot 3 = 6$ & $\afrac{1}{4}, \afrac{3}{4}$ & $0, \afrac{1}{2}$ & $\Q$  \\
           & $2\cdot 6 = 12$ & $\afrac{1}{2}$ & 0 & $\Q(\sqrt{-1})$, from $\Q$ \\ \hline
   \multirow{2}{*}{$\Fsf_1\Lsf_3$} & 3 & $\afrac{1}{4},\afrac{1}{2},\afrac{3}{4}$ & $0,0,0$ & $\Q$ \\ 
           & 18 & $\afrac{1}{14}, \afrac{9}{14},\afrac{11}{14}$ & $0, \afrac{1}{4}, \afrac{3}{4}$ & $\Q(\zeta_7)$, from $\Q(\sqrt{-7})$  \\
 \hline
   &  3 & $\afrac{1}{4},\afrac{1}{2},\afrac{3}{4}$ & $0,0,0$ & $\Q$ \\
	   & $3\cdot 2 = 6$ & - & - & $\Q(\zeta_8)$, from $\Q$ \\
\multirow{2}{*}{$\Fsf_2\Lsf_2$}		     	   & 2 & $\afrac{1}{4}, \afrac{3}{4}$ & $0, \afrac{1}{2}$ & $\Q$ \\  
   & 2 & $\afrac{1}{2}$ & $0$ & $\Q(\sqrt{-1})$, from $\Q$ \\ 

	   & 8 & $\afrac{1}{8}, \afrac{5}{8}$ & $0,\afrac{1}{4}$ & $\Q(\zeta_8)$,  from $\Q(\sqrt{-1})$ \\
\hline
 &  3 & $\afrac{1}{4},\afrac{1}{2},\afrac{3}{4}$ & $0,0,0$ & $\Q$ \\
\multirow{2}{*}{$\Lsf_2\Lsf_2$} 	   & $2\cdot 4 = 8$ & - & - & $\Q(\sqrt{-1})$, from $\Q$ \\
 	   & 2 & $\afrac{1}{4}, \afrac{3}{4}$ & $0, \afrac{1}{2}$ & $\Q$ \\  
	   & 8 & $\afrac{1}{8}, \afrac{3}{8}, \afrac{5}{8}, \afrac{7}{8}$ & $0,\afrac{1}{4}, \afrac{1}{2}, \afrac{3}{4}$ & $\Q(\sqrt{-1})$, from $\Q$ \\
	    \hline
   & 3 & $\afrac{1}{4},\afrac{1}{2},\afrac{3}{4}$ & $0,0,0$ & $\Q$ \\ 
  $\Lsf_4$	   	   & $1\cdot 2 = 2$ & - & - & $\Q$ \\
	   & 16 & $\afrac{1}{5}, \afrac{2}{5}, \afrac{3}{5}, \afrac{4}{5}$ & $0, \afrac{1}{4}, \afrac{1}{2}, \afrac{3}{4}$ & $\Q(\zeta_5)$, from $\Q$  \\

\end{tabular}}
\end{equation}

We extensively checked the equality of Euler factors in Main Theorem \ref{mainthm} in numerical cases (for many primes and values of the parameter $\psi$): for K3 surfaces we used code written by Costa \cite{CT}, and for the finite field hypergeometric sums we used code in Pari/GP and Magma \cite{Magma}, the latter available for download \cite{jvcode}.  See also Example \ref{exm:computedtable}.

Additionally, each pencil has the common factor $L_S(H(\frac{1}{4},\frac{1}{2},\frac{3}{4};0,0,0\,|\,t),s)$, giving another proof of a result in previous work \cite{commonfactor}: we have a factorization over $\Q[T]$
\begin{equation} 
P_{\diamond,\psi,p}(T) = Q_{\diamond,\psi,p}(T)R_{\psi,p}(T) 
\end{equation}
with $R_{\psi,p}(T)$ of degree $3$ independent of $\diamond \in \calF$.  
The common factor $R_{\psi,p}(T)$ is given by the action of Frobenius on the transcendental part in cohomology, and the associated completed $L$-function $L(H(\tfrac{1}{4},\tfrac{1}{2},\tfrac{3}{4};0,0,0 \,|\, t),s)$ is automorphic by  Elkies--Sch\"utt \cite{ES} (or see our summary \cite[\S 5.2]{commonfactor}): it arises from a family of classical modular forms on $\GL_2$ over $\Q$, and in particular, it has analytic continuation and functional equation.  See also recent work of Naskr\k{e}cki \cite{Naskrecki2}.

The remaining factors in each pencil in Main Theorem \ref{mainthm} yield a factorization of $Q_{\diamond,\psi,p}(T)$, corresponding to the algebraic part in cohomology (i.e., the Galois action on the N\'eron--Severi group).  In particular, the polynomial $Q_{\diamond,\psi,p}(T)$ has reciprocal roots of the form $p$ times a root of unity.  The associated hypergeometric functions are \emph{algebraic} by the criterion of Beukers--Heckman \cite{BH}, and the associated $L$-functions can be explicitly identified as Artin $L$-functions: see section \ref{sec:alghyp}. The algebraic $L$-series can also be explicitly computed when they are defined over $\Q$ \cite{CohenNotes,Naskrecki}. For example, if we look at the Artin $L$-series associated to the Dwork pencil $\Fsf_4$, Cohen has given the following $L$-series relations (see Proposition~\ref{Lseries gen cor F4}):
\begin{equation} \label{eqn:LSH}
\begin{aligned}
    L_S(H(\tfrac14, \tfrac34; 0, \tfrac12\,|\,\psi^{-4}), s, \phi_{-1}) &= L_S(s, \phi_{1-\psi^2})L_S(s,\phi_{-1-\psi^2}) \\
    L_S(H(\tfrac12;0\,|\,\psi^{-4}, \Q(\sqrt{-1}),s, \phi_{\sqrt{-1}}) &= L_S(s, \phi_{2(1-\psi^4)})L_S(s,\phi_{-2(1-\psi^4)}).
    \end{aligned}
\end{equation}
In particular, it follows that the minimal field of definition of the N\'eron--Severi group of $X_{\Fsf_4,\psi}$ is $\Q(\zeta_8,\sqrt{1-\psi^2},\sqrt{1+\psi^2})$.  
The expressions \eqref{eqn:LSH}, combined with Main Theorem \ref{mainthm}(a), resolve a conjecture of Duan \cite{LianDuan}. (For geometric constructions of the N\'eron--Severi group of $X_{\Fsf_4,\psi}$, see Bini--Garbagnati \cite{BG} and Kloosterman \cite{Klo17}; the latter also provides an approach to explicitly construct generators of the N\'eron--Severi group for four of the five families studied here, with the stubborn case $\Fsf_1 \Lsf_3$ still unresolved.)  Our theorem yields a explicit factorization of $Q_{\diamond,\psi,q}(T)$ for the Dwork pencil over $\F_q$ for any odd $q$ (see Corollary~\ref{Q for F4}).  As a final application, Corollary \ref{cor:factorzeta} shows how the algebraic hypergeometric functions imply the existence of a factorization of $Q_{\diamond,\psi,p}(T)$ over $\Q[T]$ depending only on $q$ for all families.

\begin{rmk}
Our main theorem can be rephrased as saying that the motive associated to primitive middle-dimensional cohomology for each pencil of K3 surfaces decomposes into the direct sum of hypergeometric motives as constructed by Katz \cite{Katz}.  These motives then govern both the arithmetic and geometric features of these highly symmetric pencils.  Absent a reference, we do not invoke the theory of hypergeometric motives in our proof.  
\end{rmk}

\subsection{Contribution and relation to previous work}\label{subsec:PreviousWork}

Our main result gives a complete decomposition of the cohomology for the five K3 pencils into hypergeometric factors.
We provide formulas for each pencil and for all prime powers $q$, giving an understanding of the pencil over $\Q$.  Addressing these subtleties, and consequently giving a result for the global $L$-function, are unique to our treatment.
Our point of view is computational and explicit; we expect that our methods will generalize and perhaps provide an algorithmic approach to the hypergeometric decomposition for other pencils. 

As mentioned above, the study of the hypergeometricity of periods and point counts enjoys a long-standing tradition.  Using his $p$-adic cohomology theory, Dwork \cite[\S 6j, p.~73]{padic} showed for the family $\Fsf_4$ that middle-dimensional cohomology decomposes into pieces according to three types of differential equations.  Kadir in her Ph.D.\ thesis \cite[\S 6.1]{kadir} recorded a factorization of the zeta function for $\Fsf_4$, a computation due to de la Ossa.  Building on the work of Koblitz \cite{Koblitz}, Salerno \cite[\S 4.2.1--4.2.2]{Salerno:Thesis} used Gauss sums in her study of the Dwork pencil in arbitrary dimension; under certain restrictions on $q$, she gave a formula for the number of points modulo $p$ in terms of truncated hypergeometric functions as defined by Katz \cite{Katz} as well as an explicit formula \cite[\S 5.4]{Sal13} for the point count for the family $\Fsf_4$.  Goodson \cite[Theorems 1.1--1.3]{Goodson} looked again at $\Fsf_4$ and proved a similar formula for the point counts over $\F_q$ for all primes $q=p$ and prime powers $q \equiv 1 \pmod{4}$.  In \cite{ling}, Fuselier et al.${}\!{}$  define an alternate finite field hypergeometric function (which differs from those by Katz, McCarthy, and Beukers-Cohen-Mellit) that makes it possible to prove identities that are analogous to well-known ones for classical hypergeometric functions. They then use these formulas to compute the number of points of certain hypergeometric varieties. 

Several authors have also studied the role of hypergeometric functions over finite fields for the Dwork pencil in arbitrary dimension, which for K3 surfaces is the family $\Fsf_4$ given in Main Theorem \ref{mainthm}.  McCarthy \cite{McC16} extended the definition of $p$-adic hypergeometric functions to provide a formula for the number of $\F_p$ points on the Dwork pencil in arbitrary dimension for all odd primes $p$, extending his results \cite{McC12} for the quintic threefold pencil.  Goodson \cite[Theorem 1.2]{Goodson2} then used McCarthy's formalism to rewrite the formula for the point count for the Dwork family in arbitrary dimension in terms of hypergeometric functions when $(n+1) \mid (q-1)$ and $n$ is even. 
See also Katz \cite{katz:dwork}, who took another look at the Dwork family.

Miyatani \cite[Theorem 3.2.1]{Miyatani} has given a general formula that applies to each of the five families, but with hypotheses on the congruence class of $q$.  It is not clear that one can derive our decomposition from the theorem of Miyatani.

A different line of research has been used to describe the factorization structure of the zeta function for pencils of K3 surfaces or Calabi--Yau varieties that can recover part of Main Theorem \ref{mainthm}.   
Kloosterman \cite{Klo07, Klo07:err} has shown that one can use a group action to describe the distinct factors of the zeta function for any one-parameter monomial deformation of a diagonal hypersurface in weighted projective space. He then applied this approach \cite{Klo17} to study the K3 pencils above and generalize our work on the common factor. His approach is different from both that work and the present one: he uses the Shioda map \cite{shioda} to provide a dominant rational map from a monomial deformation of a diagonal (Fermat) hypersurface to the K3 pencils.  The Shioda map has been used in the past \cite{BvGK} to recover the result of Doran--Greene--Judes matching Picard--Fuchs equations for the quintic threefold examples, and it was generalized to hypersurfaces of fake weighted-projective spaces and BHK mirrors \cite{Bini, Kel}.  Kloosterman also provides some information about the other factors in some cases.

\subsection{Proof strategy and plan of paper}

The proof of Main Theorem~\ref{mainthm} is an involved calculation.  Roughly speaking, we use the action of the group of symmetries to calculate hypergeometric periods and then use this decomposition to guide an explicit decomposition of the point count into finite field hypergeometric sums.  

Our proof follows three steps.  First, in section~\ref{S:PFEqns}, we find all Picard--Fuchs equations via the diagrammatic method developed by Candelas--de la Ossa--Rodr\'{i}guez-Villegas \cite{CORV, CORV2} and Doran--Greene--Judes \cite{DGJ08} for the Dwork pencil of quintic threefolds.  For each of our five families, we give the Picard--Fuchs equations in convenient hypergeometric form.  

Second, in section~\ref{S:PointCounts}, we carry out the core calculations by counting points over $\mathbb{F}_q$ for the corresponding pencils using Gauss sums.  This technique begins with the original method of Weil \cite{Weil}, extended by Delsarte and Furtado Gomida, and fully explained by Koblitz \cite{Koblitz}.  We then take these formulas and, using the hypergeometric equations found in section~\ref{S:PFEqns} and careful manipulation, link these counts to finite field hypergeometric functions.  The equations computed in section \ref{S:PFEqns} do not enter \emph{directly} into the proof of the theorem, but they give an answer that can then be verified by some comparatively straightforward manipulations.  These calculations confirm the match predicted by Manin's ``unity'' (see \cite{clemens}).

Finally, in section~\ref{S:mainthmpf}, we use the point counts from section~\ref{S:PointCounts} to explicitly describe the $L$-series for each pencil, and prove Main Theorem~\ref{mainthm}.  We conclude by relating the $L$-series to factors of the zeta function for each pencil.

\subsection{Acknowledgements}

The authors heartily thank Xenia de la Ossa for her input and many discussions about this project. They also thank Simon Judes for sharing his expertise, Frits Beukers, David Roberts, Fernando Rodr\'{i}guez-Villegas, and Mark Watkins for numerous helpful discussions, Edgar Costa for sharing his code for computing zeta functions, and the anonymous referee for helpful corrections and comments.  The authors would like to thank the American Institute of Mathematics (AIM) and its SQuaRE program, the Banff International Research Station, SageMath, and the MATRIX Institute for facilitating their work together. Doran acknowledges support from NSERC and the hospitality of the ICERM at Brown University and the CMSA at Harvard University. Kelly acknowledges that this material is based upon work supported by the NSF under Award No.\ DMS-1401446 and the EPSRC under EP/N004922/1.  Voight was supported by an NSF CAREER Award (DMS-1151047) and a Simons Collaboration Grant (550029). 

\section{Picard--Fuchs equations}\label{S:PFEqns}

In this section, we compute the Picard--Fuchs equations associated to all primitive cohomology for our five symmetric pencils of K3 surfaces defined in \eqref{table:5families}. Since we are working with pencils in projective space, we are able to represent 21 of the $h^2(X_\psi) = 22$ dimensions of the second degree cohomology as elements in the Jacobian ring, that is, the primitive cohomology of degree two for the quartic pencils in $\PP^3$. We employ a more efficient version of the Griffiths--Dwork technique which exploits discrete symmetries. This method was previously used by Candelas--de la Ossa--Rodr\'{i}guez-Villegas \cite{CORV, CORV2} and Doran-Greene--Judes \cite{DGJ08}. G\"{a}hrs \cite{Gahrs} used a similar combinatorial technique to study Picard--Fuchs equations for holomorphic forms on invertible pencils. After explaining the Griffiths--Dwork technique for symmetric pencils in projective space, we carry out the computation for two examples in thorough detail, and then state the results of the computation for three others.

\subsection{Setup}

We briefly review the computational technique of Griff\-iths--Dwork \cite{CORV, CORV2,DGJ08}, and we begin with the setup in some generality.

Let $X \subset \PP^n$ be a smooth projective hypersurface over $\C$ defined by the vanishing of $F(x_0, \ldots, x_n) \in \C[x_0,\dots,x_n]$ homogeneous of degree $d$.  Let $A^i(X)$ be the space of rational $i$-forms on $\PP^n$ with polar locus contained in $X$, or equivalently regular $i$-forms on $\PP^n \setminus X$.  
By Griffiths \cite[Corollary 2.1]{Gri}, any $\varphi \in A^n(X)$ can be written as
\begin{equation}
\varphi = \frac{Q(x_0, \ldots, x_n)}{F(x_0,\ldots, x_n)^k}\Omega_0,
\end{equation} 
where $k \geq 0$ and $Q \in \C[x_0,\dots,x_n]$ is homogeneous of degree $\deg Q = k\deg F - (n+1)$ and
\begin{equation}
\Omega_0 \colonequals  \sum_{i=0}^n (-1)^i x_i \,\dD{x_0}\wedge \ldots \wedge \dD{x_{i-1}} \wedge \dD{x_{i+1}} \wedge \ldots \wedge \dD{x_n}.
\end{equation}
We define the de Rham cohomology groups
\begin{equation} 
\mathcal{H}^i(X) \colonequals \frac{A^i(X)}{\dD{A^{i-1}}(X)}.
\end{equation}

There is a \defi{residue map} 
\[ \Res \colon \mathcal{H}^n(X) \rightarrow H^{n-1}(X,\C) \] 
made famous by seminal work of Griffiths \cite{Gri}, mapping into the middle-dimensional Betti cohomology of the hypersurface $X$. Given $\varphi \in A^n(X)$, we choose an $(n-1)$-cycle $\gamma$ in $X$ and $T(\gamma)$ a circle bundle over $\gamma$ with an embedding into the complement $\PP^n \setminus X$ that encloses $\gamma$, and define $\Res(\varphi)$ to be the $(n-1)$-cocycle such that
\begin{equation}\label{residueRelation}
\frac{1}{2\pi \ii } \int_{T(\gamma)} \varphi = \int_{\gamma} \Res(\varphi),
\end{equation}
well-defined for $\varphi \in \mathcal{H}^n(X)$.  Two circle bundles $T(\gamma)$ with small enough radius are homologous in $H_n(\PP^n \setminus X, \ZZ)$, so the class $\Res(\varphi)  \in H^{n-1}(X,\C)$ is well-defined. 

There is a filtration on $\mathcal{H}^n(X)$ by an upper bound on the order of the pole along $X$: 
$$
\mathcal{H}^n_1(X) \subseteq \mathcal{H}^n_2(X) \subseteq \ldots \subseteq \mathcal{H}^n_{n}(X) = \mathcal{H}^n(X).
$$
This filtration on $\mathcal{H}^n(X)$ is compatible with the Hodge filtration on $H^{n-1}(X,\C)$: if we define 
\[ \Fil^k(X) \colonequals H^{n-1,0}(X,\C) \oplus \ldots \oplus H^{k, n-k-1}(X, \C), \] 
then the residue map restricts to $\Res\colon \mathcal{H}^n_k(X) \rightarrow \Fil^{n-k}(X)$.  

In certain circumstances, we may be able to reduce the order of the pole \cite[Formula 4.5]{Gri}: we have
\begin{equation}\label{reducePoleOrder}
\frac{\Omega_0}{F(x_i)^{k+1}} \sum_{j=0}^n Q_j(x_i) \frac{ \partial F(x_i)}{\partial x_j} = \frac{1}{k} \frac{\Omega_0}{F(x_i)^k} \sum_{j=0}^n \frac{\partial Q_j(x_i)}{\partial x_j} + \omega
\end{equation}
where $\omega$ is an exact rational form.  In fact, equation \eqref{reducePoleOrder} implies that the order of a form $\varphi$ can be lowered (up to an exact form) if and only if the polynomial $Q$ is in the Jacobian ideal $J(F)$, that is, the (homogeneous) ideal generated by all partial derivatives of $F$.  So for $k \geq 1$ we have a natural identification
\begin{equation}
\frac{\mathcal{H}^n_k(X)}{\mathcal{H}^n_{k-1}(X)} \xrightarrow{\sim} \left(\frac{\C[x_0, \ldots, x_n]}{J(F)}\right)_{k\deg F -(n+1)}
\end{equation}
which by the residue map induces an identification
\begin{equation} \label{eqn:Cxo}
\left(\frac{\C[x_0, \ldots, x_n]}{J(F)}\right)_{k\deg F -(n+1)} \rightarrow H^{n-k, k-1}(X),
\end{equation}
whose image is the primitive cohomology group $H^{n-k,k-1}_{\textup{prim}}(X)$, which we know is the cohomology orthogonal to the hyperplane class since $X$ is a hypersurface in $\PP^n$. 

\begin{exm}  For $X$ a quartic hypersurface in $\PP^3$, the identification \eqref{eqn:Cxo} reads
\begin{equation} \label{eqn:jp4}
\frac{\C[x_0, x_1,x_2,x_3]_{4k}}{J(F)} \simeq H^{2-k,k}_{\textup{prim}}(X).
\end{equation}
In this case, the Hodge numbers are given by $h^{2,0} = 1$, $h^{1,1} = 35 - 4\cdot 4 = 19$, and $h^{0,2} = 165 - 4\cdot 56 + 6\cdot 10 = 1$.
\end{exm}

\subsection{Griffiths--Dwork technique}

Now suppose that $X_{\psi}$ is a pencil of hypersurfaces in the parameter $\psi$, defined by $F_\psi=0$.  
Let $\{\gamma_j\}_j$ be a basis for $H_{n-1}(X_{\psi},\C)$ with cardinality $h_{n-1} \colonequals\dim_{\C} H_{n-1}(X_{\psi},\C)$. 

\begin{rmk}
There is a subtle detail about taking a parallel transport using an Ehresmann connection to obtain a (locally) unique horizontal family of homology classes \cite[\S2.3]{DGJ08}.  This detail does not affect our computations. 
\end{rmk}

We then choose a basis of (possibly $\psi$-dependent) $(n-1)$-forms $\Omega_{X_\psi, i} \in H^{n-1}(X,\C)$ so that each of the forms $\Omega_{X_\psi, i} \in H^{n-1}(X,\C) $ has fixed bidegree $(p,q)$ which provides a basis for the Hodge decomposition $H^{n-1}(X,\C) = \bigoplus_{p+q = n-1} H^{p,q}(X)$ for each fixed $\psi$. 
We now examine the period integrals 
$$
\int_{\gamma_j} \Omega_{X_\psi, i}
$$
for $1 \leq i,j \leq  h_{n-1}$. 

We want to understand how these integrals vary with respect to the pencil parameter $\psi$. To do so, we simply differentiate with respect to $\psi$, or equivalently integrate on the complement of $X_\psi$ in $\PP^n$ as outlined above.  Using the residue relation \eqref{residueRelation}, we rewrite: 
\begin{equation}
\int_{\gamma_{j}} \Omega_{X_\psi, i} = \int_{T(\gamma_{j})} \frac{Q_i}{F_\psi^{k}} \Omega_0,
\end{equation}
for some $Q_i \in \C[x_0, \ldots, x_n]_{k\deg F_{\psi} -(n+1)}$ and $k \in \Z_{\geq 0}$ (and circle bundle $T(\gamma_{j})$ with sufficiently small radius as above). By viewing $F_\psi$ as a function $F\colon \C\rightarrow \C[x_0,\ldots, x_n]$ with parameter $\psi$, we can differentiate $F(\psi)$ with respect to $\psi$ and study how this period integral varies:
\begin{equation}\label{ddpsi}
\frac{\dD}{\dD\psi} \int_{T(\gamma_j)} \frac{Q_i}{F(\psi)^k}\Omega_0 = -k \int_{T(\gamma_j)} \frac{Q_i}{F(\psi)^{k+1}} \frac{\dD F}{\dD \psi}\Omega_0 
\end{equation}
Note that the right-hand side of \eqref{ddpsi} gives us a new $(n-1)$-form. 

We know that we will find a linear relation if we differentiate $\dim_\C H^{n-1}(X_\psi,\C)$ times, giving us a single-variable ordinary differential equation called the \defi{Picard--Fuchs equation} for the period $\int_{\gamma_{j}} \Omega_{X_\psi, i}$.  In practice, fewer derivatives may be necessary.

For simplicity, we suppose that $F_\psi$ is linear in the variable $\psi$.  Then the \defi{Griffiths--Dwork technique} for finding the Picard--Fuchs equation is the following procedure (see \cite{CoxKatz} or \cite{DGJ08} for a more detailed exposition):
\begin{enumerate}[1.]
\item Differentiate the period $b$ times, $1 \leq b \leq h_{n-1}$. We obtain the equation
$$
\left(\frac{\dD}{\dD\psi}\right)^b \int_{T(\gamma_j)} \frac{Q_i}{F(\psi)^k}\Omega_0 = \frac{(k+b-1)!}{(k-1)!} \int_{T(\gamma_j)} \frac{Q_i}{F(\psi)^{k+b}}\left(-\displaystyle{\frac{\dD F}{\dD \psi}}\right)^b\Omega_0.
$$

\item Write 
\begin{equation} 
Q_i \left(-\displaystyle{\frac{\dD F}{\dD \psi}}\right)^b = \sum_{j=1}^{h_{n-1}} \alpha_j Q_j + J
\end{equation}
where $\alpha_k \in \C(\psi)$ and $J$ is in the Jacobian ideal, so we may write $J = \sum_i A_i \displaystyle{\frac{\partial F_\psi}{\partial x_i}}$ with $A_i \in \C(\psi)[x_0, \ldots, x_n]$ for all $i$.  

\item Use \eqref{reducePoleOrder} to reduce the order of the pole of $\displaystyle{\frac{J}{F(\psi)^k}\Omega_0}$.  We obtain a new numerator polynomial of lower degree.
\item Repeat steps 2 and 3 for the new numerator polynomials, until the $b$th derivative is expressed in terms of the chosen basis for cohomology.
\item Use linear algebra to find a $\C(\psi)$-linear relationship between the derivatives.
\end{enumerate}

While algorithmic and assured to work, this method can be quite tedious to perform.  Moreover,  the structure of the resulting differential equation may not be readily apparent.

\subsection{A diagrammatic Griffiths--Dwork method}

In this section, we give a computational technique that uses discrete symmetries of pencils of Calabi--Yau hypersurfaces introduced by Candelas--de la Ossa--Rodr\'{i}guez-Villegas \cite{CORV, CORV2}.  To focus on the case at hand, we specialize to the case of quartic surfaces and explain this method so their diagrammatic and effective adaptation of the Griffiths--Dwork technique can be performed for the five pencils that we want to study. 

Let $x^{\mathbf{v}} \colonequals x_0^{v_0}x_1^{v_1}x_2^{v_2}x_3^{v_3}$ and let $k(\mathbf{v}) \colonequals \frac{1}{4} \sum_i v_i$; for a monomial arising from \eqref{eqn:jp4}, we have $k(\mathbf{v}) \in \Z_{\geq 0}$. Fix a cycle $\gamma$, and consider the periods 
\begin{equation}
(v_0,v_1,v_2,v_3)  \colonequals  \int_{T(\gamma)}  \frac{x^{\mathbf{v}}}{F_\psi^{k(\mathbf{v})+1}}\Omega_0 .
\end{equation}

Consider the relation:
\begin{equation} \label{eqn:partiali}
\partial_i \left( \frac{x_i x^{\mathbf{v}}}{F_\psi^{k(\mathbf{v})+1}}\right) =  \frac{x^{\mathbf{v}}}{F_\psi^{k(\mathbf{v})+1}}(1+v_i) - (k(\mathbf{v})+1) \frac{ x^{\mathbf{v}}}{F_\psi^{k(\mathbf{v})+2}}x_i \partial_i F_\psi.
\end{equation}
We can use \eqref{eqn:partiali} in order to simplify the computation of the Picard--Fuchs equation: integrating over $T(\gamma)$, the left hand side vanishes, so we can solve for $(v_0,v_1,v_2,v_3)$: 
\begin{equation}\label{DiagramPartials}
(1+v_i) (v_0,v_1,v_2,v_3)  \colonequals  \int_{T(\gamma)} \frac{x^{\mathbf{v}}}{F_\psi^{k(\mathbf{v})+1}} \Omega_0 = (k(\mathbf{v})+1)\int_{T(\gamma)} \frac{x^{\mathbf{v}}x_i\partial_iF_\psi}{F_\psi^{k(\mathbf{v})+2}}\Omega_0.
\end{equation}

\begin{exm}\label{dworkexamp}
Consider the Dwork pencil $\Fsf_4$, the pencil defined by the vanishing of 
$$
F_{\psi} = x_0^4 + x_1^4 + x_2^4 + x_3^4 - 4\psi x_0x_1x_2x_3.
$$
Simplifying the right-hand side of \eqref{DiagramPartials} gives us the relation of periods:
\[ (1+v_i) (v_0,v_1,v_2,v_3)= 4(k(\mathbf{v})+1)\left((v_0,\ldots, v_i + 4, \ldots, v_3) - \psi (v_0+1, v_1+1, v_2+1,v_3+1)\right) \]
or in a more useful form
\begin{equation}\label{stair}
(v_0,\ldots, v_i + 4, \ldots, v_3) = \frac{1+v_i}{4(k(\mathbf{v})+1)} (v_0,v_1,v_2,v_3) +  \psi (v_0+1, v_1+1, v_2+1,v_3+1)
\end{equation}
for $i=0,1,2,3$.
\end{exm}

Recall we can also find a relation between various $(v_0,v_1,v_2,v_3)$ by differentiating with respect to $\psi$. Rewriting \eqref{ddpsi} in the current notation, we obtain
\begin{equation}\label{psiderivative}
\frac{\dD}{\dD\psi} (v_0,v_1,v_2,v_3) = (4(k(\mathbf{v}) +1)) (v_0+1,v_1+1,v_2+1,v_3+1),
\end{equation}
yielding a dependence of the monomials with respect to the successive derivatives with respect to $\psi$.

Using the relations \eqref{DiagramPartials} and \eqref{psiderivative}, we will compute the Picard--Fuchs equations associated to periods that come from primitive cohomology.  The key observation is that these two operations respect the symplectic symmetry group.  

Restricting now to our situation, let $\diamond \in \{\Fsf_4,\Fsf_2\Lsf_2,\Fsf_1\Lsf_3,\Lsf_2\Lsf_2,\Lsf_4\}$ signify one of the five \textup{K3}\/ families in \textup{\eqref{table:5families}} defined by $F_{\diamond,\psi}$ and having symmetry group $H=H_{\diamond}$ as in \eqref{table:5families}.  Then $H$ acts on the $19$-dimensional $\C$-vector space
\begin{equation} \label{eqn:defofVJF}
V \colonequals (\C[x_0, x_1,x_2, x_3] / J(F_\psi))_4 
\end{equation}
giving a representation $H \to \GL(V)$. As $H$ is abelian, we may decompose $V = \bigoplus_\chi W_\chi$ where $H$ acts on $W_\chi$ by a (one-dimensional) character $\chi\colon H \to \C^\times$.  
Conveniently, each subspace $W_{\chi}$ has a monomial basis.  Moreover, the relations from the Jacobian ideal \eqref{DiagramPartials} and \eqref{psiderivative} respect the action of $H$, so we can apply the Griffiths--Dwork technique to the smaller subspaces $W_\chi$. 

\subsection{Hypergeometric differential equations} \label{sec:hypdiffeq}

In fact, we will find that all of our Picard--Fuchs differential equations are hypergeometric. In this section, we briefly recall the definitions \cite{slater}.

\begin{defn}\label{defn:hypergeometricFunction}
Let $n,m\in\Z$, let $\alpha_1,\dots,\alpha_n \in \Q$ and $\beta_1,\dots,\beta_m\in\Q_{>0}$, and write $\pmb{\alpha}=\{\alpha_j\}_j$ and $\pmb{\beta}=\{\beta_j\}_j$ as multisets.  The \defi{(generalized) hypergeometric function}  is the formal series 
\begin{equation}
F(\pmb{\alpha};\pmb{\beta} \,|\, z) \colonequals \sum_{k=0}^{\infty}\frac{(\alpha_1)_k\cdots(\alpha_n)_k}{(\beta_1)_k\cdots(\beta_m)_k}z^k \in \Q[[z]],\label{E:hyperg}
\end{equation}
\noindent where $(x)_k$ is the rising factorial (or \defi{Pochhammer symbol})
\[(x)_k \colonequals x(x+1)\cdots(x+k-1)=\frac{\Gamma(x+k)}{\Gamma(x)} \]
and $(x)_0 \colonequals 1$.  We call $\pmb{\alpha}$ the \defi{numerator parameters} and $\pmb{\beta}$ the \defi{denominator parameters}.
\end{defn}

We consider the differential operator 
\[ \theta \colonequals z \displaystyle{\frac{\dD}{\dD z}} \] 
and define the \defi{hypergeometric differential operator} 
\begin{equation} \label{eqn:Dab'lldoit}
D(\pmb{\alpha};\pmb{\beta} \,|\,z) \colonequals (\theta +\beta_1 -1)\cdots(\theta + \beta_m -1) - z (\theta+\alpha_1)\cdots(\theta + \alpha_n). 
\end{equation}
When $\beta_1=1$, the hypergeometric function $F(\pmb{\alpha};\pmb{\beta} \,|\, z)$ is annihilated by $D(\pmb{\alpha};\pmb{\beta}\,|\,z)$.

\subsection{The Dwork pencil \texorpdfstring{$\Fsf_4$}{Fsf4}}\label{PF:Dwork}

We now proceed to calculate Picard--Fuchs equations for our five pencils.  
We begin in this section with the Dwork pencil $\Fsf_4$, the one-parameter family of projective hypersurfaces $X_\psi \subset \PP^4$
defined by the vanishing of the polynomial 
$$
F_{ \psi}  \colonequals   x_0^4+x_1^4 + x_2^4 + x_3^4 - 4\psi x_0x_1x_2x_3.
$$
The differential equations associated to this pencil were studied by Dwork \cite[\S 6j]{padic}; our approach is a bit more detailed and explicit, and this case is a good warmup as the simplest of the five cases we will consider.

There is a $H=(\Z/4\Z)^2$ symmetry of this family generated by the automorphisms
\begin{equation}
\begin{aligned}
g_1(x_0:x_1:x_2:x_3) &= (-\ii x_0:\ii x_1:x_2:x_3) \\
g_2(x_0:x_1:x_2:x_3) &= (-\ii x_0:x_1:\ii x_2:x_3).
\end{aligned}
\end{equation}
A character $\chi\colon H \rightarrow \C^\times$ is determined by $\chi(g_1),\chi(g_2) \in \langle \sqrt{-1} \rangle$, and we write $\chi_{(a_1,a_2)}$ for the character with $\chi_{(a_1,a_2)}(g_i) = \ii^{a_i}$ with $a_i \in \Z/4\Z$ for $i=1,2$, totalling $16$ characters.  We then decompose $V$ defined in \eqref{eqn:defofVJF} into irreducible subspaces with a monomial basis.  We cluster these subspaces into three types up to the permutation action by $S_4$ on coordinates:
\begin{enumerate}[(i)]
\item $(a_1,a_2)=(0,0)$ (the $H$-invariant subspace), spanned by $x_0x_1x_2x_3$;
\item $(a_1,a_2)$ both even but not both zero, e.g., the subspace with $(a_1,a_2)=(0,2)$ spanned by $x_0^2x_1^2, x_2^2 x_3^2$; and
\item $(a_1,a_2)$ not both even, e.g., the subspace with $(a_1,a_2)=(2,1)$, spanned by $x_0^3x_1$.
\end{enumerate}
Up to permutation of coordinates, there are $1,3,12$ subspaces of types (i),(ii),(iii), respectively. 
By symmetry, we just need to compute the Picard--Fuchs equations associated to one subspace of each of these types. In other words, we only need to find equations satisfied by the monomials $x_0x_1x_2x_3, x_0^2x_1^2, x_2^2 x_3^2, $ and $x_0^3x_1$, 
corresponding to $(1,1,1,1)$, $(2,2,0,0)$, $(0,0,2,2)$, and $(3,1,0,0)$, respectively. 

The main result for this subsection is as follows.

\begin{prop} \label{prop:F4H2result}
The primitive middle-dimensional cohomology group $H^2_{\textup{prim}}(X_{\Fsf_4, \psi},\C)$ has $21$ periods whose Picard--Fuchs equations are hypergeometric differential equations as follows:
\[ 
\begin{gathered}
\text{$3$ periods are annihilated by $D(\tfrac{1}{4}, \tfrac{1}{2}, \tfrac{3}{4} ; 1, 1, 1 \,|\,\psi^{-4})$,} \\
\text{$6$ periods are annihilated by $D(\tfrac{1}{4}, \tfrac{3}{4}; 1, \tfrac{1}{2} \,|\,\psi^{-4})$, and} \\
\text{$12$ periods are annihilated by $D(\tfrac{1}{2};1 \,|\,\psi^{-4})$.}
\end{gathered} \]
\end{prop}

By the interlacing criterion \cite[Theorem 4.8]{BH}, the latter two hypergeometric equations have algebraic solutions.

We state and prove each case of Proposition~\ref{prop:F4H2result} with an individual lemma.

\begin{lem} \label{lem:F4hol}
The Picard--Fuchs equation associated to the period $\psi (0,0,0,0)$ is the hypergeometric differential equation $D(\tfrac{1}{4}, \tfrac{1}{2}, \tfrac{3}{4} ; 1, 1, 1 \,|\,\psi^{-4})$.
\end{lem}
\begin{proof}
We recall the equations \eqref{stair} and \eqref{psiderivative}:
\begin{equation}
(v_0,\ldots, v_i + 4, \ldots, v_3) = \frac{1+v_i}{4(k(\mathbf{v})+1)} (v_0,v_1,v_2,v_3) +  \psi (v_0+1, v_1+1, v_2+1,v_3+1);
\end{equation}
\begin{equation}
\frac{\dD}{\dD\psi} (v_0,v_1,v_2,v_3) = (4(k(\mathbf{v}) +1)) (v_0+1,v_1+1,v_2+1,v_3+1).
\end{equation}
These equations imply a dependence among the terms 
\[ (v_0,v_1,v_2,v_3), (v_0+1,v_1+1,v_2+1,v_3+1), \text{ and }(v_0,\ldots, v_i + 4, \ldots, v_3) \] 
denoted in the following diagram:
$$\xymatrix{
(v_0,v_1,v_2,v_3) \ar[r] \ar[d] & (v_0+1,v_1+1,v_2+1,v_3+1) \\
(v_0,\ldots, v_i + 4, \ldots, v_3) 
}$$

In order to use these dependences, we build up a larger diagram:
\begin{equation}
\begin{gathered}
\xymatrix{
	& (0,0,0,0) \ar[r] \ar[d] & (1,1,1,1)  \ar[r] \ar[d] & (2,2,2,2) \ar[r] \ar[d] & (3,3,3,3) \\
	& (4,0,0,0) \ar[r] \ar[d] & (5,1,1,1)  \ar[r] \ar[d] & (6,2,2,2) & \\
	& (4,4,0,0) \ar[r] \ar[d] & (5,5,1,1) &  & \\
	(3,3,3,-1) \ar[r] \ar[d] & (4,4,4,0) &&& \\
	(3,3,3,3) &&&&&}
\end{gathered}
\end{equation}
It may be useful to point out that the same period must appear in two places by simple linear algebra: the vectors $(4,0,0,0), (0,4,0,0), (0,0,4,0), (0,0,0,4)$ and $(1,1,1,1)$ are linearly dependent.

Using \eqref{stair} and \eqref{psiderivative} and letting $\eta \colonequals \psi \displaystyle{\frac{\dD}{\dD\psi}}$, we see that:
\begin{equation}\begin{aligned}
(0,0,0,0) &= \frac{1}{4} (\eta + 1) (4,0,0,0) \\
(4,0,0,0) &= \frac{1}{8} (\eta + 1) (4,4,0,0) \\
(4,4,0,0) &= \frac{1}{12} (\eta + 1) (4,4,4,0) \\
\psi (4,4,4,0) &= (3,3,3,3)
\end{aligned}\end{equation}
Now, we can use the fact that $(\eta - a)\psi^a = \psi^a \eta$ for $a \in \Z$ to great effect:
\begin{equation}\begin{aligned} 
\eta (0,0,0,0) &= 4 \psi (1,1,1,1) \\
(\eta - 1)\eta (0,0,0,0) &= 4 \psi \eta(1,1,1,1) = 8\cdot 4 \psi^2 (2,2,2,2) \\
(\eta -2)(\eta -1)\eta (0,0,0,0) & =12 \cdot 8\cdot 4 \psi^2 \eta(2,2,2,2) = 12\cdot 8\cdot 4 \psi^3 (3,3,3,3) \\
	&= 12\cdot 8\cdot 4 \psi^4 (4,4,4,0) \\
	&= 8 \cdot 4 \psi^4 (\eta+1)(4,4,0,0) \\
	&= 4\psi^4(\eta+1)^2 (4,0,0,0) \\
	&= \psi^4(\eta+1)^3 (0,0,0,0).
\end{aligned}\end{equation}
We conclude that
\begin{equation}
\left[  (\eta -2)(\eta -1)\eta - \psi^4 (\eta+1)^3\right] (0,0,0,0) = 0.
\end{equation}
We then multiply by $\psi$ to obtain
\begin{align*}
\left[  \psi(\eta -2)(\eta -1)\eta - \psi^4 \psi (\eta+1)^3\right] (0,0,0,0) &= 0\\
\left[  (\eta -3)(\eta -2)(\eta-1) - \psi^4 (\eta)^3\right] \psi (0,0,0,0) &= 0.
\end{align*}
Finally, substitute $t  \colonequals \psi^{-4}$ and let $\theta  \colonequals  t \displaystyle{\frac{\dD}{\dD t}} = -\eta/4$ to see that
\begin{equation}\begin{aligned}\label{hyperF4hol}
\left[  (-4\theta -3)(-4\theta -2)(-4\theta-1) - t^{-1} (-4\theta)^3\right] \psi (0,0,0,0) &= 0\\
\left[  -t(\theta +\tfrac{3}{4})(\theta +\tfrac{1}{2})(\theta+\tfrac{1}{4}) + \theta^3\right] \psi (0,0,0,0) &= 0\\
\left[  \theta^3 -t(\theta +\tfrac{1}{4})(\theta +\tfrac{1}{2})(\theta+\tfrac{3}{4}) \right] \psi (0,0,0,0) &= 0,
\end{aligned}\end{equation}
which is the differential equation $D(\tfrac{1}{4}, \tfrac{1}{2}, \tfrac{3}{4} ; 1, 1, 1 \,|\,t)$.
\end{proof}

\begin{lem} \label{lem:F4twoSquares}
The Picard--Fuchs equation associated to both $\psi (2,2,0,0)$ and $\psi(0,0,2,2)$ is $D(\tfrac{1}{4}, \tfrac{3}{4}; 1, \tfrac{1}{2} \,|\, \psi^{-4})$.
\end{lem}

\begin{proof}
By iterating the use of \eqref{stair}, we can construct a diagram including both $ (2,2,0,0)$ and $(0,0,2,2)$:
\begin{equation} \label{eqn:0022}
\begin{gathered}
\xymatrix{
&&& (0,0,2,2) \ar[d] \ar[r] & (1,1,3,3) \\
&&(3,-1,1,1) \ar[d] \ar[r] & (4,0,2,2) & \\
&(2,2,0,0) \ar[d] \ar[r] &(3,3,1,1) && \\
(1,1,3,-1) \ar[d] \ar[r] & (2,2,4,0) &&& \\
(1,1,3,3) &&&&}
\end{gathered}
\end{equation}
We then obtain the following relations:
\begin{equation}\begin{aligned}
\eta (2,2,0,0) &= \psi^2(\eta + 1) (0,0,2,2); \\
\eta (0,0,2,2) &= \psi^2(\eta + 1) (2,2,0,0).
\end{aligned}\end{equation}
We then can use these relations to make a Picard--Fuchs equation associated to the period $ (2,2,0,0)$:
\begin{equation}\begin{aligned}
(\eta - 2) \eta (2,2,0,0) & = \psi^2  (\eta + 1)\eta (0,0,2,2) \\
	& = \psi^2(\eta+1) \left(\psi^2(\eta+1) (2,2,0,0)\right) \\
	& = 2\psi^4(\eta +1) (2,2,0,0) + \psi^4(\eta+1) \eta (2,2,0,0) \\
	&\qquad + \psi^4(\eta+1) (2,2,0,0) \\
	& = \psi^4(\eta^2+4\eta +3) (2,2,0,0) \\
	& = \psi^4(\eta+1)(\eta+3) (2,2,0,0).
\end{aligned}\end{equation}

By symmetry, we get the same equation for the period $(0,0,2,2)$, so we have:
\begin{equation}\begin{aligned}
\left[ (\eta-2)\eta - \psi^4(\eta+1)(\eta+3)\right] (2,2,0,0) &= 0 \\
\left[ (\eta-2)\eta - \psi^4(\eta+1)(\eta+3) \right] (0,0,2,2) &= 0
\end{aligned}\end{equation}
Now multiply by $\psi$ and then change variables to $t \colonequals \psi^{-4}$ with $\theta \colonequals t \tfrac{\dD}{\dD t} = -{4} \eta$ to obtain:
\begin{align*}
\left[\psi (\eta-2)\eta - \psi\psi^4(\eta+1)(\eta+3) \right] (2,2,0,0) &= 0\\
\left[ (\eta-3)(\eta-1) - \psi^4\eta(\eta+2) \right] \psi(2,2,0,0)&= 0\\
\left[ (-4\theta-3)(-4\theta-1) - t^{-1}(-4\theta)(-4\theta+2)\right] \psi(2,2,0,0) &= 0 \\
\left[ t(\theta + \tfrac{3}{4})(\theta + \tfrac{1}{4}) - \theta(\theta-\tfrac{1}{2}) \right] \psi(2,2,0,0) &= 0 \\
\left[\theta(\theta-\tfrac{1}{2}) -t(\theta + \tfrac{1}{4})(\theta +\tfrac{3}{4}) \right]\psi (2,2,0,0) &= 0.
\end{align*}
 This Picard--Fuchs equation is $D(\tfrac{1}{4}, \tfrac{3}{4}; 1, \tfrac{1}{2} \,|\, \psi^{-4})$.
\end{proof} 

\begin{lem} \label{lem:F4mono31}
The Picard--Fuchs equation associated to $\psi (3,1,0,0)$ is $D(\frac{1}{2};1 \,|\,\psi^{-4})$.
\end{lem}
\begin{proof}
Our strategy again is to use \eqref{stair} and \eqref{psiderivative} in the order represented by the diagram below to study the period $(3,1,0,0)$:
\begin{equation}
\begin{gathered}
\xymatrix{
&&&(2,0,-1,3) \ar[d] \ar[r] & (3,1,0,4) \\
&& (1,-1,2,2) \ar[d] \ar[r] & (2,0,3,3) &\\
(-1,1,0,0)\ar[d] \ar[r]  & (0,2,1,1) \ar[d] \ar[r] & (1,3,2,2) &&\\
(3,1,0,0) \ar[d] \ar[r] & (4,2,1,1) &&& \\
(3,1,0,4)
}
\end{gathered}
\end{equation}
Using \eqref{stair} iteratively in the upper part of the diagram, we see that:
\begin{equation}
(1,3,2,2) = \psi^2(3,1,0,4).
\end{equation}
Then using \eqref{psiderivative}, we then have that 
\begin{equation}
\eta (0,2,1,1) = 8\psi(1,3,2,2) = 8\psi^3 (3,1,0,4).
\end{equation}
Now, using \eqref{stair} again, we have that $(3,1,0,0) = \psi (0,2,1,1)$ and we can then compute:
\begin{equation}\begin{aligned}
(\eta -1) (3,1,0,0) &= \psi \eta(0,2,1,1) \\
	&=8\psi^4(3,1,0,4) \\
	&=8 \psi^4\left[\frac{1}{8} (3,1,0,0) + \psi (4,2,1,1)\right] \\
	&= 8 \psi^4 \left[\frac{1}{8} (3,1,0,0) + \frac{1}{8} \eta (3,1,0,0)\right] \\
	&= \psi^4 (\eta + 1) (3,1,0,0).
\end{aligned}\end{equation}
We then get the Picard--Fuchs equation associated to the period $(3,1,0,0)$:
\begin{equation}
\left[ (\eta - 1) - \psi^4 (\eta +1) \right] (3,1,0,0) = 0.
\end{equation}

We now will multiply by $\psi$ and then change variables to $t = \psi^{-4}$ as in the previous lemma 
to obtain:
\begin{align*}
\left[\psi (\eta - 1) - \psi\psi^4 (\eta +1) \right] (3,1,0,0) &= 0\\
\left[ (\eta - 2) - \psi^4 \eta \right] \psi(3,1,0,0) &= 0\\
\left[ (-4\theta - 2) - t^{-1} (-4\theta) \right]\psi (3,1,0,0) &= 0 \\
\left[ \theta -t(\theta + \tfrac{1}{2}) \right]\psi (3,1,0,0) &= 0,
\end{align*}
giving rise to the hypergeometric differential equation $D(\frac{1}{2};1 \,|\,\psi^{-4})$.
\end{proof}

We now conclude this section with the proof of the main result.

\begin{proof}[Proof of Proposition \textup{\ref{prop:F4H2result}}]
We combine Lemmas~\ref{lem:F4hol},~\ref{lem:F4twoSquares}, and~\ref{lem:F4mono31} with the consideration of the number of subspaces of each type described above.
\end{proof}

\subsection{The Klein--Mukai pencil \texorpdfstring{$\Fsf_1\Lsf_3$}{Fsf1Lsf3}} \label{PF:KM}

We now consider the Klein--Mukai pencil $\Fsf_1\Lsf_3$, the one-parameter family of hypersurfaces $X_\psi \subset \PP^4$ defined by the vanishing of
$$
F_\psi \colonequals x_0^3x_1+ x_1^3x_2 + x_2^3x_0 + x_3^4 - 4\psi x_0x_1x_2x_3.
$$
The polynomial $F_\psi$ is related to the defining polynomial~\eqref{table:5families} by a change in the order of variables.

There is a $H=\Z / 7\Z$ scaling symmetry of this family generated by the automorphism $(x_i)$ by the element
$$
g(x_0: x_1: x_2: x_3) = (\xi x_0 : \xi^4 x_1 : \xi^2x_2 : x_3),
$$
where $\xi$ is a seventh root of unity.  There are seven characters $\chi_k\colon H \to \C^\times$ defined by $\chi_k(g) = \xi^k$ for $k \in \Z/7\Z$.  
Note that the monomial bases for the subspaces $W_{\chi_1}, W_{\chi_2},$ and $W_{\chi_4}$ are cyclic permutations of one another under the variables $x_0, x_1,$ and $x_2$. Analogously, so are subspaces $W_{\chi_3}, W_{\chi_5},$ and $W_{\chi_6}$. So we have three types of clusters:
\begin{enumerate}[(i)]
\item   $W_{\chi_0}$ has the monomial basis $\{x_0x_1x_2x_3\}$;
\item  $W_{\chi_1}$ has the monomial basis $\{x_1^2x_3^2,  x_0^2x_1x_2, x_2^2x_1x_3\}$; and
\item  $W_{\chi_3}$ has the monomial basis $\{x_2^3x_1,  x_1^2x_2x_3, x_3^2x_0x_2\}$.
\end{enumerate}

There is one cluster of type (i) and three clusters each of types (ii) and (iii), so $h^{1,1}$ is decomposed as $19 = 1 + 3\cdot 3 + 3 \cdot 3$.

\begin{prop} \label{prop:kleinmukaidiff}
The group $H^2_{\textup{prim}}(X_{\Fsf_1\Lsf_3, \psi})$ has $21$ periods whose Picard--Fuchs equations are hypergeometric differential equations, with $3$ periods annihilated by
\[ D(\tfrac{1}{4}, \tfrac{1}{2}, \tfrac{3}{4} ; 1, 1, 1 \,|\,\psi^{-4}) \]
and $3$ periods each annihilated by the following $6$ operators:
\[ 
\begin{gathered}
D(\tfrac{1}{14}, \tfrac{9}{14}, \tfrac{11}{14} ; \tfrac{1}{4}, \tfrac{3}{4}, 1 \,|\,\psi^{4}), 
D(\tfrac{-3}{14}, \tfrac{1}{14}, \tfrac{9}{14} ; 0, \tfrac{1}{4}, \tfrac{3}{4} \,|\,\psi^{4}), 
D(\tfrac{-5}{14}, \tfrac{-3}{14}, \tfrac{1}{14} ; \tfrac{-1}{4}, 0, \tfrac{1}{4} \,|\,\psi^{4}), \\ 
D(\tfrac{3}{14}, \tfrac{5}{14}, \tfrac{13}{14} ; \tfrac{1}{4}, \tfrac{3}{4}, 1 \,|\,\psi^4), 
D(\tfrac{-1}{14}, \tfrac{3}{14}, \tfrac{5}{14} ; 0, \tfrac{1}{4}, \tfrac{3}{4} \,|\,\psi^4), 
D(\tfrac{-11}{14}, \tfrac{-1}{14}, \tfrac{5}{14} ; \tfrac{-1}{4}, 0, \tfrac{1}{4} \,|\,\psi^4). 
\end{gathered}
\]
\end{prop}

Again, the latter $6$ operators have an algebraic solution.  To prove Proposition \ref{prop:kleinmukaidiff}, we again use the diagrammatic method outlined above, but in this case we have different periods that are related. Notice that we have the following differentials $\partial_i$ multiplied by $x_i$: 
\begin{equation}\begin{aligned}
x_0\partial_0 F_\psi  &= 3 x_0^3 x_1 + x_2^3x_0 - 4 \psi x_0x_1x_2x_3 \\
x_1\partial_1 F_\psi  &= 3 x_1^3 x_2 + x_0^3x_1 - 4 \psi x_0x_1x_2x_3 \\
x_2\partial_2 F_\psi  &= 3 x_2^3 x_0 + x_1^3x_2 - 4 \psi x_0x_1x_2x_3 \\
x_3 \partial_3 F_\psi &= 4x_4^4 - 4\psi x_0x_1x_2x_3
\end{aligned}\end{equation}
We can make linear combinations of these equations so that the right hand side is just a linear combination of two monomials, for example,
\begin{equation}
(9x_0\partial_0 + x_1\partial_1 - 3x_2\partial_2)F_\psi = 28(x_0^3x_1 - \psi x_0x_1x_2x_3).
\end{equation}
Now using \eqref{DiagramPartials}, we obtain the following period relations analogous to \eqref{stair}, written in multi-index notation:
\begin{equation}\begin{aligned}\label{F1L3stair}
\mathbf{v}+(3,1,0,0) &= \frac{f_0(\mathbf{v})}{28(k(\mathbf{v})+1)}\mathbf{v} + \psi (\mathbf{v}+(1,1,1,1)), \\
\mathbf{v}+(0,3,1,0) &= \frac{f_1(\mathbf{v})}{28(k(\mathbf{v})+1)}\mathbf{v} + \psi (\mathbf{v}+(1,1,1,1)), \\
\mathbf{v}+(0,0,3,1) &= \frac{f_2(\mathbf{v})}{28(k(\mathbf{v})+1)}\mathbf{v} + \psi (\mathbf{v}+(1,1,1,1)), \text{ and} \\
\mathbf{v}+(0,0,0,4) &= \frac{1+v_3}{4(k(\mathbf{v})+1)}\mathbf{v} + \psi (\mathbf{v}+(1,1,1,1));
\end{aligned}\end{equation}
where 
\begin{equation}\begin{aligned}
f_0(\mathbf{v}) &\colonequals 9(v_0+1) +(v_1+1) - 3(v_2+1), \\
f_1(\mathbf{v}) &\colonequals -3(v_0+1) +9(v_1+1) + (v_2+1), \text{ and} \\
f_2(\mathbf{v}) &\colonequals (v_0+1) -3 (v_1+1) + 9(v_2+1).
\end{aligned}\end{equation}

\begin{lem} \label{prop:F1L3hol}
Let $t = \psi^{-4}$. The Picard--Fuchs equation associated to the period $\psi (0,0,0,0)$ is the hypergeometric differential equation $D(\tfrac{1}{4}, \tfrac{1}{2}, \tfrac{3}{4} ; 1, 1, 1 \,|\,t)$.
\end{lem}
\begin{proof}
We  build the following diagram using \eqref{F1L3stair} and \eqref{psiderivative}:
$$
\xymatrix{
	& (0,0,0,0) \ar[r] \ar[d] & (1,1,1,1)  \ar[r] \ar[d] & (2,2,2,2) \ar[r] \ar[d] & (3,3,3,3) \\
	& (3,1,0,0) \ar[r] \ar[d] & (4,2,1,1)  \ar[r] \ar[d] & (5,3,2,2) & \\
	& (3,4,1,0) \ar[r] \ar[d] & (4,5,2,1) &  & \\
	(3,3,3,-1) \ar[r] \ar[d] & (4,4,4,0) &&& \\
	(3,3,3,3) &&&&&}
$$
When one runs through this computation, one can see that we get the same Picard--Fuchs equation for the invariant period as we did with the Fermat:
\begin{equation}
\left[(\eta-2)(\eta-1)\eta - \psi^4 (\eta+1)^3\right] (0,0,0,0) = 0
 \end{equation}
By multiplying by $\psi$ and changing variables to $t = \psi^{-4}$ and $\theta = t \displaystyle{\frac{\dD}{\dD t}}$, we can see by following through the computation seen in ~\eqref{hyperF4hol} that:
$$
\left[  \theta^3 -t(\theta +\tfrac{3}{4})(\theta +\tfrac{1}{2})(\theta+\tfrac{1}{4}) \right] \psi (0,0,0,0) = 0,
$$
which is the differential equation $D(\tfrac{1}{4}, \tfrac{1}{2}, \tfrac{3}{4} ; 1, 1, 1 \,|\,t)$.
\end{proof}
 
\begin{lem}\label{prop:F1L3char1}
For the Klein--Mukai family $X_\psi$, 
\[ \begin{gathered}
\text{the period $(0,1,2,1)$  is annihilated by $D(\tfrac{1}{14}, \tfrac{9}{14}, \tfrac{11}{14} ; \tfrac{1}{4}, \tfrac{3}{4}, 1 \,|\,\psi^{4})$,} \\ 
\text{the period $\psi(0,2,0,2)$  is annihilated by $D(\tfrac{-3}{14}, \tfrac{1}{14}, \tfrac{9}{14} ; 0, \tfrac{1}{4}, \tfrac{3}{4} \,|\,\psi^{4})$, and} \\ 
\text{the period $\psi^3 (2,1,1,0)$  is annihilated by  $D(\tfrac{-5}{14}, \tfrac{-3}{14}, \tfrac{1}{14} ; \tfrac{-1}{4}, 0, \tfrac{1}{4} \,|\,\psi^{4})$.} 
\end{gathered} \]
\end{lem}
\begin{proof}
For the character $\chi_1(g)=\xi$ associated to $1 \in \Z/7\Z$, we have the following diagram:
$$
\xymatrix{ &&& (0,1,2,1) \ar[r] \ar[d] & (1,2,3,2) \\
	&& (2,1,1,0) \ar[r] \ar[d] & (3,2,2,1) & \\
	& (1,3,1,-1) \ar[r] \ar[d] & (2,4,2,0)&& \\
	(0,2,0,2) \ar[r] \ar[d] & (1,3,1,3) &&& \\
	(1,2,3,2) &&&&}
$$
We then have the following relations:
\begin{equation}\begin{aligned}
\eta (0,2,0,2) &= \frac{11}{7} \psi^2 + \psi^2 \eta (2,1,1,0); \\
\eta (2,1,1,0) & = \frac{2}{7} \psi (0,1,2,1) + \psi \eta (0,1,2,1); \\
\eta (0,1,2,1) & = \frac{1}{7} \psi (0,2,0,2) + \psi \eta (0,2,0,2).
\end{aligned}\end{equation}
Now we can use these relations to compute the Picard--Fuchs equations associated to $(2,1,1,0)$, $(0,2,0,2),$ and $(0,1,2,1)$. We first do this for the period $(0,1,2,1)$: 
\begin{equation}\begin{aligned}
\eta (0,1,2,1) & = \frac{1}{7} \psi (0,2,0,2) \psi \eta (0,2,0,2) \\
(\eta-1)\eta (0,1,2,1) &= \frac{165}{49} \psi^3 (2,1,1,0) + \frac{26}{7} \psi^3 \eta (2,1,1,0) + \psi^3 \eta^2 (2,1,1,0) \\
(\eta-3)(\eta-1) \eta (0,1,2,1) &= \psi^4\left(\eta + \tfrac{2}{7}\right)\left(\eta + \tfrac{18}{7}\right) \left(\eta + \tfrac{22}{7}\right) 
\end{aligned}\end{equation}
This gives us the Picard--Fuchs equation for the period $(0,1,2,1)$:
\begin{equation}
\left[  (\eta-3)(\eta-1) \eta - \psi^4\left(\eta + \tfrac{2}{7}\right)\left(\eta + \tfrac{18}{7}\right) \left(\eta + \tfrac{22}{7}\right)\right] (0,1,2,1) = 0,
\end{equation}
Letting $u=\psi^4$ and $\sigma = u \displaystyle{\frac{\dD}{\dD u}}$, we get the following hypergeometric form:
\begin{align*}
\left[  (4\sigma-3)(4\sigma-1) 4\sigma - u\left(4\sigma + \tfrac{2}{7}\right)\left(4\sigma + \tfrac{18}{7}\right) \left(4\sigma + \tfrac{22}{7}\right)\right] (0,1,2,1) &= 0\\
\left[  (\sigma-\tfrac{3}{4})(\sigma-\tfrac{1}{4}) \sigma - u\left(\sigma + \tfrac{1}{14}\right)\left(\sigma + \tfrac{9}{14}\right) \left(\sigma + \tfrac{11}{14}\right)\right] (0,1,2,1) &= 0,
\end{align*}
which is the hypergeometric differential equation  $D(\tfrac{1}{14}, \tfrac{9}{14}, \tfrac{11}{14} ; 1, \tfrac{1}{4}, \tfrac{3}{4} \,|\,u)$.

We then do the same for $(0,2,0,2)$:
\begin{equation}\begin{aligned}
\eta (0,2,0,2) &=\frac{11}{7} \psi^2 + \psi^2 \eta (2,1,1,0) \\
(\eta - 2) \eta (0,2,0,2) &= \frac{36}{49} \psi^3 (0,1,2,1) + \frac{20}{7} \psi^3 \eta (0,1,2,1) + \psi^3 \eta^2 (0,1,2,1) \\
(\eta-3)(\eta-2) \eta (0,2,0,2) &= \psi^4\left(\eta + \tfrac{1}{7} \right)\left(\eta + \tfrac{9}{7}\right)\left(\eta + \tfrac{25}{7}\right) (0,2,0,2).
\end{aligned}\end{equation}
This gives us the Picard--Fuchs equation for the period $(0,2,0,2)$:
\begin{equation}
\left[ (\eta-3)(\eta-2)\eta - \psi^4\left(\eta + \tfrac{1}{7} \right)\left(\eta + \tfrac{9}{7}\right)\left(\eta + \tfrac{25}{7}\right)  \right] (0,2,0,2) = 0.
\end{equation}
By multiplying by $\psi$ and changing variables to $u=\psi^4$ and $\sigma = u \displaystyle{\frac{\dD}{\dD u}}$, we get:
\begin{align*}
\psi\left[ (\eta-3)(\eta-2)\eta - \psi^4\left(\eta + \tfrac{1}{7} \right)\left(\eta + \tfrac{9}{7}\right)\left(\eta + \tfrac{25}{7}\right)  \right] (0,2,0,2) &= 0 \\
\left[ (\eta-4)(\eta-3)(\eta-1) - \psi^4\left(\eta - \tfrac{6}{7} \right)\left(\eta + \tfrac{2}{7}\right)\left(\eta + \tfrac{18}{7}\right)  \right] \psi(0,2,0,2) &= 0 \\
\left[ (4\sigma-4)(4\sigma-3)(4\sigma-1) - u\left(4\sigma - \tfrac{6}{7} \right)\left(4\sigma + \tfrac{2}{7}\right)\left(4\sigma + \tfrac{18}{7}\right)  \right] \psi(0,2,0,2) &= 0 \\
\left[ (\sigma-1)(\sigma-\tfrac{3}{4})(\sigma-\tfrac{1}{4}) - u\left(\sigma - \tfrac{3}{14} \right)\left(\sigma + \tfrac{1}{14}\right)\left(\sigma + \tfrac{9}{14}\right)  \right] \psi(0,2,0,2) &= 0,
\end{align*}
which is the hypergeometric differential equation $D(\tfrac{1}{14}, \tfrac{9}{14}, \tfrac{-3}{14} ; 0, \tfrac{1}{4}, \tfrac{3}{4} \,|\,u)$.

We finally look at $(2,1,1,0)$:
\begin{equation}\begin{aligned}
\eta (2,1,1,0) & = \frac{2}{7} \psi (0,1,2,1) + \psi \eta (0,1,2,1) \\
(\eta-1)\eta (2,1,1,0)  &= \frac{2}{7} \psi \eta (0,1,2,1) + \psi \eta^2 (0,1,2,1) \\
	&= \frac{9}{49} \psi^2 (0,2,0,2) + \frac{10}{7} \psi^2 \eta (0,2,0,2) + \psi^2 \eta^2 (0,2,0,2) \\
(\eta-2)(\eta-1)\eta (2,1,1,0) &=\frac{9}{49} \psi^2\eta (0,2,0,2) + \frac{10}{7} \psi^2 \eta^2 (0,2,0,2) + \psi^2 \eta^3 (0,2,0,2) \\
	&= \psi^4\left(\eta + \tfrac{11}{7}\right) \left(\eta + \tfrac{15}{7}\right)\left(\eta + \tfrac{23}{7}\right) (2,1,1,0).
\end{aligned}\end{equation}
This gives us the Picard-Fuchs equation for the period $(2,1,1,0)$:
\begin{equation}
\left[ (\eta-2)(\eta-1)\eta - \psi^4 \left(\eta + \tfrac{11}{7}\right) \left(\eta + \tfrac{15}{7}\right)\left(\eta + \tfrac{23}{7}\right)  \right] (2,1,1,0) = 0
\end{equation}
By multiplying by $\psi^3$ and again changing variables 
we get:
\begin{align*}
\psi^3\left[ (\eta-2)(\eta-1)\eta - \psi^4 \left(\eta + \tfrac{11}{7}\right) \left(\eta + \tfrac{15}{7}\right)\left(\eta + \tfrac{23}{7}\right)  \right] (2,1,1,0) &= 0 \\
\left[ (\eta-5)(\eta-4)(\eta-3) - \psi^4 \left(\eta - \tfrac{10}{7}\right) \left(\eta - \tfrac{6}{7}\right)\left(\eta + \tfrac{2}{7}\right)  \right] \psi^3(2,1,1,0) &= 0 \\
\left[ (4\sigma-5)(4\sigma-4)(4\sigma-3) - u \left(4\sigma - \tfrac{10}{7}\right) \left(4\sigma - \tfrac{6}{7}\right)\left(4\sigma + \tfrac{2}{7}\right)  \right] \psi^3(2,1,1,0) &= 0 \\
\left[ (\sigma-\tfrac{5}{4})(\sigma-1)(\sigma-\tfrac{3}{4}) - u \left(\sigma - \tfrac{5}{14}\right) \left(\sigma - \tfrac{3}{14}\right)\left(\sigma + \tfrac{1}{14}\right)  \right] \psi^3(2,1,1,0) &= 0;
\end{align*}
at last, we have the hypergeometric differential equation $D(\tfrac{1}{14}, \tfrac{-5}{14}, \tfrac{-3}{14} ; 0, \tfrac{-1}{4}, \tfrac{1}{4} \,|\,u)$.
\end{proof}

\begin{lem}\label{prop:F1L3char3}
For the Klein--Mukai family $X_\psi$, 
\[ \begin{gathered}
\text{the period $(0,2,1,1)$  is annihilated by $D(\tfrac{3}{14}, \tfrac{5}{14}, \tfrac{13}{14} ; \tfrac{1}{4}, \tfrac{3}{4}, 1 \,|\,\psi^4)$,} \\ 
\text{the period $\psi(1,0,1,2) $  is annihilated by $D(\tfrac{-1}{14}, \tfrac{3}{14}, \tfrac{5}{14} ; 0, \tfrac{1}{4}, \tfrac{3}{4} \,|\,\psi^4)$, and} \\ 
\text{the period $\psi^3 (0,1,3,0)$  is annihilated by  $D(\tfrac{-11}{14}, \tfrac{-1}{14}, \tfrac{5}{14} ; \tfrac{-1}{4}, 0, \tfrac{1}{4}  \,|\,\psi^4)$.} 
\end{gathered} \]
\end{lem}

\begin{proof}
We use the following diagram:
$$
\xymatrix{ &&& (0,1,3,0) \ar[r] \ar[d] & (1,2,4,1) \\
	&& (2,1,2,-1) \ar[r] \ar[d]& (3,2,3,0) & \\
	& (1,0,1,2) \ar[r] \ar[d] & (2,1,3,3)&& \\
	(0,2,1,1) \ar[r] \ar[d]  & (1,3,2,2) &&& \\
	(1,2,4,1) &&&&}
$$
and then compute the following period relations:
\begin{equation}\begin{aligned}
\eta (0,1,3,0) &= \frac{10}{7} \psi (0,2,1,1) + \psi \eta (0,2,1,1); \\
\eta (0,2,1,1) &= \frac{5}{7} \psi (1,0,1,2) + \psi \eta (1,0,1,2); \\
\eta (1,0,1,2) &= \frac{-1}{7} \psi^2 (0,1,3,0) + \psi^2 \eta (0,1,3,0).
\end{aligned}\end{equation}

By cyclically using these relations, we get the following Picard--Fuchs equations:
\begin{equation}\begin{aligned}
\left[(\eta-3)(\eta-1)\eta - \psi^4\left(\eta + \tfrac{6}{7}\right)\left(\eta + \tfrac{10}{7}\right) \left( \eta + \tfrac{26}{7}\right)  \right] (0,2,1,1) &= 0;\\
\left[(\eta-3)(\eta - 2) \eta - \psi^4 \left(\eta + \tfrac{5}{7}\right)\left(\eta + \tfrac{13}{7}\right)\left(\eta + \tfrac{17}{7}\right)  \right] (1,0,1,2) &=0; \\
\left[(\eta-2)(\eta-1)\eta - \psi^4\left( \eta - \tfrac{1}{7}\right)\left(\eta + \tfrac{19}{7}\right)\left( \eta + \tfrac{31}{7}\right)  \right] (0,1,3,0) &=0.
\end{aligned}\end{equation}
We then multiply these equations above by $1, \psi,$ and $\psi^3$, respectively and then change coordinates to $u = \psi^4$ and $\sigma = u \displaystyle{\frac{\dD}{\dD u}}$ to obtain the following:
\begin{align*}
\left[(\sigma-\tfrac{3}{4})(\sigma-\tfrac{1}{4})\sigma - u\left(\sigma + \tfrac{3}{14}\right)\left(\sigma + \tfrac{5}{14}\right) \left( \sigma + \tfrac{13}{14}\right)  \right] (0,2,1,1) &= 0;\\
\left[(\sigma-1)(\sigma - \tfrac{3}{4})(\sigma-\tfrac{1}{4}) - u \left(\sigma - \tfrac{1}{14}\right)\left(\sigma + \tfrac{3}{14}\right)\left(\sigma + \tfrac{5}{14}\right)  \right] \psi(1,0,1,2) &=0; \\
\left[(\sigma-\tfrac{5}{4})(\sigma-1)(\sigma-\tfrac{3}{4}) - u\left( \sigma - \tfrac{11}{14}\right)\left(\sigma - \tfrac{1}{14}\right)\left( \sigma + \tfrac{5}{14}\right)  \right] \psi^3 (0,1,3,0) &=0.
\end{align*}
which are $D(\tfrac{3}{14}, \tfrac{5}{14}, \tfrac{13}{14} ; 1, \tfrac{1}{4}, \tfrac{3}{4} \,|\,u )$, $D(\tfrac{3}{14}, \tfrac{5}{14}, \tfrac{-1}{14} ; 0, \tfrac{1}{4}, \tfrac{3}{4} \,|\,u )$, and $D(\tfrac{-11}{14}, \tfrac{5}{14}, \tfrac{-1}{14} ; 0, \tfrac{1}{4}, \tfrac{-1}{4} \,|\,u)$, respectively.
\end{proof}

We conclude this section by combining these results.

\begin{proof}[{Proof of Proposition \textup{\ref{prop:kleinmukaidiff}}}]
Combine Lemmas ~\ref{prop:F1L3hol}, \ref{prop:F1L3char1}, and \ref{prop:F1L3char3}.
\end{proof}

\subsection{Remaining pencils}

For the remaining three pencils $\Fsf_2\Lsf_2$, $\Lsf_2\Lsf_2$, and $\Lsf_4$, the Picard--Fuchs equations can be derived in a similar manner.  The details can be found in Appendix \ref{appendix:pfs}; we state here only the results.

\begin{prop} \label{prop:F2L2diff}
The group  $H^2_{\textup{prim}}(X_{\Fsf_2\Lsf_2, \psi}, \C)$ has $15$ periods whose Picard--Fuchs equations are hypergeometric differential equations as follows:
\[ \begin{gathered}
\text{$3$ periods are annihilated by $D(\tfrac{1}{4}, \tfrac{1}{2}, \tfrac{3}{4} ; 1, 1, 1 \,|\,\psi^{-4})$,} \\ 
\text{ $2$ periods are annihilated by $D(\tfrac{1}{4}, \tfrac{3}{4} ; 1, \tfrac{1}{2} \,|\,\psi^{-4})$,} \\ 
\text{$2$ periods are annihilated by $D(\tfrac{1}{2} ; 1 \,|\,\psi^{4})$,} \\ 
\text{$4$ periods are annihilated by $D(\tfrac{1}{8}, \tfrac{5}{8} ; 1, \tfrac{1}{4} \,|\,\psi^{4})$, and } \\ 
\text{$4$ periods are annihilated by $D(\tfrac{1}{8}, \tfrac{-3}{8} ; 0, \tfrac{1}{4} \,|\,\psi^{4})$.} 
\end{gathered} \]
\end{prop}

\begin{proof}
See Proposition \ref{prop:F2L2diff-appendix}.
\end{proof}

\begin{prop} 
The group $H^2_{\textup{prim}}(X_{\Lsf_2\Lsf_2, \psi}, \C)$ has $13$ periods whose Picard--Fuchs equations are hypergeometric differential equations as follows:
\[ \begin{gathered}
\text{$3$ periods are annihilated by $D(\tfrac{1}{4}, \tfrac{1}{2}, \tfrac{3}{4} ; 1, 1, 1 \,|\,\psi^{-4})$,} \\
\text{$8$ periods are annihilated by $D(\tfrac{1}{8}, \tfrac{3}{8}, \tfrac58, \tfrac78 ; 0, \tfrac{1}{4}, \tfrac12, \tfrac34 \,|\,\psi^{4})$, and } \\
\text{$2$ periods are annihilated by $D(\tfrac{1}{4}, \tfrac{3}{4} ; 1,  \tfrac12 \,|\,\psi^{4})$.}
\end{gathered} \]
\end{prop}

\begin{proof}
See Proposition \ref{PFL2L2}.
\end{proof}

\begin{prop} 
The group  $H^2_{\textup{prim}}(X_{\Lsf_4, \psi},\C)$ has $19$ periods whose Picard--Fuchs equations are hypergeometric differential equations as follows:
\[ \begin{gathered}
\text{$3$ periods are annihilated by $D(\tfrac{1}{4}, \tfrac{1}{2}, \tfrac{3}{4} ; 1, 1, 1 \,|\,\psi^{-4})$,} \\ 
\text{$4$ periods are annihilated by $D(\tfrac{1}{5}, \tfrac{2}{5}, \tfrac{3}{5}, \tfrac{4}{5}; 1, \tfrac{1}{4}, \tfrac{1}{2}, \tfrac 34 \,|\, \psi^4)$,} \\ 
\text{$4$ periods are annihilated by  $D(\tfrac{-1}{5}, \tfrac{1}{5}, \tfrac{2}{5}, \tfrac{3}{5}; 0, \tfrac{1}{4}, \tfrac{1}{2}, \tfrac 34 \,|\,\psi^4)$,} \\ 
\text{$4$ periods are annihilated by $D( \tfrac{-2}{5}, \tfrac{-1}{5}, \tfrac{1}{5}, \tfrac{2}{5}; \tfrac{-1}{4}, 0, \tfrac{1}{4}, \tfrac{1}{2} \,|\,\psi^4)$, and} \\ 
\text{$4$ periods are annihilated by  $D(\tfrac{-3}{5}, \tfrac{-2}{5}, \tfrac{-1}{5}, \tfrac{1}{5}; 0, \tfrac{1}{4}, \tfrac{-1}{2}, \tfrac{-1}{4} \,|\,\psi^4)$.} 
\end{gathered} \]
\end{prop}

\begin{proof}
See Proposition \ref{prop:L4H2result}.
\end{proof}

\section{Explicit formulas for the number of points}\label{S:PointCounts}

In this section, we derive explicit formulas for the number of points and identify the hypergeometric periods according to the action of the group of symmetries, matching the Picard--Fuchs equations computed in section~\ref{S:PFEqns}.  

\subsection{Hypergeometric functions over finite fields} \label{defn:hgmdef} 
We begin by defining the finite field analogue of the generalized hypergeometric function (defined in section \ref{sec:hypdiffeq}); we follow Beukers--Cohen--Mellit \cite{BCM}.

Let $q=p^r$ be a prime power.  We use the convenient abbreviation $$\qq \colonequals q-1.$$  

Let $\omega \colon \F_q^{\times} \to \C^\times$ be a generator of the character group on $\F_q^{\times}$. Let $\Theta \colon \F_q \to \C^\times$ be a nontrivial (additive) character, defined as follows: let $\zeta_p \in \C$ be a primitive $p$th root of unity, and define $\Theta(x)=\zeta_p^{\Tr_{\F_q|\F_p}(x)}$.   For $m \in \Z$, we define the \defi{Gauss sum}
\begin{equation} \label{GS}
g(m)\colonequals \sum_{x \in \F_q^{\times}} \omega(x)^m \Theta(x).
\end{equation}
We suppress the dependence on $q$ in the notation, and note that $g(m)$ depends only on $m \in \Z/\qq\Z$ (and the choice of $\omega$ and $\zeta_p$).

\begin{rmk}\label{additiveindependence} 
Every generator of the character group on $\F_q^{\times}$ is of the form $\omega_k(x) \colonequals \omega(x)^k$ for $k \in (\Z/\qq\Z)^\times$, and
\[ \sum_{x \in \F_q^\times} \omega_k(x)^m \Theta(x) = g(km). \]
Similarly, every additive character of $\F_q$ is of the form $\Theta_k(x) \colonequals \zeta_p^{k\Tr(x)}$ for $k \in (\Z/p\Z)^\times$, and 
\[ \sum_{x \in \F_q^\times} \omega(x)^m \Theta_k(x) = \omega(k)^{-m}g(m) \]
(see e.g.\ Berndt \cite[Theorem 1.1.3]{Berndt}).  Accordingly, we will see below that our definition of finite field hypergeometric functions will not depend on these choices.  
\end{rmk}
 
We will need four basic identities for Gauss sums.
\begin{lem}\label{gausssumidentities}
The following relations hold:
\begin{enumalph}
\item $g(0) = -1$.
\item $g(m) g(-m) = (-1)^m q$ for every $m \not\equiv 0 \pmod{\qq}$, and in particular 
\[ g(\tfrac{\qq}{2})^2 = (-1)^{\qq\qm/2} q. \]
\item For every $N \mid \qq$ with $N>0$, we have
\begin{equation}\label{HasseDavenport}
g(Nm) = -\omega (N)^{Nm} \prod_{j=0}^{N-1} \frac{g(m + j\qq\qm/N)}{g(j\qq\qm/N)}.
\end{equation}
\item $g(pm)=g(m)$ for all $m \in \Z$.
\end{enumalph}
\end{lem}

\begin{proof}
For parts (a)--(c), see Cohen \cite[Lemma 2.5.8, Proposition 2.5.9, Theorem 3.7.3]{Cohen2}.  For (d), we replace $x$ by $x^p$ in the definition and use the fact that $\Theta(x^p)=\Theta(x)$ as it factors through the trace.
\end{proof}

\begin{rmk}Lemma \ref{gausssumidentities}(c) is due originally to Hasse and Davenport, and is called the \defi{Hasse--Davenport product relation}.\end{rmk}

We now build our hypergeometric sums.  Let $\pmb{\alpha}=\{\alpha_1,\dots,\alpha_d\}$ and $\pmb{\beta}=\{\beta_1,\dots,\beta_d\}$ be multisets of $d$ rational numbers.  Suppose that $\pmb{\alpha}$ and $\pmb{\beta}$ are \defi{disjoint modulo $\Z$}, i.e., $\alpha_i-\beta_j \not\in \Z$ for all $i,j=1,\dots,d$. 

Based on work of Greene \cite{Greene}, Katz \cite[p.\ 258]{Katz}, but normalized following McCarthy \cite[Definition 3.2]{McCarthy} and Beukers--Cohen--Mellit \cite[Definition 1.1]{BCM}, we make the following definition.

\begin{defn}\label{classic HGF over FF} 
Suppose that
\begin{equation} \label{eqn:qqalpha}
\qq \alpha_i, \qq \beta_i \in \Z
\end{equation}
for all $i=1,\dots,d$.  For $t \in \F_q^\times$, we define the \defi{finite field hypergeometric sum} by
\begin{equation} 
H_q(\pmb{\alpha}, \pmb{\beta} \,|\, t) \colonequals -\frac{1}{\qq} \sum_{m=0}^{q-2} \omega((-1)^dt)^m G(m+\pmb{\alpha}\qq,-m-\pmb{\beta}\qq)
\end{equation}
where
\begin{equation} \label{eqn:gmalphabeta}
G(m+\pmb{\alpha}\qq,-m-\pmb{\beta}\qq) \colonequals \prod_{i=1}^d \frac{g(m+ \alpha_i\qq)g(-m - \beta_i\qq)}{g(\alpha_i \qq)g(-\beta_i \qq)}
\end{equation}
for $m \in \Z$.
\end{defn}

In this definition (and the related ones to follow), the sum $H_q(\pmb{\alpha},\pmb{\beta}\,|\,t)$ only depends on the classes in $\Q/\Z$ of the elements of $\pmb{\alpha}$ and $\pmb{\beta}$.  Moreover, the sum is independent of the choice of characters $\omega$ and $\Theta$ by a straightforward application of Remark \ref{additiveindependence}. 
The hypothesis \eqref{eqn:qqalpha} is unfortunately rather restrictive---but it is necessary for the definition to make sense as written.  Fortunately, Beukers--Cohen--Mellit \cite{BCM} provided an alternate definition that allows all but finitely many $q$ under a different hypothesis, as follows.  

\begin{defn} \label{def:fieldofdef}
The \defi{field of definition} $K_{\pmb{\alpha},\pmb{\beta}} \subset \C$ associated to $\pmb{\alpha},\pmb{\beta}$ is the field generated by the coefficients of the polynomials
\begin{equation}
\prod_{j=1}^d (x-e^{2\pi\sqrt{-1} \alpha_j}) \quad \text{and} \quad \prod_{j=1}^d (x-e^{2\pi\sqrt{-1} \beta_j}). 
\end{equation}
\end{defn}

\noindent Visibly, the number field $K_{\pmb{\alpha},\pmb{\beta}}$ is an abelian extension of $\Q$. 

 Suppose that $\pmb{\alpha},\pmb{\beta}$ is defined over $\Q$, i.e., $K_{\pmb{\alpha},\pmb{\beta}}=\Q$. Then by a straightforward verification, there exist $p_1, \ldots, p_r,q_1, \ldots, q_s \in \Z_{\geq 1}$ such that
\begin{equation}
\prod_{j=1}^d \frac{ (x-e^{2\pi\sqrt{-1} \alpha_j})}{(x-e^{2\pi\sqrt{-1} \beta_j})} = \frac{\prod_{j=1}^r x^{p_j} - 1}{\prod_{j=1}^s x^{q_j} - 1}.
\end{equation}
Recall we require the $\pmb{\alpha},\pmb{\beta}$ to be disjoint, which implies that the sets $\{p_1,\dots,p_r\}$ and $\{q_1,\dots,q_s\}$ are also disjoint. 
 
Let  $D(x)  \colonequals  \gcd(\prod_{j=1}^r (x^{p_j} - 1), \prod_{j=1}^s (x^{q_j} - 1))$ and $M \colonequals  \bigl(\prod_{j=1}^r p_j^{p_j}\bigr) \bigl(\prod_{j=1}^s q_j^{-q_j}\bigr)$.  Let $\epsilon \colonequals (-1)^{\sum_{j=1}^s q_j}$, and let $s(m) \in \Z_{\geq 0}$ be the multiplicity of the root $e^{2\pi\sqrt{-1} m / \qq}$ in $D(x)$.  Finally, abbreviate
\begin{equation} \label{eqn:gpmqm}
g(\pmb{p}m,-\pmb{q}m) \colonequals g(p_1m) \cdots g(p_rm) g(-q_1m) \cdots g(-q_sm).
\end{equation}

For brevity, we say that $q$ is \defi{good} for $\pmb{\alpha},\pmb{\beta}$ if $q$ is coprime to the least common denominator of $\pmb{\alpha} \cup \pmb{\beta}$.

\begin{defn}\label{BCM HGF}
Suppose that $\pmb{\alpha},\pmb{\beta}$ are defined over $\Q$ and $q$ is good for $\pmb{\alpha},\pmb{\beta}$.  For $t \in \F_q^\times$, define
\begin{equation} 
H_q(\pmb{\alpha}, \pmb{\beta} \,|\, t) = \frac{(-1)^{r+s}}{1-q} \sum_{m=0}^{q-2} q^{-s(0) + s(m)} g(\pmb{p}m,-\pmb{q}m)\omega(\epsilon M^{-1}t)^m.
\end{equation}
\end{defn}

Again, the hypergeometric sum $H_q(\pmb{\alpha},\pmb{\beta}\,|\,t)$ is independent of the choice of characters $\omega$ and $\Theta$.  The independence on $\omega$ is just as with the previous definition, and in this case the independence from $\Theta$ comes from the fact that every root of unity has its conjugate, and so again any additional factors from changing additive characters cancel out. The apparently conflicting notation is justified by the following result, showing that Definition \ref{BCM HGF} is more general.

\begin{prop}[{Beukers--Cohen--Mellit \cite[Theorem 1.3]{BCM}}] \label{prop:bcmdef}
Suppose that $\pmb{\alpha},\pmb{\beta}$ are defined over $\Q$ and that \eqref{eqn:qqalpha} holds.  Then Definitions \textup{\ref{classic HGF over FF}} and \textup{\ref{BCM HGF}} agree.
\end{prop}

\subsection{A hybrid sum} \label{defn:hgmdef-hybrid}  

We will need a slightly more general hypothesis than allowed in the previous section.  We do not pursue the most general case as it is rather combinatorially involved, poses some issues of algebraicity, and anyway is not needed here.  Instead, we isolate a natural case, where the indices are not defined over $\Q$ but neither does \eqref{eqn:qqalpha} hold, that is sufficient for our purposes.

\begin{defn}
We say that $q$ is \defi{splittable} for $\pmb{\alpha},\pmb{\beta}$ if there exist partitions 
\begin{equation} \label{eqn:partitalph} 
\pmb{\alpha} = \pmb{\alpha}_0 \sqcup \pmb{\alpha}' \textup{ and } \pmb{\beta} = \pmb{\beta}_0 \sqcup \pmb{\beta}' 
\end{equation}
where $\pmb{\alpha}_0,\pmb{\beta}_0$ are defined over $\Q$ and 
\[ \qq \alpha_i',\qq \beta_j' \in \ZZ \]
for all $\alpha_i' \in \pmb{\alpha}'$ and all $\beta_j' \in \pmb{\beta}'$.
\end{defn}

\begin{exm}
If \eqref{eqn:qqalpha} holds, then $q$ is splittable for $\pmb{\alpha},\pmb{\beta}$ taking $\pmb{\alpha}=\pmb{\alpha}'$ and $\pmb{\beta}=\pmb{\beta}'$ and $\pmb{\alpha}_0=\pmb{\beta}_0=\emptyset$.  Likewise, if $\pmb{\alpha},\pmb{\beta}$ is defined over $\Q$, then $q$ is splittable for $\pmb{\alpha},\pmb{\beta}$ for all $q$.
\end{exm}

\begin{exm}\label{F1L3 hypergeometric example}
A splittable case that arises for us (up to a Galois action) in Proposition~\ref{prop:F1L3} below is as follows.  Let $\pmb{\alpha} = \{\frac{1}{14}, \frac{9}{14}, \frac{11}{14}\}$ and $\pmb{\beta} = \{0, \frac14, \frac34\}$.  We cannot use Definition \ref{BCM HGF} since $(x-e^{2\pi\sqrt{-1}/14})(x-e^{18\pi i/14})(x-e^{22\pi\sqrt{-1}/14}) \not \in \Q[x]$.  When $q \equiv 1 \pmod{28}$, we may use Definition \ref{classic HGF over FF}; otherwise we may not.  
However, when $q \equiv 1 \pmod{7}$ is odd, then $q$ is splittable for $\pmb{\alpha},\pmb{\beta}$: we may take $\pmb{\alpha}_0=\emptyset$, $\pmb{\alpha}'=\pmb{\alpha}$ and $\pmb{\beta}_0=\pmb{\beta}$, $\pmb{\beta}'=\emptyset$.
\end{exm}

It is now a bit notationally painful but otherwise straightforward to generalize the definition for splittable $q$, providing a uniform description in all cases we consider.  Suppose that $q$ is splittable for $\pmb{\alpha},\pmb{\beta}$.  Let $\pmb{\alpha}_0$ be the union of all submultisets of $\pmb{\alpha}$ that are defined over $\Q$; then $\pmb{\alpha}_0$ is defined over $\Q$.  Repeat this for $\pmb{\beta}_0$.  
Let $p_{1}, \ldots, p_{r}, q_{1}, \ldots, q_{s}$ be such that 
$$
\frac{\prod_{\alpha_{0j} \in \pmb{\alpha}_0} (x- e^{2\pi\sqrt{-1} \alpha_{0j}})}{\prod_{\beta_{0j} \in \pmb{\beta}_0} (x-e^{2\pi\sqrt{-1} \beta_{0j}})} = \frac{\prod_{j=1}^{r} (x^{p_{j}} -1)}{\prod_{j=1}^{s} (x^{q_{j}}-1)}.
$$
As before, let 
\[ D(x) \colonequals  \gcd(\textstyle{\prod}_{j=1}^{r} x^{p_{j}} - 1, \textstyle{\prod}_{j=1}^{s} x^{q_{j}} - 1) \quad \text{and} \quad
M \colonequals  \frac{\textstyle{\prod}_{j=1}^{r} p_{j}^{p_{j}}}{\textstyle{\prod}_{j=1}^{s} q_{j}^{q_{j}}}
\]
and let $s(m)$ be the multiplicity of the root $e^{2\pi\sqrt{-1} m / \qq}$ in $D(x)$.  
Finally, let $\delta \colonequals \deg D(x)$.
We again abbreviate
\begin{equation} \label{eqn:gpmqm0}
g(\pmb{p}m,-\pmb{q}m) \colonequals \prod_{i=1}^{r} g(p_{i}m) \prod_{i=1}^{s} g(-q_{i}m)
\end{equation}
for $m \in \Z$ and
\begin{equation} \label{eqn:gmalphabeta0}
G(m+\pmb{\alpha}'\qq,-m-\pmb{\beta}'\qq) \colonequals \prod_{\alpha_i' \in \pmb{\alpha}'}\frac{g(m+ \alpha_i'\qq)}{g(\alpha_i \qq)} 
\prod_{\beta_i' \in \pmb{\beta}'}\frac{g(-m-\beta_i'\qq)}{g(-\beta_i \qq)}.
\end{equation}

\begin{defn}\label{Hybrid Definition}
Suppose that $q$ is good and splittable for $\pmb{\alpha},\pmb{\beta}$.  For $t \in \F_q^\times$, with the notation above we define the \defi{finite field hypergeometric sum}
\[ H_q(\pmb{\alpha}, \pmb{\beta} \, | \, t)  \colonequals  \frac{(-1)^{r+s}}{1-q} \sum_{m=0}^{q-2} q^{- s(0) + s(m)} 
G(m+\pmb{\alpha}'\qq,-m-\pmb{\beta}'\qq) g(\pmb{p}m,-\pmb{q}m)\omega((-1)^{d+\delta} Mt)^m. \]
\end{defn}

The following proposition then shows that our definition encompasses the previous ones.

\begin{prop} \label{prop:qalbet}
Suppose that $q$ is good and splittable for $\pmb{\alpha},\pmb{\beta}$.  Then the following statements hold.
\begin{enumalph}
\item The hypergeometric sum $H_q(\pmb{\alpha},\pmb{\beta}\,|\,t)$ in Definition \textup{\ref{Hybrid Definition}} is independent of the choice of characters $\omega$ and $\Theta$.
\item If $\alpha_i \qq, \beta_i \qq \in \Z$ for all $i=1,\dots,d$, then Definitions \textup{\ref{classic HGF over FF}} and \textup{\ref{Hybrid Definition}} agree.
\item If $\pmb{\alpha},\pmb{\beta}$ are defined over $\Q$, then Definitions \textup{\ref{BCM HGF}} and \textup{\ref{Hybrid Definition}} agree.
\end{enumalph}
\end{prop}

\begin{proof}
Part (c) follows directly from $\pmb{\alpha}_0 = \pmb{\alpha}$ and $\pmb{\beta}_0 = \pmb{\beta}$ (and $\pmb{\alpha}' = \pmb{\beta}' = \emptyset$), so the definitions in fact coincide.  Part (a) follows directly from the independence from $\Theta$ and $\omega$ of each part of the hybrid sum. 

Part (b) follows by the same argument (due to Beukers--Cohen--Mellit) as in Proposition \ref{prop:bcmdef}; for completeness, we give a proof in Lemma \ref{lem:appendixhybridagree} in the appendix.
\end{proof}

Suppose that $q$ is good and splittable for $\pmb{\alpha},\pmb{\beta}$ and let $t \in \F_q^\times$.  Then by construction $H_q(\pmb{\alpha},\pmb{\beta}\,|\,t) \in \Q(\zeta_{\qq},\zeta_p)$.  Since $\gcd(p,\qq)=1$, we have
\[ \Gal(\Q(\zeta_{\qq},\zeta_p)\,|\,\Q) \simeq \Gal(\Q(\zeta_{\qq})\,|\,\Q) \times \Gal(\Q(\zeta_{p})\,|\,\Q). \]
We now descend the hypergeometric sum to its field of definition, in two steps.  

\begin{lem} \label{lem:descenttoQQ}
We have $H_q(\pmb{\alpha},\pmb{\beta}\,|\,t) \in \Q(\zeta_{\qq})$.
\end{lem}

\begin{proof}
The action of $\Gal(\Q(\zeta_p)\,|\,\Q) \simeq (\Z/p\Z)^\times$ by $\zeta_p \mapsto \zeta_p^k$ changes only the additive character $\Theta$.  By Proposition \ref{prop:qalbet}, the sum is independent of this choice, so it descends by Galois theory.
\end{proof}

The group $\Gal(\Q(\zeta_{\qq})\,|\,\Q) \simeq  (\Z/\qq\Z)^\times$ by $\sigma_k(\zeta_{\qq})=\zeta_({\qq})^k$ for $k \in (\Z/\qq\Z)^\times$ acts on the finite field hypergeometric sums as follows.

\begin{lem} \label{lem:fieldofdef}
The following statements hold.
\begin{enumalph}
\item Let $k \in \Z$ be coprime to $\qq$.  Then $\sigma_k(H_q(\pmb{\alpha},\pmb{\beta}\,|\, t)) = H_q(k\pmb{\alpha},k\pmb{\beta}\,|\,t)$.
\item We have $H_q(\pmb{\alpha},\pmb{\beta}\,|\, t) \in K_{\pmb{\alpha},\pmb{\beta}}$.
\item We have $H_q(p\pmb{\alpha},p\pmb{\beta}\,|\, t) = H_q(\pmb{\alpha},\pmb{\beta}\,|\, t^p)$.
\end{enumalph}
\end{lem}

\begin{proof}
To prove (a), note $\sigma_k(\omega(x))=\omega^k(x)$ since $\omega$ takes values in $\mu_{\qq}$; therefore $\sigma_k(g(m))=g(km)$, and we have both
\[ \sigma_k(g(\pmb{p} m, -\pmb{q} m)) = g(\pmb{p} km, -\pmb{q} km) \]
and
\[ \sigma_k(G(m+\pmb{\alpha}'\qq,-m-\pmb{\beta}'\qq)) = G(km+k\pmb{\alpha}'\qq,-km-k\pmb{\beta}'\qq)). \]
We have $s(km)=s(m)$ since $D(x) \in \Q[x]$.  Moreover, $k\pmb{\alpha}_0=\pmb{\alpha}_0$ and the same with $\pmb{\beta}$, so the values $p_{i},q_{i}$ remain the same when computed for $k\pmb{\alpha},k\pmb{\beta}$.  Now plug these into the definition of $H_q(\pmb{\alpha},\pmb{\beta}\,|\,t)$ and just reindex the sum by $km \leftarrow m$ to obtain the result.

Part (b) follows from part (a): the field of definition $K_{\pmb{\alpha},\pmb{\beta}}$ is precisely the fixed field under the subgroup of $k \in (\Z/\qq\Z)^\times$ such that $k\pmb{\alpha},k\pmb{\beta}$ are equivalent to $\pmb{\alpha},\pmb{\beta}$ as multisets in $\Q/\Z$.  

Finally, part (c).  Starting with the left-hand side, we reindex $m \leftarrow pm$ then substitute using Lemma \ref{gausssumidentities}(d)'s implication that $g(pm)=g(m)$ to get 
\begin{center}
$G(pm+p\pmb{\alpha}'\qq,-pm-p\pmb{\beta}'\qq)=G(m+\pmb{\alpha}'\qq,-m-\pmb{\beta}'\qq)$ and $g(\pmb{p}(pm),\pmb{q}(pm))=g(\pmb{p}m,\pmb{q}m)$, 
\end{center}
noting that the quantities $\pmb{p}$ and $\pmb{q}$ do not change, as $p\pmb{\alpha}_0 = \pmb{\alpha}_0$ and $p\pmb{\beta}_0=\pmb{\beta}_0$ modulo $\Z$ (as they are defined over $\Q$).  Noting that $((-1)^{d+\delta} M)^p = (-1)^{d+\delta}M \in \F_p \subseteq \F_q$, we then obtain the result.
\end{proof}

Before concluding this primer on finite field hypergeometric functions, we combine the Gauss sum identities and our hybrid definition to expand one essential example; this gives a flavor of what is to come.  First, we prove a new identity.

\begin{lem}\label{Gauss Identity Lemma F1L3}
We have the following identity of Gauss sums:
\[ g(\tfrac{\qq}{14})g(\tfrac{9\qq}{14})g(\tfrac{11\qq}{14}) = g(\tfrac{\qq}{2})^3 = (-1)^{\qq\qm/2}q g(\tfrac{\qq}{2}). \]
\end{lem}
\begin{proof}
Since $q$ is odd, we use the Hasse--Davenport product relation (Lemma \ref{gausssumidentities}(c)) for $N=2$ and $m = \tfrac{\qq}{14}, \tfrac{9\qq}{14},\tfrac{11\qq}{14}$, solving for $g(\tfrac{\qq}{14}), g(\tfrac{9\qq}{14}),g(\tfrac{11\qq}{14})$, respectively to find:
\begin{align*} 
g(\tfrac{\qq}{14}) &= \frac{g(\tfrac{\qq}{7}) g(\tfrac{\qq}{2})}{g(\tfrac{11\qq}{7})} \omega(2)^{-\qq\qm/7} \\
g(\tfrac{9\qq}{14}) &= \frac{g(\tfrac{9\qq}{7}) g(\tfrac{\qq}{2})}{g(\tfrac{\qq}{7})} \omega(2)^{-9\qq\qm/7} \\
g(\tfrac{11\qq}{14}) &= \frac{g(\tfrac{11\qq}{7}) g(\tfrac{\qq}{2})}{g(\tfrac{9\qq}{7})} \omega(2)^{-11\qq\qm/7}.
\end{align*}
Multiply all of these together, divide by $g(\tfrac{\qq}{2})$, and cancel to obtain:
\begin{equation}\begin{aligned}
\frac{g(\tfrac{\qq}{14})g(\tfrac{9\qq}{14})g(\tfrac{11\qq}{14})}{g(\tfrac{\qq}{2})} 
&= g(\tfrac{\qq}{2})^2 \omega(2)^{-3\qq} = (-1)^{\qq\qm/2} q
\end{aligned}\end{equation}
applying Lemma \ref{gausssumidentities}(b) in the last step.
\end{proof}

Next, we consider our example.

\begin{exm}\label{HypergeometricFunction for first F1L3 mod 7 example}
Going back to Example~\ref{F1L3 hypergeometric example}, in the case where $q \equiv 1 \pmod 7$ and $q$ odd, we have $\pmb{\alpha} = \{\frac{1}{14}, \frac{9}{14}, \frac{11}{14}\}$ and $\pmb{\beta} = \{0, \frac14, \frac34\}$.   Then 
$\pmb{\alpha}_0=\emptyset$ and $\pmb{\beta}_0=\pmb{\beta}$.  Thus
\[ 
\frac{\prod_{\alpha_{0j} \in \pmb{\alpha}_0} (x- e^{2\pi\sqrt{-1} \alpha_{0j}})}{\prod_{\beta_{0j} \in \pmb{\beta}_0} (x-e^{2\pi\sqrt{-1} \beta_{0j}})} = \frac{1}{(x-1)(x^2+1)} = \frac{(x^2-1)}{(x-1)(x^4-1)}. \]
Thus $D(x) = x^2-1$ and $M = 4^3$; and $s(m) = 1$ if $m=0, \tfrac{\qq}{2}$ and $s(m)=0$ otherwise.  Therefore Definition \ref{Hybrid Definition} and simplification using Lemma \ref{gausssumidentities}(a)--(b) gives
\[ \begin{aligned}
&H_q(\tfrac{1}{14}, \tfrac{9}{14}, \tfrac{11}{14}; 0,\tfrac14,\tfrac34 \,|\, t) \\
&\quad= \frac{-1}{1-q} \sum_{m=0}^{q-2} q^{s(m) - 1} \frac{g(m+ \frac{1}{14}\qq)g(m+ \frac{9}{14}\qq)g(m+ \frac{11}{14}\qq)}{g(\frac{1}{14} \qq)g(\frac{9}{14} \qq)g(\frac{11}{14} \qq)} g(2m) g(-m) g(-4m) \omega(-4^3t)^m.
\end{aligned} \]
When $m=0$, the summand is just $(-1)(-1)^3/(1-q)=1/\qq$. When $m=\tfrac{\qq}{2}$, applying Lemma \ref{Gauss Identity Lemma F1L3} we obtain
\[ \frac{-g(\frac{4\qq}{7})g(\frac{\qq}{7})g(\frac{2\qq}{7})}{\qq g(\frac{1}{14} \qq)g(\frac{9}{14} \qq)g(\frac{11}{14} \qq)} g(\tfrac{\qq}{2}) \omega(-4^3t)^{\qq\qm/2} = \frac{-1}{q\qq} g(\tfrac{\qq}{7})g(\tfrac{2\qq}{7})g(\tfrac{4\qq}{7})\omega(t)^{\qq\qm/2}.
\]
Therefore
\begin{equation} \label{eqn:hgmtherealdeal}
\begin{aligned}
&H_q(\tfrac{1}{14}, \tfrac{9}{14}, \tfrac{11}{14}; 0,\tfrac14,\tfrac34 \,|\, t) \\
&\quad= \frac{1}{\qq} -\frac{1}{q\qq} g(\tfrac{\qq}{7})g(\tfrac{2\qq}{7})g(\tfrac{4\qq}{7})\omega(t)^{\qq\qm/2} \\
&\qquad +\frac{1}{q\qq} \sum_{\substack{m=1 \\ m \neq \qq\qm/2}}^{q-2} \frac{g(m+ \frac{1}{14}\qq)g(m+ \frac{9}{14}\qq)g(m+ \frac{11}{14}\qq)}{g(\frac{1}{14} \qq)g(\frac{9}{14} \qq)g(\frac{11}{14} \qq)} g(2m) g(-m) g(-4m) \omega(-4^3t)^m.
\end{aligned} 
\end{equation}
\end{exm}

\subsection{Counting points}

Following work of Delsarte \cite{Delsarte} and Furtado Gomida \cite{FG}, Koblitz \cite{Koblitz} gave a formula for the number of points on monomial deformations of diagonal hypersurfaces (going back to Weil \cite{Weil}).  In this subsection, we outline their approach for creating closed formulas that compute the number of points for hypersurfaces in projective space in terms of Gauss sums. 

Let $X \subseteq \PP^n$ be the projective hypersurface over $\F_q$ defined by the vanishing of the nonzero polynomial 
\[
\sum_{i=1}^r a_i x_0^{\nu_{i0}}\cdots x_n^{\nu_{in}} \in \F_q[x_0,\dots,x_n]
\]
so that $a_i \in \F_q$ and $\nu_{ij} \in \Z_{\geq 0}$ for $i=1,\dots,r$ and $j=0,\dots,n$.  Suppose that $\qq \colonequals q-1$ does not divide any of the $\nu_{ij}$.  Let $U$ be the intersection of $X$ with the torus $\mathbb{G}_m^{n+1}/\mathbb{G}_m \subseteq \PP^n$, so that the points of $U$ are the points of $X$ with all nonzero coordinates.

Let $\tyS$ be the set of $s = (s_1,\ldots, s_r) \in (\Z/\qq\Z)^r$ such that the following condition holds:
\begin{equation}\label{WCong}
\textstyle{\sum_{i=1}^r} s_i\equiv0\psmod{\qq} \quad \text{and}\quad \sum_{i=1}^r \nu_{ij}s_i\equiv0\psmod{\qq} \quad \text{ for all $j=1,\dots,n$}. 
\end{equation}
Let $\mu_\qq$ be the group of $\qq$-th roots of unity. Any element $s = (s_1,\ldots, s_r) \in \tyS$ corresponds to a multiplicative character \begin{align*}
\chi_{s} \colon \mu_({\qq})^r &\rightarrow \C^\times \\
\chi_{s}(x_1, \ldots, x_r) &= \omega( \Pi_{i=1}^r x_i^{s_i}).
\end{align*}

Given $s \in \tyS$, we define
\begin{equation} \label{The cs's}
c_s \colonequals \frac{({\qq})^{n-r+1}}{q}\prod_{i=1}^r g(s_i)
\end{equation}
if $s \neq 0$ and
\[ c_{0} \colonequals ({\qq})^{n-r+1}\frac{({\qq})^{r-1}-(-1)^{r-1}}{q}. \]

With this notation, we have the following result of Koblitz, rewritten in terms of Gauss sums so that we can apply it in our context.

\begin{thm}[Koblitz]\label{Dels} 
We have
\begin{equation} \label{eqn:Nstaralpha}
\# U(\F_q)=\sum_{s\in\tyS}\omega(a)^{-s}c_{s}.
\end{equation}
where $\omega(a)^{-s} \colonequals \omega(a_1^{-s_1}\cdots a_r^{-s_r})$.
\end{thm}

\begin{proof}
We unpack and repack a bit of notation.  Koblitz \cite[Theorem 1]{Koblitz} proves that 
\[ \# U(\F_q)=\sum_{s}\omega(a)^{-s}c_{s}' \] 
where the sum is over all characters of $\mu_({\qq})^r / \Delta$ where $\Delta$ is the diagonal---this set is in natural bijection with the set $S$---and where for $s \neq 0$
$$
c_s' = -\frac{1}{q} ({\qq})^{n-r+1} J(s_1, \ldots, s_r)
$$
where $J(s_1,\ldots, s_r)$ is the Jacobi sum and where $c_0'=c_0$ as in~\eqref{The cs's}.  It only remains to show that $c'_s=c_s$ for $s \neq 0$.  If $s_i \neq 0$ for all $i$, then \cite[(2.5)]{Koblitz}
$$
J(s_1,\ldots, s_r) = \frac{g(s_1) \cdots g(s_r)}{g(s_1+\ldots+s_r)} = - g(s_1) \cdots g(s_r)
$$
so $c_s' = c_s$ by definition.  If $r>1$ and $s_i =0$ for some $i$, then \cite[below (2.5)]{Koblitz}
\[ J(s_1,\ldots, s_r) = -J(s_1, \ldots, s_{i-1},s_{i+1}, \ldots, s_r), \]
so iterating and using Lemma \ref{gausssumidentities}(a),
\begin{equation}
J(s_1,\ldots, s_r) = -\prod_{\substack{i=1,\dots,r \\ s_i = 0}} (-1) \prod_{\substack{i=1,\dots,r \\ s_i \neq 0}} g(s_i)  = -\prod_{\substack{i=1,\dots,r \\ s_i = 0}} g(0) \prod_{\substack{i=1,\dots,r \\ s_i \neq 0}} g(s_i) = -  \prod_{i=1}^r g(s_i). \qedhere
\end{equation}
\end{proof}

In the remaining sections, we apply the preceding formulas to each of our five pencils.

\subsection{The Dwork pencil \texorpdfstring{$\Fsf_4$}{Fsf4}}

 In this subsection, we will give a closed formula in terms of finite field hypergeometric sums for the number of points in a given member of the Dwork family. Throughout this section, we suppose that $q$ is odd.

\begin{prop}\label{prop:F4}
For $\psi \in \F_q^\times$, the following statements hold.
\begin{enumalph}
\item If $q\equiv 3\pmod 4$, then
\[ \#X_{\Fsf_4,\psi}(\F_q)=q^2+q+1+H_q(\tfrac{1}{4},\tfrac{1}{2},\tfrac{3}{4};0,0,0\,|\,\psi^{-4})-3qH_q(\tfrac{1}{4},\tfrac{3}{4};0, \tfrac{1}{2}\,|\,\psi^{-4}).\]

\item If $q\equiv 1\pmod{4}$, then
\begin{align*}
\#X_{\Fsf_4,\psi}(\F_q) &= q^2+q+1+H_q(\tfrac{1}{4},\tfrac{1}{2},\tfrac{3}{4};0,0,0\,|\,\psi^{-4})+3qH_q(\tfrac{1}{4},\tfrac{3}{4};0, \tfrac{1}{2}\,|\,\psi^{-4}) \\
&\qquad +12(-1)^{(q-1)/4}qH_q(\tfrac{1}{2};0\,|\,\psi^{-4}). 
\end{align*}
\end{enumalph}
\end{prop}
  
Proposition \ref{prop:F4} has several equivalent formulations and has seen many proofs: see section \ref{subsec:PreviousWork} in the introduction for further references.  We present another proof for completeness and to illustrate the method we will apply to all five families in this well-studied case.  

\begin{rmk}
Quite beautifully, the point counts in Proposition \ref{prop:F4} in terms of finite field hypergeometric sums match (up to twisting factors) the indices with multiplicity in the Picard--Fuchs equations computed in Proposition \ref{prop:F4H2result}.  Although we are not able to use this matching directly, it guides the decomposition of the sums by means of lemmas that can be proven in a technical but direct manner.
\end{rmk}

We prove Proposition \ref{prop:F4} in four steps:
\begin{enumerate}
\item[1.] We compute the relevant characters and cluster them.
\item[2.] We use Theorem \ref{Dels} to count points where no coordinate is zero and rewrite the sums into hypergeometric functions. 
\item[3.] We count points where at least one coordinate is zero.
\item[4.] We combine steps 2 and 3 to finally prove Proposition \ref{prop:F4}.
\end{enumerate}
The calculations are somewhat involved, but we know how to cluster and the answer up to the scaling factors in front: indeed, the parameters of the finite field hypergeometric sums are given by the calculation of the Picard--Fuchs equations for the Dwork pencil given by Proposition~\ref{prop:F4H2result}.  

\subsubsection*{Step 1: Computing and clustering the characters}

In order to use Theorem \ref{Dels} we must compute the subset $\tyS\subset (\Z/\qq\Z)^r$ given by the constraints in \eqref{WCong}. This is equivalent to solving the system of congruences: 
\begin{equation} 
\begin{pmatrix}
4&0&0&0&1\\
0&4&0&0&1\\
0&0&4&0&1\\
0&0&0&4&1\\
1&1&1&1&1
\end{pmatrix}
\begin{pmatrix} s_1 \\ s_2 \\ s_3 \\ s_4 \\ s_5 \end{pmatrix} \equiv 0 \pmod{\qq}. 
\end{equation}

If $q\equiv 3\pmod 4$, then by linear algebra over $\Z$ we obtain
\[\tyS = \left\{(1,1,1,1,-4)k_1 + \tfrac{\qq}{2}(0,0,1,-1,0) k_2 +  \tfrac{\qq}{2} (0,1,0,-1,0)k_3:k_i \in\Z/\qq\Z \right\}. \]
These solutions can be clustered in an analogous way as done in section~\ref{PF:Dwork}: 
\begin{enumerate}[(i)]
\item $\tyS_1 \colonequals \{ k(1,1,1,1,-4) : k \in \Z / \qq\Z\}$,
\item $\tyS_2 \colonequals \{ k(1,1,1,1,-4) + \tfrac{\qq}{2} (0,1,1,0,0) : k \in \Z/\qq \Z\}$,
\item $\tyS_3 \colonequals \{ k(1,1,1,1,-4) + \tfrac{\qq}{2} (0,1,0,1,0) : k \in \Z/\qq \Z\}$, and
\item $\tyS_4 \colonequals \{ k(1,1,1,1,-4) + \tfrac{\qq}{2} (0,0,1,1,0) : k \in \Z/\qq \Z\}$. 
\setcounter{DworkClusters}{\value{enumi}}
\end{enumerate}
The last three (ii)--(iv) all behave in the same way, due to the evident symmetry.

If instead $q\equiv1\pmod 4$, then
\[ \tyS = \left\{(1,1,1,1,-4)k_1 +  \tfrac{\qq}{4}(0,1,0,-1,0)k_2 +  \tfrac{\qq}{4} (0,0,1,-1,0)k_3 : k_i \in \Z/\qq\Z\right \}; \] 
we cluster again,  getting the four clusters above but now together with twelve new clusters:
\begin{enumerate}[(i)]
\setcounter{enumi}{\value{DworkClusters}}
\item three sets of the form $\tyS_5 \colonequals \{k(1,1,1,1,-4) + \tfrac{\qq}{4}(0,1,2,1,0) : k \in \Z/\qq\Z\}$,
\item three sets of the form $\tyS_6 \colonequals \{k(1,1,1,1,-4) - \tfrac{\qq}{4}(0,1,2,1,0) : k \in \Z/\qq\Z\}$, and
\item six sets of the form $\tyS_7 \colonequals \{k(1,1,1,1,-4) + \tfrac{\qq}{4}(0,0,1,3,0) : k \in \Z/\qq\Z\}$,
\end{enumerate}
where the number of sets is given by the number of distinct permutations of the middle three coordinates.

\subsubsection*{Step 2: Counting points on the open subset with nonzero coordinates}

We now give a formula for $\#U_{\Fsf_4, \psi}(\F_q)$ for the number of points, applying Theorem \ref{Dels}.  

We go through each cluster $\tyS_i$, linking each to a hypergeometric function.
\begin{lem}\label{Cluster no shifts}
For all odd $q$, 
\begin{equation}
 \sum_{s \in \tyS_1} \omega(a)^{-s} c_s=  q^2-3q+3  + H_q(\tfrac14, \tfrac12, \tfrac34; 0,0,0 \,|\, \psi^{-4}).
\end{equation}
\end{lem}

\begin{proof}
By Definition~\ref{BCM HGF}, 
\begin{align*}
H_q(\tfrac14, \tfrac12, \tfrac34; 0,0,0 \,|\, \psi^{-4}) &= \frac{1}{\qq} \sum_{m=0}^{q-2} q^{-s(0) + s(m)} g(4m) g(-m)^4 \omega(4\psi)^{-4m} \\ 
	&= \frac{-1}{\qq} + \frac{-q^{-1} g(\tfrac{\qq}{2})^4}{\qq} +  \frac{1}{\qq}\sum_{\substack{m=1 \\ m\neq \qq\qm/2}}^{q-2} q^{-1} g(4m) g(-m)^4 \omega(4\psi)^{-4m} \\ 
	&=-  \frac{1}{\qq} - \frac{g(\tfrac{\qq}{2})^4}{q\qq} + \frac{1}{q\qq}\sum_{\substack{k=1 \\ k\neq \qq\qm/2}}^{q-2} g(-4k) g(k)^4 \omega(4\psi)^{4k}
	\end{align*}
	the latter by substituting $k=-m$.
Now we expand to match terms:
\begin{align}\label{S1 for Dwork}
\sum_{s \in \tyS_1} \omega(a)^{-s} c_s &= c_{(0,0,0,0,0)} + c_{(\qq\qm/2)(1,1,1,1,0)} + \sum_{\substack{k=1 \\ k\neq \qq\qm/2}}^{q-2} \omega(-4\psi)^{4k} c_{(k,k,k,k,-4k)} \nonumber \\ 
&= \frac{({\qq})^4 - (-1)^4}{q\qq} - \frac{1}{q\qq} g(\tfrac{\qq}{2})^4 + \frac{1}{q\qq}\sum_{\substack{k=1 \\ k\neq \qq\qm/2}}^{q-2} \omega(4\psi)^{4k} g(k)^4 g(-4k) \\
	&= \frac{q^3-4q^2+6q-4}{\qq} - \frac{g(\tfrac{\qq}{2})^4}{q\qq} + \frac{1}{q\qq} \sum_{k=1, k\neq \tfrac{\qq}{2}}^{q-2} \omega(4\psi)^{4k} g(k)^4 g(-4k) \nonumber \\ 
	&=   q^2-3q+3  + H_q(\tfrac14, \tfrac12, \tfrac34; 0,0,0 \,|\, \psi^{-4}). \qedhere
\end{align}
\end{proof}

\begin{lem}\label{Cluster two half shifts}
For $i = 2, 3, 4$,
\begin{equation}
\sum_{s \in \tyS_i} \omega(a)^{-s} c_s = (-1)^{\qq\qm/2}\bigl(2 + qH_q(\tfrac14, \tfrac{3}{4}; 0, \tfrac12 \,|\, \psi^{-4})\bigr).
\end{equation}
\end{lem}
\begin{proof}
By definition,
\begin{equation}\begin{aligned} \label{eqn:Hq14}
H_q(\tfrac14, \tfrac{3}{4}; 0, \tfrac12 \,|\, \psi^{-4}) &= \frac{1}{\qq} \sum_{m=0}^{q-2} q^{-s(0) + s(m)}g(4m)g(-2m)^2 \omega(2\psi)^{-4m} \\ 
&= \frac{-2}{\qq} + \frac{1}{\qq} \sum_{\substack{m=1 \\ m \neq \qq\qm/2}}^{q-2} q^{-1}g(4m)g(-2m)^2 \omega(2\psi)^{-4m}
\end{aligned}\end{equation}
using $s(m) = 1$ if $m = 0,\tfrac{\qq}{2}$ and $s(m)=0$ otherwise. 

By symmetry, 
$$
\sum_{s \in \tyS_2} \omega(a)^{-s} c_s  = \sum_{s \in \tyS_3} \omega(a)^{-s} c_s  = \sum_{s \in \tyS_4} \omega(a)^{-s} c_s. 
$$
So we only need to consider $i=2$.  Then: 
\begin{equation} \label{S2 cluster Dwork 0} 
\begin{aligned} 
\sum_{s \in \tyS_2} \omega(a)^{-s} c_s &= c_{(\qq\qm/2)(0,1,1,0,0)} + c_{(\qq\qm/2)( 1,0,0,1,0)} + \sum_{\substack{k=1 \\ k\neq \qq\qm/2}}^{q-2} \omega(-4\psi)^{4k} c_{(k(1,1,1,1,-4) + (\qq\qm/2)(0,1,1,0,0))} \\
	&= -2\frac{g(\tfrac{\qq}{2})^2}{q\qq} + \frac{1}{q\qq}  \sum_{\substack{k=1 \\ k\neq \qq\qm/2}}^{q-2} \omega(-4\psi)^{4k}g(k)^2g(k+\tfrac{\qq}{2})^2g(-4k). 
	\end{aligned} 
	\end{equation}
Next, we use the Hasse--Davenport product relation (Lemma \ref{gausssumidentities}(c)) with $N=2 \mid \qq$ to get
\[ g(2k) = -\omega(2)^{2k} \frac{g(k)}{g(0)} \frac{g(k+\tfrac{\qq}{2})}{g(\frac{\qq}{2})} \]
which rearranges using $g(0)=-1$ to
\begin{equation} \label{eqn:gkkq2}
g(k)g(k+\tfrac{\qq}{2}) = \omega(2)^{-2k} g(2k) g(\tfrac{\qq}{2}). 
\end{equation}
Using Lemma \ref{gausssumidentities}(b) gives $g(\tfrac{\qq}{2})^2 = (-1)^{\qq\qm/2} q$; substituting this and \eqref{eqn:gkkq2} into \eqref{S2 cluster Dwork 0} simplifies to
\[ \begin{aligned} 
\sum_{s \in \tyS_2} \omega(a)^{-s} c_s &= -2\frac{(-1)^{\qq\qm/2}}{\qq} + \frac{1}{q\qq}  \sum_{\substack{k=1 \\ k\neq \qq\qm/2}}^{q-2} \omega(-4\psi)^{4k}(\omega(2)^{-2k}g(2k) g(\tfrac{\qq}{2}))^2 g(-4k) \\
	&= -2\frac{(-1)^{\qq\qm/2}}{\qq} + \frac{1}{\qq}  \sum_{\substack{k=1 \\ k\neq \qq\qm/2}}^{q-2} (-1)^{\qq\qm/2}\omega(-2\psi)^{4k}g(2k)^2  g(-4k) \\
	&= (-1)^{\qq\qm/2}\left(-\frac{2}{\qq} + \frac{1}{\qq}  \sum_{\substack{k=1 \\ k\neq \qq\qm/2}}^{q-2} \omega(-2\psi)^{4k}g(2k)^2  g(-4k)\right).
	\end{aligned} \]
Looking back at \eqref{eqn:Hq14}, we rearrange and insert a factor $q$ to find the hypergeometric sum:
\[ \begin{aligned} 
(-1)^{\qq\qm/2} \sum_{s \in \tyS_2} \omega(a)^{-s} c_s
	&= \frac{2q-2}{\qq} - \frac{2q}{\qq} + \frac{q}{\qq} \sum_{\substack{m=1 \\ m\neq \qq\qm/2}}^{q-2} q^{-1}\omega(2\psi)^{-4m}g(-2m)^2  g(4m) \\
	&= 2 + qH_q(\tfrac14, \tfrac{3}{4}; 0, \tfrac12 \,|\, \psi^{-4})
\end{aligned} \]
as claimed.
\end{proof}

\begin{lem}\label{Cluster 2 quarters 1 half shifts}
Suppose $q \equiv 1 \pmod 4$. Then 
$$
\sum_{s \in \tyS_5} \omega(a)^{-s} c_s = (-1)^{\qq\qm/4} q H_q(\tfrac12; 0 \,|\, \psi^{-4}) + (-1)^{\qq\qm/4} - \frac{g(\tfrac{\qq}{4})^2 + g(\tfrac{3\qq}{4})^2}{g(\tfrac{\qq}{2})}.
$$
\end{lem}
\begin{proof}
Plugging into the definition of the finite field hypergeometric sum and then pulling out terms $m=j\qq\qm/4$ with $j=0,1,2,3$, we get
\begin{equation} \label{eqn:Hq12}
\begin{aligned}
H_q(\tfrac12; 0 \,|\, \psi^{-4}) &= \frac{1}{\qq} \sum_{m=0}^{q-2} \omega(-\psi^{-4})^m \frac{g(m + \tfrac{\qq}{2})g(-m)}{g(\tfrac{\qq}{2})}  \\
	&= -\frac{2}{\qq} + \frac{(-1)^{\qq\qm/4}}{\qq g(\tfrac{\qq}{2})}\bigl(g(\tfrac{\qq}{4})^2+g(\tfrac{3\qq}{4})^2\bigr)  + \frac{1}{\qq} \sum_{\substack{m=0 \\ \qq \nmid 4m}}^{q-2} \omega(-\psi^{-4})^{m} \frac{g(m + \tfrac{\qq}{2})g(-m)}{g(\tfrac{\qq}{2})}.
\end{aligned} 
\end{equation}

Hasse--Davenport (Lemma \ref{gausssumidentities}(c)) implies
\begin{equation} \label{eqn:g4m01}
g(4m) = -\omega(4)^{4m} \frac{g(m) g(m + \tfrac{\qq}{4}) g(m + \tfrac{\qq}{2}) g(m + \tfrac{3\qq}{4}) } {g(0)  g(\tfrac{\qq}{4}) g( \tfrac{\qq}{2}) g(\tfrac{3\qq}{4}) }
\end{equation}
For $m \neq j\tfrac{\qq}{4}$, multiplying \eqref{eqn:g4m01} by $g(-m-\tfrac{3\qq}{4})g(-4m)$ and simplifying, we get:
\begin{equation} \label{gaussSumTrickery} 
\begin{aligned}
(-1)^{4m} q g(-m-\tfrac{3\qq}{4}) &= \omega(4)^{4m}(-1)^{m} \frac{g(m) g(m + \tfrac{\qq}{4}) g(m + \tfrac{\qq}{2}) g(-4m)} { g( \tfrac{\qq}{2})  } \\
g(m) g(m + \tfrac{\qq}{4}) g(m + \tfrac{\qq}{2}) g(-4m) &= (-1)^{-m} \omega(4)^{-4m} q g(\tfrac{\qq}{2})g(-m-\tfrac{3\qq}{4}).
\end{aligned}
\end{equation}

Now we look at the point count.  First, we take the definition:
\begin{equation} \label{eqn:s5start}
\begin{aligned}
\sum_{s \in \tyS_5} \omega(a)^{-s} c_s &= \frac{1}{q\qq} \sum_{k=0}^{q-2} \omega(-4\psi)^{4k} g(k) g(k+\tfrac{\qq}{4})^2 g(k + \tfrac{\qq}{2}) g(-4k).
\end{aligned} 
\end{equation}
We then tease out the four terms with $k=j\qq\qm/4$.  The cases $k=0,\frac{\qq}{2}$ give
\begin{equation}  \label{eqn:k012}
\frac{1}{q\qq} g(\tfrac{\qq}{4})^2 g(\tfrac{\qq}{2}) + \frac{1}{q\qq} g(\tfrac{\qq}{2}) g(\tfrac{3\qq}{4})^2 = \frac{g(\tfrac{\qq}{4})^2+g(\tfrac{3\qq}{4})^2}{\qq g(\tfrac{\qq}{2})} 
\end{equation}
because $g(\tfrac{\qq}{2})^2=q$ as $q \equiv 1 \pmod{4}$.  The terms with $k=\tfrac{\qq}{4},\tfrac{3\qq}{4}$ are
\begin{equation}  \label{eqn:k1434}
-\frac{1}{q\qq} g(\tfrac{\qq}{4}) g(\tfrac{\qq}{2})^2 g(\tfrac{3\qq}{4}) -\frac{1}{q\qq} g(\tfrac{3\qq}{4}) g( \tfrac{\qq}{4}) = -(-1)^{\qq\qm/4}\frac{q+1}{\qq} = (-1)^{\qq\qm/4}\left(1-\frac{2q}{\qq}\right).
\end{equation}
using Lemma \ref{gausssumidentities}(b) with $m=\tfrac{\qq}{4}$ to get $g(\tfrac{\qq}{4})g(\tfrac{3\qq}{4})=(-1)^{\qq\qm/4} q$.  

For the remaining terms in the sum, we plug in \eqref{gaussSumTrickery} to get
\begin{equation}
\begin{aligned} 
&\frac{1}{q\qq} \sum_{\substack{k=0 \\ \qq \nmid 4k} }^{q-2} \omega(-4\psi)^{4k} g(k) g(k+\tfrac{\qq}{4})^2 g(k + \tfrac{\qq}{2}) g(-4k) \\
&\qquad = 
\frac{1}{q\qq} \sum_{\substack{k=0 \\ \qq \nmid 4k} }^{q-2} \omega(-4\psi)^{4k} g(k+\tfrac{\qq}{4})(-1)^{-k} \omega(4)^{-4k} q g(\tfrac{\qq}{2})g(-k-\tfrac{3\qq}{4}) \\
&\qquad =\frac{q}{\qq} \sum_{\substack{k=0 \\ \qq \nmid 4k} }^{q-2} \omega(-\psi^4)^{k} \frac{ g(k+\tfrac{\qq}{4})g(-k-\tfrac{3\qq}{4})}{g(\tfrac{\qq}{2})}.
\end{aligned}
\end{equation}
Next, we reindex this summation with the substitution $m = -k - \tfrac{\qq}{4}$ to obtain
\begin{equation} \label{eqn:kmrest}
\frac{q}{\qq} \sum_{\substack{m=0 \\ \qq \nmid 4m} }^{q-2} \omega(-\psi^4)^{-m-\qq\qm/4}  \frac{g(-m) g(m+ \tfrac{\qq}{2})}{g(\tfrac{\qq}{2})} = (-1)^{\qq\qm/4}\frac{q}{\qq} \sum_{\substack{m=0 \\ \qq \nmid 4m} }^{q-2} \omega(-\psi^{-4})^{m} \frac{g(m+\tfrac{\qq}{2}) g(-m)}{g(\tfrac{\qq}{2})}.
\end{equation}

Taking \eqref{eqn:s5start}, expanding and substituting \eqref{eqn:k012}, \eqref{eqn:k1434}, and \eqref{eqn:kmrest} then gives
\begin{align*} 
&(-1)^{\qq\qm/4}\sum_{s \in \tyS_5} \omega(a)^{-s} c_s \\
&\qquad = (-1)^{\qq\qm/4}\frac{g(\tfrac{\qq}{4})^2+g(\tfrac{3\qq}{4})^2}{\qq g(\tfrac{\qq}{2})} 
+1-\frac{2q}{\qq} + \frac{q}{\qq} \sum_{\substack{m=0 \\ \qq \nmid 4m} }^{q-2} \omega(-\psi^{-4})^{m} \frac{g(m+\tfrac{\qq}{2}) g(-m)}{g(\tfrac{\qq}{2})}.
\end{align*}
We are quite close to \eqref{eqn:Hq12}, but the first term is off by a factor $q$.  Adding and subtracting gives
\[ (-1)^{\qq\qm/4}\sum_{s \in \tyS_5} \omega(a)^{-s} c_s = 1+qH_q(\tfrac{1}{2};0 \,|\,\psi^{-4}) - (-1)^{\qq\qm/4}\frac{g(\tfrac{\qq}{4})^2+g(\tfrac{3\qq}{4})^2}{g(\tfrac{\qq}{2})} \]
as claimed.
\end{proof}

\begin{lem}\label{Cluster 2 3quarter 1 half shifts}
If $q \equiv 1 \pmod 4$, then 
\[ \sum_{s \in \tyS_5} \omega(a)^{-s} c_s = \sum_{s \in \tyS_6} \omega(a)^{-s} c_s =  \sum_{s \in \tyS_7} \omega(a)^{-s} c_s. \]
\end{lem}
\begin{proof}
We start with \eqref{eqn:s5start} and reindex with $m = k+\tfrac{\qq}{2}$:
\begin{align*}
\sum_{s \in \tyS_5} \omega(a)^{-s} c_s &= \frac{1}{q\qq} \sum_{k=0}^{q-2} \omega(-4\psi)^{4k} g(k) g(k+\tfrac{\qq}{4})^2 g(k + \tfrac{\qq}{2}) g(-4k) \\
&= \frac{1}{q\qq} \sum_{m=0}^{q-2} \omega(-4\psi)^{4m} g(m + \tfrac{\qq}{2}) g(m+\tfrac{3\qq}{4})^2 g(m) g(-4m)  \\
&= \sum_{s \in \tyS_6} \omega(a)^{-s} c_s. 
\end{align*} 
The equality for $\tyS_7$ holds reindexing with $m=k+\tfrac{\qq}{4}$.  
\end{proof}

We now put these pieces together to give the point count for the toric hypersurface.

\begin{prop}\label{OpenF4}Let $\psi \in \F_q^\times$.
\begin{enumalph}
\item If $q \equiv 3 \pmod 4$, then 
\begin{equation}
\#U_{\Fsf_4, \psi}(\F_q) = q^2 -3q -3  + H_q(\tfrac14, \tfrac12, \tfrac34; 0,0,0 \,|\, \psi^{-4}) - 3 qH_q(\tfrac14, \tfrac{3}{4}; 0, \tfrac12 \,|\, \psi^{-4})).
\end{equation}
\item If $q \equiv 1 \pmod 4$, then
\begin{equation}\begin{aligned}
\#U_{\Fsf_4, \psi}(\F_q) &= q^2-3q+9 + H_q(\tfrac14, \tfrac12, \tfrac34; 0,0,0 \,|\, \psi^{-4}) + 3qH_q(\tfrac14, \tfrac{3}{4}; 0, \tfrac12 \,|\, \psi^{-4})) \\
	 &\qquad+ 12\left((-1)^{\qq\qm/4}q H_q(\tfrac12; 0 \,|\, \psi^{-4}) + (-1)^{\qq\qm/4} - \frac{g(\tfrac{\qq}{4})^2 + g(\tfrac{3\qq}{4})^2}{g(\tfrac{\qq}{2})}\right).
\end{aligned}\end{equation}
\end{enumalph}
\end{prop}

\begin{proof}
For $q\equiv 3\pmod 4$, we have from Lemmas~\ref{Cluster no shifts} and~\ref{Cluster two half shifts}:
\begin{equation}\begin{aligned}
\#U_{\Fsf_4, \psi}(\F_q) &= \sum_{i=1}^4 \sum_{s \in \tyS_i} \omega(a)^{-s} c_s \\
	&=  q^2-3q+3  + H_q(\tfrac14, \tfrac12, \tfrac34; 0,0,0 \,|\, \psi^{-4}) + 3(-2 - qH_q(\tfrac14, \tfrac{3}{4}; 0, \tfrac12 \,|\, \psi^{-4})) \\
	&=q^2 -3q -3  + H_q(\tfrac14, \tfrac12, \tfrac34; 0,0,0 \,|\, \psi^{-4}) - 3 qH_q(\tfrac14, \tfrac{3}{4}; 0, \tfrac12 \,|\, \psi^{-4}))
\end{aligned}\end{equation}
For $q\equiv 1 \pmod 4$, we have from Lemmas~\ref{Cluster no shifts},~\ref{Cluster two half shifts},~\ref{Cluster 2 quarters 1 half shifts}, and~\ref{Cluster 2 3quarter 1 half shifts}, we have that:
\begin{equation}\begin{aligned}
\#U_{\Fsf_4, \psi}(\F_q) &= \sum_{i=1}^4 \sum_{s \in \tyS_i} \omega(a)^{-s} c_s + 12\sum_{s \in \tyS_5} \omega(a)^{-s} c_s \\
	&= q^2-3q+3 + H_q(\tfrac14, \tfrac12, \tfrac34; 0,0,0 \,|\, \psi^{-4}) + 3(2 + qH_q(\tfrac14, \tfrac{3}{4}; 0, \tfrac12 \,|\, \psi^{-4})) \\
	&\qquad+ 12\left((-1)^{\qq\qm/4} q H_q(\tfrac12; 0 \,|\, \psi^{-4}) + (-1)^{\qq\qm/4} - \frac{g(\tfrac{\qq}{4})^2 + g(\tfrac{3\qq}{4})^2}{g(\tfrac{\qq}{2})}\right)
	\end{aligned}
	\end{equation}
which simplifies to the result.  
\end{proof}

\subsubsection*{Step 3: Count points when at least one coordinate is zero}

\begin{lem}\label{oneCoordZero q 3mod4 Dwork}
If $q \equiv 3 \pmod 4$, then 
$$
\#X_{\Fsf_4,\psi}(\F_q) - \#U_{\Fsf_4,\psi}(\F_q) = 4q+4.
$$
\end{lem}
\begin{proof}
First we compute the number of points when $x_3=0$ and $x_0x_1x_2 \neq 0$, i.e., count points on the Fermat quartic curve $V\colon x_0^4+x_1^4+x_2^4=0$ with coordinates in the torus.  All points in $V(\F_q)$ lie on the torus: if e.g.\ $x_0=0$ and $x_1 \neq 0$, then $-1=(x_2/x_1)^4$, but $-1 \not\in\F_q^{\times 2}$ since $q \equiv 3 \pmod{4}$.  We claim that $\#V(\F_q)=q+1$; this can be proven in many ways.  First, we sketch an elementary argument, working affinely on $x_0^4+x_1^4=-1$.  The map $(x_0,x_1) \mapsto (x_0^2,x_1^2)$ gives a map to the affine curve defined by $C\colon w_0^2+w_1^2=-1$.  The number of points on this curve over $\F_q$ is $q+1$ (the projective closure is a smooth conic with no points at infinity), and again all such solutions have $w_0,w_1 \in \F_{q}^\times$.  Since $q \equiv 3 \pmod{4}$, the squaring map $\F_q^{\times 2} \to \F_q^{\times 4}$ is bijective.  Therefore, for the four points $(\pm w_0,\pm w_1)$ with $w_0^2+w_1^2=-1$, there are exactly four points $(\pm x_0, \pm x_1)$ with $x_i^4=w_i^2$ for $i=0,1$.  Thus $\#V(\F_q)=\#C(\F_q)=q+1$.  (Alternatively, the map $(x_0,x_1) \mapsto (x_0^2,x_1)$ is bijective, with image a supersingular genus $1$ curve over $\F_q$.)

Second, and for consistency, we again apply the formula of Koblitz!  For the characters, we solve
\begin{equation} \label{3 variable Fermat}
\begin{pmatrix}
4&0&0\\
0&4&0\\
0&0&4\\
1&1&1
\end{pmatrix}
\begin{pmatrix} s_1 \\ s_2 \\ s_3 \end{pmatrix} \equiv 0 \pmod{\qq}. 
\end{equation}
There are exactly four solutions when $q \equiv 3 \pmod 4$:
$$
\tyS = \{ (0,0,0), \tfrac{\qq}{2} (1,1,0), \tfrac{\qq}{2} (1,0,1), \tfrac{\qq}{2} (0,1,1)\}.
$$
Then by Theorem \ref{Dels},
\begin{equation}\begin{aligned}
\#V(\F_q) 
	&= c_{(0,0,0)} + c_{(\qq\qm/2)(1,1,0)} + c_{(\qq\qm/2)(1,0,1)}+ c_{(\qq\qm/2)(0,1,1)}  \\
	&= \frac{(q-1)^{2} - (-1)^2}{q} + 3(-1)^{2} \frac{1}{q} g(\tfrac{\qq}{2})^2 g(\tfrac{\qq}{2}) g(0) \\
	&= \frac{q^2 - 2q}{q} + 3 = q+1. 
\end{aligned}\end{equation}

By symmetry, repeating in each of the four coordinate hyperplanes, we obtain $\#X_{\Fsf_4,\psi}(\F_q) - \#U_{\Fsf_4,\psi}(\F_q) = 4(q+1) = 4q + 4$.
\end{proof}

\begin{lem}\label{oneCoordZero q 1mod4 Dwork}
If $q\equiv 1\pmod 4$, then 
$$
\#X_{\Fsf_4, \psi}(\F_q) - \#U_{\Fsf_4, \psi}(\F_q) = 4q - 8 -12(-1)^{\qq\qm/4} + 12\frac{g(\tfrac{\qq}{4})^2+g(\tfrac{3\qq}{4})^2}{g(\tfrac{\qq}{2})}.
$$
\end{lem}

\begin{proof}
We repeat the argument in the preceding lemma.  We cluster solutions to \eqref{3 variable Fermat} and count the number of solutions in the following way:
\begin{equation}\begin{aligned}
\#V(\F_q) &= c_{(0,0,0)}  + 6c_{(\qq\qm/4)(1,3,0)} + 3c_{(\qq\qm/4)(2,2,0)} + 3c_{(\qq\qm/4)(1,1,2)} + 3c_{(\qq\qm/4)(3,3,2)}  \\
	&= q-2 - 6 (-1)^{\qq\qm/4} - 3 + \frac{3}{q} g(\tfrac{\qq}{4})^2 g(\tfrac{\qq}{2}) + \frac{3}{q} g(\tfrac{3\qq}{4})^2g(\tfrac{\qq}{2}) \\ 
	&= q-5 - 6 (-1)^{\qq\qm/4} + 3\frac{g(\tfrac{\qq}{4})^2+g(\tfrac{3\qq}{4})^2}{g(\tfrac{\qq}{2})}.
\end{aligned}\end{equation}
There are $2$ solutions to $x_0^4+x_1^4=0$ with $x_0x_1 \neq 0$ if $q \equiv 1 \pmod{8}$ and zero otherwise, so $2+2(-1)^{\qq\qm/4}$ solutions in either case.  Adding up, we get
\begin{align*}
\#X_{\Fsf_4, \psi}(\F_q) - \#U_{\Fsf_4, \psi}(\F_q) &= 4 \#V(\F_q) + 6(2+2(-1)^{\qq\qm/4})  \\ 
	&= 4q - 8 -12(-1)^{\qq\qm/4} + 12 \frac{1}{q} g(\tfrac{\qq}{4})^2 g(\tfrac{\qq}{2}) + 12\frac{g(\tfrac{\qq}{4})^2+g(\tfrac{3\qq}{4})^2}{g(\tfrac{\qq}{2})}. \qedhere
\end{align*}
\end{proof}

\subsubsection*{Step 4: Conclude}

We now conclude the proof.

\begin{proof}[Proof of Proposition~\textup{\ref{prop:F4}}]
We combine Proposition~\ref{OpenF4} with Lemmas~\ref{oneCoordZero q 3mod4 Dwork} and~\ref{oneCoordZero q 1mod4 Dwork}. If $q\equiv3 \pmod4$, then
\begin{align*}
\#X_{\Fsf_4, \psi}(\F_q) &= \#U_{\Fsf_4, \psi}(\F_q) + (\#X_{\Fsf_4, \psi}(\F_q) - \#U_{\Fsf_4, \psi}(\F_q))  \\
	&= (q^2 -3q -3  + H_q(\tfrac14, \tfrac12, \tfrac34; 0,0,0 \,|\, \psi^{-4}) - 3 qH_q(\tfrac14, \tfrac{3}{4}; 0, \tfrac12 \,|\, \psi^{-4})) + (4q+4)\\
	&= q^2 + q + 1  + H_q(\tfrac14, \tfrac12, \tfrac34; 0,0,0 \,|\, \psi^{-4}) - 3 qH_q(\tfrac14, \tfrac{3}{4}; 0, \tfrac12 \,|\, \psi^{-4}).
\end{align*}
If $q\equiv 1\pmod4$, then the ugly terms cancel, and we have simply
\begin{align*}
\#X_{\Fsf_4, \psi}(\F_q) &= \#U_{\Fsf_4, \psi}(\F_q) + (\#X_{\Fsf_4, \psi}(\F_q) - \#U_{\Fsf_4, \psi}(\F_q)) \\
&= (q^2-3q+9 + H_q(\tfrac14, \tfrac12, \tfrac34; 0,0,0 \,|\, \psi^{-4}) + 3qH_q(\tfrac14, \tfrac{3}{4}; 0, \tfrac12 \,|\, \psi^{-4}) \\
&\qquad + 12(-1)^{\qq\qm/4}q H_q(\tfrac12; 0 \,|\, \psi^{-4})) + (4q-8) \\
	 &= q^2 + q + 1 + H_q(\tfrac14, \tfrac12, \tfrac34; 0,0,0 \,|\, \psi^{-4}) + 3qH_q(\tfrac14, \tfrac{3}{4}; 0, \tfrac12 \,|\, \psi^{-4}) \\
	 &\qquad + 12(-1)^{\qq\qm/4} q H_q(\tfrac12; 0 \,|\, \psi^{-4}).  \qedhere
\end{align*}
\end{proof}

\subsection{The Klein--Mukai pencil \texorpdfstring{$\Fsf_1\Lsf_3$}{Fsf1Lsf3}}

In this section, we repeat the steps of the previous section but for the Klein--Mukai pencil $\Fsf_1\Lsf_3$.  We suppose throughout this section that $q$ is coprime to $14$.  Our main result is as follows.

\begin{prop}\label{prop:F1L3}
For $q$ coprime to $14$ and $\psi \in \F_q^\times$, the following statements hold. 
\begin{enumalph}
\item If $q\not\equiv1\pmod7$, then
\[ \#X_{\Fsf_1\Lsf_3, \psi}(\F_q)=q^2+q+1+H_q(\tfrac{1}{4},\tfrac{1}{2},\tfrac{3}{4};0,0,0\,|\,\psi^{-4}).\]

\item If  $q\equiv1\pmod7$, then
\begin{align*} \#X_{\Fsf_1\Lsf_3, \psi}(\F_q)&=q^2+q+1+H_q(\tfrac{1}{4},\tfrac{1}{2},\tfrac{3}{4};0,0,0\,|\,\psi^{-4}) \\&\qquad+3qH_q(\tfrac{1}{14},\tfrac{9}{14},\tfrac{11}{14};0,\tfrac{1}{4},\tfrac{3}{4}\,|\,\psi^4)+3qH_q(\tfrac{3}{14},\tfrac{5}{14},\tfrac{13}{14};0,\tfrac{1}{4},\tfrac{3}{4}\,|\,\psi^4). 
\end{align*}
\end{enumalph}
\end{prop}

\begin{rmk}
The new parameters $\tfrac{1}{14},\tfrac{9}{14},\tfrac{11}{14};0,\tfrac{1}{4},\tfrac{3}{4}$ and $\tfrac{3}{14},\tfrac{5}{14},\tfrac{13}{14};0,\tfrac{1}{4},\tfrac{3}{4}$ match the Picard--Fuchs equations in Proposition~\ref{prop:kleinmukaidiff} as elements of $\Q/\Z$, with the same multiplicity.  
\end{rmk}

\subsubsection*{Step 1: Computing and clustering the characters}

As before, we first have to compute the solutions to the system of congruences:
\[ \begin{pmatrix}
3&0&1&0&1\\
1&3&0&0&1\\
0&1&3&0&1\\
0&0&0&4&1\\
1&1&1&1&1\\
\end{pmatrix}\begin{pmatrix} w_1\\w_2\\w_3\\w_4\\w_5 \end{pmatrix} \equiv \begin{pmatrix} 0\\0\\0\\0\\0\end{pmatrix} \pmod{\qq}. \]
By linear algebra over $\Z$, we compute that if $q\not\equiv1\pmod 7$, then the set of solutions is 
\[\tyS=\{(1,1,1,1,-4)w: w\in\Z/\qq\Z\}. \]
On the other hand if $q \equiv 1 \pmod{7}$, then the set splits into three classes:
\begin{enumerate}[(i)]
\item the set $\tyS_1 = \{k(1,1,1,1,-4): k\in\Z/\qq\Z\}$,
\item three sets of the form $\tyS_8 = \left\{k(1,1,1,1,-4) + \tfrac{\qq}{7}(1,4,2,0,0): k\in\Z/\qq\Z\right\}$, and
\item three sets of the form $\tyS_9 = \left\{k(1,1,1,1,-4) + \tfrac{\qq}{7}(3,5,6,0,0): k\in\Z/\qq\Z\right\}$.
\end{enumerate}
The multiplicity of the latter two sets corresponds to cyclic permutations yielding the same product of Gauss sums.

\subsubsection*{Step 2: Counting points on the open subset with nonzero coordinates}

As in the previous section, the hard work is in counting points in the toric hypersurface.  We now proceed with each cluster.

\begin{lem}\label{q not 1 mod 4 shift 1 fourteenth}
If $q \equiv 1 \pmod 7$, then
\begin{equation}
\sum_{s \in \tyS_8} \omega(a)^{-s} c_s = qH_q(\tfrac{1}{14}, \tfrac{9}{14}, \tfrac{11}{14}; 0, \tfrac{1}{4}, \tfrac{3}{4} \,|\, \psi^4) - \frac{1}{q} g(\tfrac{\qq}{7})g(\tfrac{4\qq}{7}) g(\tfrac{2\qq}{7}).
\end{equation}
\end{lem}

\begin{proof}
Recall our hybrid hypergeometric sum \eqref{eqn:hgmtherealdeal} from Example~\ref{HypergeometricFunction for first F1L3 mod 7 example}, plugging in $t=\psi^4$:
\begin{equation} \label{eqn:h1414}
\begin{aligned}
&H_q(\tfrac{1}{14}, \tfrac{9}{14}, \tfrac{11}{14}; 0,\tfrac14,\tfrac34 \,|\, \psi^4) \\
&\qquad= \frac{1}{\qq} -\frac{1}{q\qq} g(\tfrac{\qq}{7})g(\tfrac{2\qq}{7})g(\tfrac{4\qq}{7}) \\
&\qquad\qquad +\frac{1}{q\qq} \sum_{\substack{m=1 \\ m \neq \qq\qm/2}}^{q-2} \frac{g(m+ \frac{1}{14}\qq)g(m+ \frac{9}{14}\qq)g(m+ \frac{11}{14}\qq)}{g(\frac{1}{14} \qq)g(\frac{9}{14} \qq)g(\frac{11}{14} \qq)} g(2m) g(-m) g(-4m) \omega(-4^3\psi^4)^m.
\end{aligned} 
\end{equation}
Our point count formula expands to
\begin{align*}
\sum_{s \in \tyS_8} \omega(a)^{-s} c_s 
	&= \frac{1}{q\qq} g(\tfrac{\qq}{7})g(\tfrac{4\qq}{7}) g(\tfrac{2\qq}{7}) - \frac{1}{q\qq} g(\tfrac{9\qq}{14})g(\tfrac{\qq}{14}) g(\tfrac{11\qq}{14})g(\tfrac{\qq}{2})  \\
	&\qquad +\frac{1}{q\qq}  \sum_{\substack{k=0 \\ \qq \nmid 2k}}^{q-2} \omega(-4\psi)^{4k} g(k+\tfrac{\qq}{7}) g(k+\tfrac{4\qq}{7})g(k + \tfrac{2\qq}{7}) g(k)g(-4k).
\end{align*}
We work on the sum.  Changing indices to $m=k+\tfrac{\qq}{2}$, using the identity 
$$
g(m+\tfrac{\qq}{2}) = \omega(4)^{-m} (-1)^mq^{-1} g(\tfrac{\qq}{2})g(-m)g(2m)
$$
found by using Hasse--Davenport for $N=2$, and applying Lemma~\ref{Gauss Identity Lemma F1L3}, gives us the summand
\begin{align*}
&\omega(-4\psi)^{4m} g(m+\tfrac{9\qq}{14}) g(m+\tfrac{\qq}{14})g(m + \tfrac{11\qq}{14}) g(m + \tfrac{\qq}{2})g(-4m) \\
&\qquad =\omega(-4\psi)^{4m} g(m+\tfrac{\qq}{14}) g(m+\tfrac{9\qq}{14})g(m + \tfrac{11\qq}{14})  \omega(4)^{-m} (-1)^mq^{-1} g(\tfrac{\qq}{2})g(-m)g(2m)g(-4m) \\
&\qquad =\omega(-4^3\psi^4)^{m}\frac{g(m+\tfrac{\qq}{14}) g(m+\tfrac{9\qq}{14})g(m + \tfrac{11\qq}{14})}{g(\tfrac{\qq}{14}) g(\tfrac{9\qq}{14})g(\tfrac{11\qq}{14})}q^{-1} g(\tfrac{\qq}{2})^4 g(2m)g(-m)g(-4m) \\
&\qquad =q\omega(-4^3\psi^4)^{m}\frac{g(m+\tfrac{\qq}{14}) g(m+\tfrac{9\qq}{14})g(m + \tfrac{11\qq}{14})}{q g(\tfrac{\qq}{14}) g(\tfrac{9\qq}{14})g(\tfrac{11\qq}{14})}g(2m)g(-m)g(-4m).
\end{align*}
Plugging back in, we can relate this to the hypergeometric function \eqref{eqn:h1414}:
\begin{align*}
\sum_{s \in \tyS_8} \omega(a)^{-s} c_s 
	&= \frac{1}{q\qq} g(\tfrac{\qq}{7})g(\tfrac{2\qq}{7}) g(\tfrac{4\qq}{7}) - \frac{1}{q\qq} g(\tfrac{\qq}{14})g(\tfrac{9\qq}{14}) g(\tfrac{11\qq}{14})g(\tfrac{\qq}{2})  \\
	&\qquad + \frac{1}{q\qq} \sum_{\substack{m=0 \\ \qq \nmid 2m}}^{q-2} q\omega(-4^3\psi^4)^{m}\frac{g(m+\tfrac{\qq}{14}) g(m+\tfrac{9\qq}{14})g(m + \tfrac{11\qq}{14})}{g(\tfrac{\qq}{14}) g(\tfrac{9\qq}{14})g(\tfrac{11\qq}{14})}g(2m)g(-m)g(-4m) \\
	&= \frac{1}{q\qq} g(\tfrac{\qq}{7})g(\tfrac{2\qq}{7}) g(\tfrac{4\qq}{7}) - \frac{1}{q\qq} g(\tfrac{\qq}{14})g(\tfrac{9\qq}{14}) g(\tfrac{11\qq}{14})g(\tfrac{\qq}{2})  \\
	&\qquad + qH_q(\tfrac{1}{14}, \tfrac{9}{14}, \tfrac{11}{14}; 0, \tfrac{1}{4}, \tfrac{3}{4} \,|\, \psi^4) + \frac{q}{\qq}  - \frac{1}{\qq}g(\tfrac{\qq}{7}) g(\tfrac{2\qq}{7}) g(\tfrac{4\qq}{7}) \\
	&= qH_q(\tfrac{1}{14}, \tfrac{9}{14}, \tfrac{11}{14}; 0, \tfrac{1}{4}, \tfrac{3}{4} \,|\, \psi^4) - \frac{1}{q} g(\tfrac{\qq}{7})g(\tfrac{2\qq}{7}) g(\tfrac{4\qq}{7}). \qedhere
\end{align*}
\end{proof}

\begin{lem}\label{shift by 3 fourteenths}
If $q \equiv 1 \pmod 7$ then
\begin{equation}
\sum_{s \in \tyS_9} \omega(a)^{-s} c_s = qH_q(\tfrac{3}{14}, \tfrac{5}{14}, \tfrac{13}{14}; 0, \tfrac{1}{4}, \tfrac{3}{4} \,|\, \psi^4) - \frac{1}{q} g(\tfrac{3\qq}{7})g(\tfrac{5\qq}{7}) g(\tfrac{6\qq}{7}).
\end{equation}
\end{lem}
\begin{proof}
Apply complex conjugation to Lemma \ref{q not 1 mod 4 shift 1 fourteenth}; the effect is to negate indices, as in Proposition \ref{prop:qalbet}.
\end{proof}

We now put the pieces together to prove the main result in this step.

\begin{prop}\label{OpenF1L3}
Suppose $\psi\in \F_q^{\times}$. 
\begin{enumalph}
\item If $q \not\equiv 1 \pmod 7$ then 
\begin{equation}
\#U_{\Fsf_1\Lsf_3, \psi}(\F_q) =  q^2-3q+3  + H_q(\tfrac14, \tfrac12, \tfrac34; 0,0,0 \,|\, \psi^{-4}).
\end{equation}
\item If $q\equiv 1 \pmod 7$ then 
\begin{equation}\begin{aligned}
\#U_{\Fsf_1\Lsf_3, \psi}(\F_q) &= q^2-3q+3  + H_q(\tfrac14, \tfrac12, \tfrac34; 0,0,0 \,|\, \psi^{-4}) \\
	&\qquad + 3qH_q(\tfrac{1}{14}, \tfrac{9}{14}, \tfrac{11}{14}; 0, \tfrac{1}{4}, \tfrac{3}{4} \,|\, \psi^4) + 3qH_q(\tfrac{3}{14}, \tfrac{5}{14}, \tfrac{13}{14}; 0, \tfrac{1}{4}, \tfrac{3}{4} \,|\, \psi^4) \\
	&\qquad  -\frac{3}{q}(g(\tfrac{\qq}{7})g(\tfrac{2\qq}{7})g(\tfrac{4\qq}{7})  + g(\tfrac{3\qq}{7})g(\tfrac{5\qq}{7}) g(\tfrac{6\qq}{7})).
	\end{aligned}\end{equation}
\end{enumalph}
\end{prop}

\begin{proof}
When $q \not\equiv 1 \pmod 7$, there is only one cluster of characters, $\tyS_1$. By Lemma~\ref{Cluster no shifts}, we know that 
\begin{equation}
\#U_{\Fsf_1\Lsf_3, \psi}(\F_q) =  \sum_{s \in \tyS_1} \omega(a)^{-s} c_s=  q^2-3q+3  + H_q(\tfrac14, \tfrac12, \tfrac34; 0,0,0 \,|\, \psi^{-4}).
\end{equation}
When $q\equiv 1 \pmod 7$ we have three clusters of characters, the latter two ($S_8$ and $S_9$) with multiplicity $3$. By Lemmas~\ref{Cluster no shifts},~\ref{q not 1 mod 4 shift 1 fourteenth}, and~\ref{shift by 3 fourteenths}, these sum to the result.  
\end{proof}

\subsubsection*{Step 3: Count points when at least one coordinate is zero.}

Recall that $q$ is coprime to $14$. 

\begin{lem}\label{oneCoordZero F1L3}
Let $\psi \in \F_q^\times$.  
\begin{enumalph}
\item If $q\not\equiv 1 \pmod 7$, then
$$
\#X_{\Fsf_1\Lsf_3, \psi}(\F_q) - \#U_{\Fsf_1\Lsf_3, \psi}(\F_q) = 4q-2.
$$
\item If $q\equiv 1 \pmod 7$, then 
$$
\#X_{\Fsf_1\Lsf_3, \psi}(\F_q) - \#U_{\Fsf_1\Lsf_3, \psi}(\F_q) = 4q - 2 + \frac{3}{q}(g(\tfrac{\qq}{7})g(\tfrac{2\qq}{7})g(\tfrac{4\qq}{7}) + g(\tfrac{3\qq}{7}g(\tfrac{5\qq}{7}) g(\tfrac{6\qq}{7})).
$$
\end{enumalph}
\end{lem}
\begin{proof}
We count solutions with at least one coordinate zero.  If $x_3=0$ but $x_0x_1x_2 \neq 0$, we count points on $x_0^4+x_1^3x_2=0$: solving for $x_2$, we see there are $q-1$ solutions; repeating this for the cases $x_1=0$ or $x_2=0$, we get $3q-3$ points.  

Now suppose $x_0=0$ but $x_1x_2x_3 \neq 0$, we look at the equation $x_1^3x_2+x_2^3x_3+x_3^3x_1=0$ defining the Klein quartic.  Applying Theorem \ref{Dels} again, we find that 
\begin{equation}
\begin{pmatrix}
3&1&0\\
0&3&1\\
1&0&3\\
1&1&1
\end{pmatrix}
\begin{pmatrix} s_1 \\ s_2 \\ s_3 \end{pmatrix} \equiv 0 \pmod{\qq}. 
\end{equation}
If $q\not\equiv1\pmod7$, then only $(0,0,0)$ is a solution and $c_{(0,0,0)} = q-2$. If $q \equiv 1 \pmod 7$ then the solutions are 
$\left\{k(\tfrac{\qq}{7},\tfrac{4\qq}{7},\tfrac{2\qq}{7}) : k\in\Z/7\Z\right\}$ which gives the point count 
$$
q-2 +  \frac{3}{q}g(\tfrac{\qq}{7})g(\tfrac{2\qq}{7})g(\tfrac{4\qq}{7}) + \frac{3}{q}g(\tfrac{3\qq}{7})g(\tfrac{5\qq}{7}) g(\tfrac{6\qq}{7}).
$$

If now at least two of the variables among $\{x_1,x_2,x_3\}$ are zero, then the equation is just $x_0^4 = 0$ hence the last one is also zero and there is only one such point.  If $x_0=x_1=0$, then the equation is $x_2^3 x_3=0$ hence another of the first three variables is zero. Consequently there are exactly 3 such points.  Totalling up gives the result.  
\end{proof}

\subsubsection*{Step 4: Conclude}

We now prove Proposition~\ref{prop:F1L3}.

\begin{proof}[Proof of Proposition~\textup{\ref{prop:F1L3}}]
By Proposition~\ref{OpenF1L3} and Lemma~\ref{oneCoordZero F1L3}, if $q\not\equiv 1\pmod 7$, then 
\begin{equation}\begin{aligned}
\#X_{\Fsf_1\Lsf_3, \psi}(\F_q) &=  q^2-3q+3  + H_q(\tfrac14, \tfrac12, \tfrac34; 0,0,0 \,|\, \psi^{-4}) + (4q-2) \\
	&= q^2 + q + 1+ H_q(\tfrac14, \tfrac12, \tfrac34; 0,0,0 \,|\, \psi^{-4}).
\end{aligned}\end{equation}
If $q\equiv 1 \pmod 7$, then the ugly terms cancel and we get
\begin{equation}\begin{aligned}
\#X_{\Fsf_1\Lsf_3, \psi}(\F_q) &= q^2-3q+3  + H_q(\tfrac14, \tfrac12, \tfrac34; 0,0,0 \,|\, \psi^{-4}) + \\
	&\qquad + 3qH_q(\tfrac{1}{14}, \tfrac{9}{14}, \tfrac{11}{14}; 0, \tfrac{1}{4}, \tfrac{3}{4} \,|\, \psi^4) + 3qH_q(\tfrac{3}{14}, \tfrac{5}{14}, \tfrac{13}{14}; 0, \tfrac{1}{4}, \tfrac{3}{4} \,|\, \psi^4) \\
	&\qquad  - 3\tfrac{1}{q} g(\tfrac{\qq}{7})g(\tfrac{4\qq}{7}) g(\tfrac{2\qq}{7})  \\
	&\qquad + (4q - 2) \\
	&= q^2 + q + 1+ H_q(\tfrac14, \tfrac12, \tfrac34; 0,0,0 \,|\, \psi^{-4})\\
	&\qquad + 3qH_q(\tfrac{1}{14}, \tfrac{9}{14}, \tfrac{11}{14}; 0, \tfrac{1}{4}, \tfrac{3}{4} \,|\, \psi^4) + 3qH_q(\tfrac{3}{14}, \tfrac{5}{14}, \tfrac{13}{14}; 0, \tfrac{1}{4}, \tfrac{3}{4} \,|\, \psi^4)
	\end{aligned}\end{equation}
	as desired.
\end{proof}

\subsection{Remaining pencils}

For the remaining three pencils $\Fsf_2\Lsf_2$, $\Lsf_2\Lsf_2$, and $\Lsf_4$, the formula for the point counts can be derived in a similar manner.  The details can be found in Appendix \ref{appendix:pts}; we state here only the results.

\begin{prop}\label{prop:F2L2}
For $q$ odd and $\psi \in \F_q^\times$, the following statements hold.  
\begin{enumalph}
\item If $q\equiv3\pmod4$, then
\begin{align*} 
\#X_{\Fsf_2\Lsf_2,\psi}(\F_q)&=q^2-q+1+H_q(\tfrac{1}{4},\tfrac{1}{2},\tfrac{3}{4};0,0,0\,|\,\psi^{-4}) - q H_q(\tfrac{1}{4},\tfrac{3}{4};0,\tfrac{1}{2} \,|\,\psi^{-4}).
\end{align*}
\item If $q\equiv5\pmod8$, then
\begin{align*}
\#X_{\Fsf_2\Lsf_2,\psi}(\F_q)&=q^2-q+1+H_q(\tfrac{1}{4},\tfrac{1}{2},\tfrac{3}{4};0,0,0\,|\,\psi^{-4})\\
	&\qquad+qH_q(\tfrac{1}{4},\tfrac{3}{4};0,\tfrac{1}{2} \,|\,\psi^{-4})-2qH_q(\tfrac{1}{2};0\,|\,\psi^{-4}).  
\end{align*}

\item If $q\equiv1\pmod8$, then
\begin{align*}\#X_{\Fsf_2\Lsf_2,\psi}(\F_q)&= q^2+7q +1 +H_q(\tfrac{1}{4},\tfrac{1}{2},\tfrac{3}{4};0,0,0\,|\,\psi^{-4})  +qH_q(\tfrac14, \tfrac{3}{4}; 0, \tfrac12 \,|\, \psi^{-4})  \\
	&\qquad + 2q H_q(\tfrac12; 0 \,|\, \psi^{-4}) + 2\omega(2)^{\qq\qm/4}qH_q(\tfrac18, \tfrac58; 0, \tfrac14 \,|\, \psi^4) + 2\omega(2)^{\qq\qm/4}qH_q(\tfrac38, \tfrac78; 0, \tfrac34 \,|\, \psi^4).
\end{align*}
\end{enumalph}
\end{prop}

\begin{proof}
See Proposition \ref{prop:F2L2-appendix}.
\end{proof}

\begin{prop}\label{prop:L2L2}
For $q$ odd and $\psi \in \F_q^\times$, the following statements hold.
\begin{enumalph}
\item If $q\equiv3\pmod4$, then
\[ \#X_{\Lsf_2\Lsf_2,\psi}(\F_q) = q^2 + q+1 + H_q(\tfrac14, \tfrac12, \tfrac34; 0,0,0 \,|\, \psi^{-4}) - qH_q(\tfrac14, \tfrac{3}{4}; 0, \tfrac12 \,|\, \psi^{-4}). \]
\item If $q\equiv1\pmod4$, then
\begin{align*}
\#X_{\Lsf_2\Lsf_2,\psi}(\F_q) &= q^2+9q+1  + H_q(\tfrac14, \tfrac12, \tfrac34; 0,0,0 \,|\, \psi^{-4}) + qH_q(\tfrac14, \tfrac{3}{4}; 0, \tfrac12 \,|\, \psi^{-4}) \\
	&\qquad + 2 (-1)^{\qq\qm/4} \omega(\psi)^{\qq\qm/2} q H_q(\tfrac{1}{8},\tfrac{3}{8},\tfrac{5}{8},\tfrac{7}{8};0,\tfrac{1}{4},\tfrac{1}{2},\tfrac{3}{4}\,|\,\psi^{-4}).
\end{align*}
\end{enumalph}
\end{prop}

\begin{proof}
See Proposition \ref{prop:L2L2-appendix}.
\end{proof}

\begin{prop}\label{prop:L4}
For $q$ coprime to $10$ and $\psi \in \F_q^\times$, the following statements hold.
\begin{enumalph}
\item If $q\not\equiv1\pmod5$, then
\[ \#X_{\Lsf_4,\psi}(\F_q)=q^2+3q+1+H_q(\tfrac{1}{4},\tfrac{1}{2},\tfrac{3}{4};0,0,0\,|\,\psi^{-4}).\]
\item If $q\equiv 1 \pmod 5$, then
\[ \#X_{\Lsf_4,\psi}(\F_q)=q^2+3q+1+H_q(\tfrac{1}{4},\tfrac{1}{2},\tfrac{3}{4};0,0,0\,|\,\psi^{-4})+4qH_q(\tfrac{1}{5},\tfrac{2}{5}, \tfrac{3}{5},\tfrac{4}{5} ;0, \tfrac{1}{4},\tfrac{1}{2}, \tfrac{3}{4}\,|\,\psi^{4}).\]
\end{enumalph}
\end{prop}

\begin{proof}
See Proposition \ref{prop:L4-appendix}.
\end{proof}

\section{Proof of the main theorem and applications} \label{S:mainthmpf}

In this section, we prove Main Theorem \ref{mainthm} by converting the hypergeometric point count formulas in the previous section into a global $L$-series.  We conclude with some discussion and applications.

\subsection{From point counts to \texorpdfstring{$L$}{L}-series}  \label{S:finalauto}

In this section, we define $L$-series of K3 surfaces and hypergeometric functions, setting up the notation we will use in the proof of our main theorem.  

We begin with $L$-series of K3 surfaces.  Let $\psi \in \Q \smallsetminus \{0,1\}$.  Let $\diamond \in \{\Fsf_4,\Fsf_2\Lsf_2,\Fsf_1\Lsf_3,\Lsf_2\Lsf_2,\Lsf_4\}$ signify one of the five \textup{K3}\/ families in \textup{\eqref{table:5families}}.  Let $S=S(\diamond,\psi)$ be the set of bad primes in \textup{\eqref{table:5families}} together with the primes dividing the numerator or denominator of either $\psi^4$ or $\psi^4-1$.  

\begin{lem}
For $p \not \in S(\diamond,\psi)$, the surface $X_{\diamond,\psi}$ has good reduction at $p$.
\end{lem}

\begin{proof}
Straightforward calculation.
\end{proof}

Let $p \not\in S(\diamond,\psi)$.  The zeta function of $X_{\diamond,\psi}$ over $\F_p$ is of the form
\begin{equation} \label{eqn:ZpXt}
Z_p(X_{\diamond,\psi},T) = \frac{1}{(1-T)(1-pT)P_{\diamond,\psi,p}(T)(1-p^2 T)}
\end{equation}
where $P_{\diamond,\psi,p}(T) \in 1+T\Z[T]$.  The Hodge numbers of $X_{\diamond,\psi}$ imply that the polynomial $P_{\diamond,\psi,p}(T)$ has degree 21.  Equivalently, we have that
\begin{equation} 
P_{\diamond,\psi,p}(T) = \det(1 - \Frob_p^{-1} T \,|\, H_{\textup{\'et},\textup{prim}}^2(X_{\diamond,\psi},\Q_\ell))
\end{equation}
is the characteristic polynomial of the Frobenius automorphism acting on primitive second degree \'etale cohomology for $\ell \neq p$ (and independent of $\ell$).  We then define the (incomplete) $L$-series
\begin{equation} 
L_S(X_{\diamond,\psi},s) \colonequals \prod_{p \not \in S} P_{\diamond,\psi,p}(p^{-s})^{-1},
\end{equation}
convergent for $s \in \C$ in a right half-plane by elementary estimates.  

We now turn to hypergeometric $L$-series, recalling the definitions made in section \ref{defn:hgmdef}--\ref{defn:hgmdef-hybrid}.  Let $\pmb{\alpha},\pmb{\beta}$ be multisets of rational numbers that are disjoint modulo $\Z$.  Let $t \in \Q \smallsetminus \{0,1\}$, and let $S(\pmb{\alpha},\pmb{\beta},t)$ be the set of primes dividing a denominator in $\pmb{\alpha} \cup \pmb{\beta}$ together with the primes dividing the numerator or denominator of either $t$ or $t-1$.  

Recall the Definition \ref{Hybrid Definition} of the finite field hypergeometric sums $H_{q}(\pmb{\alpha};\pmb{\beta}\,|\,t) \in K_{\pmb{\alpha},\pmb{\beta}} \subseteq \C$.  For a prime power $q$ such that $\pmb{\alpha},\pmb{\beta}$ is splittable, we define the formal series
\begin{equation} 
L_q(H(\pmb{\alpha},\pmb{\beta}\,|\,t), T) \colonequals \exp\biggl(-\sum_{r=1}^{\infty} H_{q^r}(\pmb{\alpha};\pmb{\beta}\,|\,t) \frac{T^r}{r} \biggr) \in 1+TK_{\pmb{\alpha},\pmb{\beta}}[[T]]
\end{equation}
using Lemma \ref{lem:fieldofdef}(b).  (Note the negative sign; below, this normalization will yield polynomials instead of inverse polynomials.)  

For a number field $M$, a \defi{prime} of $M$ is a nonzero prime ideal of the ring of integers $\Z_M$ of $M$.  We call a prime $\frakp$ of $M$ \defi{good} (with respect to $\pmb{\alpha},\pmb{\beta},\psi$) if $\frakp$ lies above a prime $p \not\in S(\pmb{\alpha},\pmb{\beta},\psi)$.  Now let $M$ be an abelian extension of $\Q$ containing the field of definition $K \colonequals K_{\pmb{\alpha},\pmb{\beta}}$ with the following property:
\begin{equation} \label{eqn:Mgood}
\text{for all good primes $\frakp$ of $M$, we have $q=\Nm(\frakp)$ splittable for $\pmb{\alpha},\pmb{\beta}$.}
\end{equation} 
For example, if $m$ is the least common multiple of all denominators in $\pmb{\alpha} \cup \pmb{\beta}$, then we may take $M=\Q(\zeta_m)$. We will soon see that we will need to take $M$ to be nontrivial extensions of $K$ in Proposition~\ref{prop:mainthmF1L3} to deal with the splittable hypergeometric function given in Example~\ref{F1L3 hypergeometric example}. Let $m$ be the conductor of $M$, i.e., the minimal positive integer such that $M \subseteq \Q(\zeta_m)$.  Under the canonical identification $(\Z/m\Z)^\times \xrightarrow{\sim} \Gal(\Q(\zeta_m)\,|\,\Q)$ where $k \mapsto \sigma_k$ and $\sigma_k(\zeta_m)=\zeta_m^k$, let $H_M \leq (\Z/m\Z)^\times$ be such that $\Gal(M\,|\,\Q) \simeq (\Z/m\Z)^\times/H_M$.

Now let $p \not\in S(\pmb{\alpha},\pmb{\beta},\psi)$.  Let $\frakp_1,\dots,\frakp_r$ be the primes above $p$ in $M$, and let $q=p^f=\Nm(\frakp_i)$ for any $i$.  Recall (by class field theory for $\Q$) that $f$ is the order of $p$ in $(\Z/m\Z)^\times/H_M$, and $rf=[M:\Q]$.  Moreover, the set of primes $\{\frakp_i\}_i$ arise as $\frakp_i=\sigma_{k_i}(\frakp_1)$ where $k_i \in \Z$ are representatives of the quotient $(\Z/m\Z)^\times/\langle H_M, p \rangle$ of $(\Z/m\Z)^\times$ by the subgroup generated by $H_M$ and $p$.  We then define
\begin{equation}  \label{defofLfactor}
\begin{aligned}
L_p(H(\pmb{\alpha},\pmb{\beta}\,|\,t), M, T) &\colonequals \prod_{i=1}^r L_q(H(k_i\pmb{\alpha},k_i\pmb{\beta}\,|\,t),T^f) \\
&= \prod_{k_i \in (\Z/m\Z)^\times/\langle H_M, p \rangle} L_q(H(k_i\pmb{\alpha},k_i\pmb{\beta}\,|\,t),T^f) \in 1+TK[[T]].
\end{aligned}
\end{equation}
This product is well-defined up to choice of representatives $k_i$ of the cosets in $(\Z/m\Z)^\times/\langle H_M, p \rangle$. Indeed, by Lemma \ref{lem:fieldofdef}: part (c) gives
\begin{equation} L_q(H(pk\pmb{\alpha},pk\pmb{\beta}\,|\,t),T^f)=
L_q(H(k\pmb{\alpha},k\pmb{\beta}\,|\,t^p),T^f)=
L_q(H(k\pmb{\alpha},k\pmb{\beta}\,|\,t),T^f)
\end{equation}
for all $k \in (\Z/m\Z)^\times$ and all good primes $p$, since $t^p=t \in \F_p \subseteq \F_q$; and similarly part (a) implies it is well-defined for $k \in (\Z/m\Z)^\times/H_M$ as $H_M \leq H_K$.

\begin{lem} \label{lem:Hqabdef}
The following statements hold.  
\begin{enumalph}
\item We have
\begin{equation} 
L_p(H(\pmb{\alpha};\pmb{\beta}\,|\,t), M, T) \in 1+T \Q[[T]]
\end{equation}
and 
\begin{equation} \label{eqn:LpHk}
L_p(H(\pmb{\alpha};\pmb{\beta}\,|\,t), M, T)=L_p(H(k\pmb{\alpha};k\pmb{\beta}\,|\,t), M, T)
\end{equation}
for all $k \in \Z$ coprime to $p$ and $m$.  
\item Let $H_K \leq (\Z/m\Z)^\times$ correspond to $\Gal(K\,|\,\Q)$ as above, and let 
\begin{equation}
     r(M|K,p) \colonequals [\langle H_K, p\rangle :\langle H_M,p \rangle ].
\end{equation}  Then $r(M|K,p)$ is the number of primes in $M$ above a prime in $K$ above $p$, and
\begin{equation} \label{eqn:LpHab}
L_p(H(\pmb{\alpha},\pmb{\beta}\,|\,t), M, T) = \prod_{k_i \in (\Z/m\Z)^\times/\langle H_K, p \rangle} L_q(H(k_i\pmb{\alpha},k_i\pmb{\beta}\,|\,t),T^f)^{r(M|K,p)}. 
\end{equation}
\end{enumalph}
\end{lem}

\begin{proof}
For part (a), the descent to $\Q$ follows from Galois theory and Lemma \ref{lem:fieldofdef}; the equality \eqref{eqn:LpHk} follows as multiplication by $k$ permutes the indices $k_i$ in $(\Z/m\Z)^\times/\langle H_M,p\rangle$.  For part (b), the fact that $r({M|K,p})$ counts the number of primes follows again from class field theory; to get \eqref{eqn:LpHab}, use Lemma \ref{lem:fieldofdef} and the fact that the field of definition of $\pmb{\alpha},\pmb{\beta}$ is $K=K_{\pmb{\alpha},\pmb{\beta}}$.
\end{proof}

We again package these together in an $L$-series:
\begin{equation} \label{eqn:LShTs}
L_S(H(\pmb{\alpha};\pmb{\beta}\,|\,t),M, s) \colonequals \prod_{p \not\in S} L_p(H(\pmb{\alpha};\pmb{\beta}\,|\,t), M, p^{-s})^{-1}.
\end{equation}
We may expand \eqref{eqn:LShTs} as a Dirichlet series
\begin{equation} 
L_S(H(\pmb{\alpha};\pmb{\beta}\,|\,t),M, s) = \sum_{\substack{\frakn \subseteq \Z_M \\ \frakn \neq (0)}} \frac{a_{\frakn}}{\Nm(\frakn)^s} 
\end{equation}
with $a_\frakn \in K=K_{\pmb{\alpha},\pmb{\beta}} \subset \C$, and again the series converges for $s \in \C$ in a right half-plane.  If $M=K_{\pmb{\alpha},\pmb{\beta}}$, we suppress the notation $M$ and write just $L_S(H(\pmb{\alpha};\pmb{\beta}\,|\,t),s)$, etc.

Finally, for a finite order Dirichlet character $\chi$ over $M$, we let $L_S(H(\pmb{\alpha};\pmb{\beta}\,|\,t),M,s,\chi)$ denote the twist by $\chi$, defined by
\begin{equation}
L_S(H(\pmb{\alpha};\pmb{\beta}\,|\,t),M,s,\chi) \colonequals \sum_{\substack{\frakn \subseteq \Z_M \\ \frakn \neq (0)}} \frac{\chi(\frakn) a_{\frakn}}{\Nm(\frakn)^s}.
\end{equation}

\subsection{The Dwork pencil \texorpdfstring{$\Fsf_4$}{Fsf4}}
In the remaining sections, we continue with the same notation: let $t=\psi^{-4}$ and let $S=S(\diamond,\psi)$ be the set of bad primes in \eqref{table:5families} together with the set of primes dividing the numerator or denominator of $t$ or $t-1$. We now prove Main Theorem~\ref{mainthm}(a).

\begin{prop}\label{prop:mainthmF4}
Let $\psi \in \Q \smallsetminus \{0,1\}$ and let $t = \psi^{-4}$. Then
\begin{align*}
L_S(X_{\Fsf_4,\psi}, s) &= L_S( H(\tfrac{1}{4}, \tfrac{1}{2}, \tfrac{3}{4}; 0, 0, 0\,|\, t), s) \\
&\qquad \cdot L_S( H(\tfrac{1}{4}, \tfrac{3}{4}; 0, \tfrac{1}{2} \,|\, t), s-1, \phi_{-1})^3 \\
&\qquad \cdot L_S( H(\tfrac{1}{2}; 0 \,|\, t) , \Q(\ii), s-1, \phi_{\ii})^6 
\end{align*}
where
\begin{equation} \label{eqn:fermatchars}
\begin{aligned}
\phi_{-1}(p) &=\legen{-1}{p} = (-1)^{(p-1)/2} &  & \text{ is associated to $\Q(\sqrt{-1}) \,|\, \Q$, and} \\
\phi_{\ii}(\frakp)&=\legen{\ii}{\frakp}=(-1)^{(\Nm(\frakp)-1)/4} & & \text{ is associated to $\Q(\zeta_8)\,|\,\Q(\sqrt{-1})$.}
\end{aligned}
\end{equation}
\end{prop}

\begin{proof}
Recall Proposition~\ref{prop:F4}, where we wrote the number of $\F_q$ points on $\Fsf_4$ in terms of finite field hypergeometric functions.  We rewrite these for convenience: 
\begin{equation} \label{eqn:ptcntform}
\begin{aligned} 
\#X_{\Fsf_4,\psi}(\F_q)&=q^2+q+1+H_q(\tfrac{1}{4},\tfrac{1}{2},\tfrac{3}{4};0,0,0\,|\,t)+\phi_{-1}(q)3qH_q(\tfrac{1}{4},\tfrac{3}{4};0, \tfrac{1}{2}\,|\,t) \\
&\qquad +\delta[q\equiv 1\psmod{4}]12\phi_{\ii}(q)qH_q(\tfrac{1}{2};0\,|\,t) 
\end{aligned}
\end{equation}
where $\delta[\mathcal{P}]=1,0$ according as if $\mathcal{P}$ holds or not. 

Each summand in \eqref{eqn:ptcntform} corresponds to a multiplicative term in the exponential generating series.  The summand $q^2+q+1$ gives the factor $(1-T)(1-qT)(1-q^2T)$ in the denominator of \eqref{eqn:ZpXt}, so $L_S(X_{\Fsf_4,\psi}, s)$ represents the rest of the sum. The summand $ H_q(\tfrac{1}{4}, \tfrac{1}{2}, \tfrac{3}{4};0,0,0\,|\,t)$ yields $L_S( H(\tfrac{1}{4}, \tfrac{1}{2}, \tfrac{3}{4}; 0, 0, 0\,|\, t), s)$ by definition.  

Next we consider the summand $\phi_{-1}(q)3qH_q(\tfrac{1}{4},\tfrac{3}{4};0, \tfrac{1}{2}\,|\,t)$: for each $p \not \in S$, we have
\begin{equation}\begin{aligned}\label{eqn: qHq for quarters}
\exp&\left(-\sum_{r=1}^\infty \phi_{-1}(p^r)3p^rH_{p^r}\left(\tfrac{1}{4},\tfrac{3}{4};0, \tfrac{1}{2}\,|\,t\right)\frac{(p^{-s})^r}{r}\right) \\ &= \exp \left( - \sum_{r=1}^\infty \phi_{-1}(p^r)H_{p^r}(\tfrac{1}{4}, \tfrac{3}{4}; 0, \tfrac{1}{2} \,|\, t)\frac{p^{(1-s)r}}{r}\right)^{3} \\
 &= L_{p}(H_p(\tfrac{1}{4}, \tfrac{3}{4}; 0, \tfrac{1}{2} \,|\, t), p^{1-s}, \phi_{-1})^{3};
\end{aligned}
	\end{equation}
Combining these for all $p \not \in S$ then gives the $L$-series $L_S( H(\tfrac{1}{4}, \tfrac{3}{4}; 0, \tfrac{1}{2} \,|\, t), s-1, \phi_{-1})^3$.  

We conclude with the final term $12\phi_{\sqrt{-1}}(q)qH_q(\tfrac{1}{2};0\,|\,t)$ which exists only when $q \equiv 1 \pmod 4$.  We accordingly consider two cases.  First, if $p \equiv 1 \pmod 4$, then in $\Z[\ii]$, the two primes $\frakp_1,\frakp_2$ above $p$ have norm $p$.  We compute
\begin{equation}\label{expF4pmod1}
\begin{aligned}
\exp&\left(-\sum_{r=1}^\infty 12\phi_{\ii}(p^r)p^rH_{p^r}(\tfrac{1}{2};0\,|\,t)\frac{(p^{-s})^r}{r}\right) \\ 
	&= L_p(H(\tfrac{1}{2};0\,|\,t), p^{1-s}, \phi_{\ii})^{12}\\
	&= L_p(H(\tfrac{1}{2};0\,|\,t), p^{1-s}, \phi_{\ii})^{6}L_p(H(-\tfrac{1}{2};0\,|\,t), p^{1-s}, \phi_{\ii})^{6}\\
	&= L_p( H(\tfrac{1}{2}; 0 \,|\, t) , \Q(\ii), p^{1-s}, \phi_{\ii})^6
\end{aligned}
\end{equation}
where the second equality holds because the definition of the hypergeometric sum only depends on parameters modulo $\Z$ and the final equality is the definition \eqref{defofLfactor} using that $\Gal(M\,|\,\Q)$ when $M = \Q(\sqrt{-1})$ is generated by complex conjugation.

Second, if $p \equiv 3 \pmod 4$, then there is a unique prime ideal $\frakp$ above $p$ with norm $\Nm(\frakp) = p^2$, and
\begin{equation}\label{expF4pmod3}
\begin{aligned}
\exp&\left(-\sum_{r=1}^\infty 12\phi_{\ii}(p^{2r})p^{2r}H_{p^{2r}}(\tfrac{1}{2};0\,|\,t)\frac{(p^{-s})^{2r}}{2r}\right) \\  
	&= L_{p^2}(H(\tfrac{1}{2};0\,|\,t), p^{2(1-s)}, \phi_{\ii})^{6}\\
	&= L_{p}(H(\tfrac{1}{2};0\,|\,t), \Q(\sqrt{-1}), p^{1-s},\phi_{\ii})^6.
\end{aligned}
\end{equation}
Taking the product of~\eqref{expF4pmod1} and~\eqref{expF4pmod3} over all prescribed $p$, we obtain the last $L$-series factor.
\end{proof}

\subsection{The Klein--Mukai pencil \texorpdfstring{$\Fsf_1\Lsf_3$}{Fsf1Lsf3}}

We now prove Theorem~\ref{mainthm}(b).
\begin{prop}\label{prop:mainthmF1L3}
For the Klein--Mukai pencil $\Fsf_1 \Lsf_3$, 
\begin{align*}
L_S(X_{\Fsf_1\Lsf_3,\psi}, s) &= L_S( H(\tfrac{1}{4}, \tfrac{1}{2}, \tfrac{3}{4}; 0, 0, 0\,|\, t), s) \\
&\qquad \cdot L_S( H(\tfrac{1}{14}, \tfrac{9}{14}, \tfrac{11}{14}; 0, \tfrac{1}{4}, \tfrac{3}{4} \,|\, t^{-1}), \Q(\zeta_7), s-1)
\end{align*}
where for $\pmb{\alpha},\pmb{\beta}=\{\tfrac{1}{14}, \tfrac{9}{14}, \tfrac{11}{14}\},\{0, \tfrac{1}{4}, \tfrac{3}{4}\}$ we have field of definition $K_{\pmb{\alpha},\pmb{\beta}}=\Q(\sqrt{-7})$.
\end{prop}

\begin{rmk}
By Lemma \ref{lem:Hqabdef}, we have
\[ L_S( H(\tfrac{1}{14}, \tfrac{9}{14}, \tfrac{11}{14}; 0, \tfrac{1}{4}, \tfrac{3}{4} \,|\, t^{-1}), s) = L_S( H(\tfrac{3}{14}, \tfrac{5}{14}, \tfrac{13}{14}; 0, \tfrac{1}{4}, \tfrac{3}{4} \,|\, t^{-1}), s). \]
\end{rmk}

\begin{proof}
Recall that by Proposition~\ref{prop:F1L3}, we have 
\[
\begin{aligned}
\#X_{\Fsf_1\Lsf_3,\psi}(\F_q)&=q^2+q+1+H_q(\tfrac{1}{4},\tfrac{1}{2},\tfrac{3}{4};0,0,0\,|\,t) \\
&+3q\delta[q \equiv 1 \psmod{7}]\left(H_q(\tfrac{1}{14},\tfrac{9}{14},\tfrac{11}{14};0,\tfrac{1}{4},\tfrac{3}{4}\,|\,t^{-1})+H_q(\tfrac{3}{14},\tfrac{5}{14},\tfrac{13}{14};0,\tfrac{1}{4},\tfrac{3}{4}\,|\,t^{-1})\right) 
\end{aligned} 
\]
We compute that the field of definition (see Definition~\ref{def:fieldofdef}) associated to the parameters $\pmb{\alpha},\pmb{\beta}=\{\tfrac{1}{14}, \tfrac{9}{14}, \tfrac{11}{14}\},\{0, \tfrac{1}{4}, \tfrac{3}{4}\}$ and to 
$\pmb{\alpha},\pmb{\beta}=\{\tfrac{3}{14}, \tfrac{5}{14}, \tfrac{13}{14}\},\{0, \tfrac{1}{4}, \tfrac{3}{4}\}$ is $\Q(\sqrt{-7})$.  We take $M=\Q(\zeta_7)$ and consider a prime $\frakp$ of $M$, and let $q=\Nm(\frakp)$.  Then $q \equiv 1 \pmod{7}$, and in Example~\ref{F1L3 hypergeometric example}, we have seen that $q$ is splittable for $\pmb{\alpha},\pmb{\beta}$. 

We proceed in two cases, according to the splitting behavior of $p$ in $K=K_{\pmb{\alpha},\pmb{\beta}}=\Q(\sqrt{-7})$.  First, suppose that $p\equiv 1,2,4 \pmod 7$, or equivalently $p$ splits in $K$.  We have $2r(M|K,p)f=6$ so $r(M|K,p)=3$.  We then apply Lemma~\ref{lem:Hqabdef}(b), with $(\ZZ/m\ZZ)^\times / \langle H_K, p\rangle = \{\pm 1\}$, to obtain
\begin{align}\label{F1L3 mod124 L}
\exp&\left(-\sum_{r=1}^\infty 3p^{fr}H_{p^{fr}}\left(\tfrac{1}{14},\tfrac{9}{14},\tfrac{11}{14};0,\tfrac{1}{4},\tfrac{3}{4}\,|\,t\right)\frac{(p^{-s})^{fr}}{fr} -3p^{fr}H_{p^{fr}}\left(\tfrac{3}{14},\tfrac{5}{14},\tfrac{13}{14};0,\tfrac{1}{4},\tfrac{3}{4}\,|\,t\right)\frac{(p^{-s})^{fr}}{fr}  \right)  \nonumber\\
    &= L_{p^f}(H(\tfrac{1}{14},\tfrac{9}{14},\tfrac{11}{14};0,\tfrac{1}{4},\tfrac{3}{4}\,|\,t), p^{1-s})^{3/f} L_{p^f}(H(\tfrac{3}{14},\tfrac{5}{14},\tfrac{13}{14};0,\tfrac{1}{4},\tfrac{3}{4}\,|\,t), p^{1-s})^{3/f} \\
    &= \prod_{k_i \in (\ZZ/m\ZZ)^\times / \langle H_K, p\rangle} L_p(H(\tfrac{1}{14}k_i,\tfrac{9}{14}k_i,\tfrac{11}{14}k_i;0,\tfrac{1}{4}k_i,\tfrac{3}{4}k_i\,|\,t), p^{1-s})^{r(M|K,p)} \nonumber \\
    &= L_p(H(\tfrac{1}{14},\tfrac{9}{14},\tfrac{11}{14};0,\tfrac{1}{4},\tfrac{3}{4}\,|\,t), \Q(\zeta_7), p^{1-s}).\nonumber
\end{align}

To conclude, suppose $p \equiv 3,5,6 \pmod{7}$, i.e., $p$ is inert in $K$.  Now $(\ZZ/m\ZZ)^\times / \langle H_K, p\rangle = \{1\}$ and $6 = r(M|K, p)f$.  By Lemma~\ref{lem:fieldofdef}(c), for all $q \equiv 1 \pmod 7$, we have that
\begin{equation}
    H_{q}\left(\tfrac{1}{14},\tfrac{9}{14},\tfrac{11}{14};0,\tfrac{1}{4},\tfrac{3}{4}\,|\,t\right)= H_{q}\left(\tfrac{1}{14}p,\tfrac{9}{14}p,\tfrac{11}{14}p;0,-\tfrac{1}{4},-\tfrac{3}{4}\,|\,t^p\right)= H_{q}\left(\tfrac{3}{14},\tfrac{5}{14},\tfrac{13}{14};0,\tfrac{1}{4},\tfrac{3}{4}\,|\,t\right).
\end{equation}
Using the previous line and Lemma~\ref{lem:Hqabdef}(b),
\begin{align}\label{F1L3 mod356 L}
\exp&\left(-\sum_{r=1}^\infty 3p^{fr}H_{p^{fr}}\left(\tfrac{1}{14},\tfrac{9}{14},\tfrac{11}{14};0,\tfrac{1}{4},\tfrac{3}{4}\,|\,t\right)\frac{(p^{-s})^{fr}}{fr} \right. \nonumber \\
&\qquad\qquad \left. +3p^{fr}H_{p^{fr}}\left(\tfrac{3}{14},\tfrac{5}{14},\tfrac{13}{14};0,\tfrac{1}{4},\tfrac{3}{4}\,|\,t\right)\frac{(p^{-s})^{fr}}{fr}  \right)  \nonumber \\
    &= \exp\left(-\sum_{r=1}^\infty 6H_{p^{fr}}\left(\tfrac{1}{14},\tfrac{9}{14},\tfrac{11}{14};0,\tfrac{1}{4},\tfrac{3}{4}\,|\,t\right)\frac{(p^{1-s})^{fr}}{fr}\right) \nonumber \\
    &= L_{p^f}(H(\tfrac{1}{14},\tfrac{9}{14},\tfrac{11}{14};0,\tfrac{1}{4},\tfrac{3}{4}\,|\,t), p^{1-s})^{6/f}  \\
    &= L_{p^f}(H(\tfrac{1}{14},\tfrac{9}{14},\tfrac{11}{14};0,\tfrac{1}{4},\tfrac{3}{4}\,|\,t), p^{1-s})^{r(M|K, p)} \nonumber \\
    &= L_p(H(\tfrac{1}{14},\tfrac{9}{14},\tfrac{11}{14};0,\tfrac{1}{4},\tfrac{3}{4}\,|\,t), \Q(\zeta_7), p^{1-s}).  \qedhere
\end{align}
\end{proof}

\subsection{The pencil \texorpdfstring{$\Fsf_2\Lsf_2$}{Fsf2Lsf2}}

We now prove Theorem~\ref{mainthm}(c):
\begin{prop}\label{prop:mainthmF2L2}
For the pencil $\Fsf_2 \Lsf_2$, 
\begin{align*}
L_S(X_{\Fsf_2\Lsf_2,\psi}, s) &= L_S( H(\tfrac{1}{4}, \tfrac{1}{2}, \tfrac{3}{4}; 0, 0, 0\,|\, t), s) \\
&\qquad \cdot L_S(\Q(\zeta_8)\,|\,\Q,s-1)^2 \\
&\qquad \cdot L_S( H(\tfrac{1}{4}, \tfrac{3}{4}; 0, \tfrac{1}{2} \,|\, t), s-1, \phi_{-1}) \\
&\qquad \cdot L_S( H( \tfrac{1}{2};0 \,|\, t) , \Q(\sqrt{-1}), s-1, \phi_{\ii}) \\
&\qquad \cdot L_S( H(\tfrac{1}{8}, \tfrac{5}{8}; 0, \tfrac{1}{4} \,|\, t^{-1}), \Q(\zeta_8), s-1, \phi_{\sqrt{2}})
\end{align*}
where 
\begin{itemize}
\item the character $\phi_{\ii}$ is defined in \eqref{eqn:fermatchars}, 
\item for $\pmb{\alpha},\pmb{\beta}=\{\tfrac{1}{8}, \tfrac{5}{8}\},\{0, \tfrac{1}{4}\}$ we have the field of definition $K_{\pmb{\alpha},\pmb{\beta}}=\Q(\sqrt{-1})$, $M = \Q(\zeta_8)$,  and 
$$
\phi_{\sqrt{2}}(\frakp) \colonequals \legen{\sqrt{2}}{\frakp} \equiv 2^{(\Nm(\frakp)-1)/4} \psmod{\frakp} \text{ is associated to $\Q(\zeta_8,\sqrt[4]{2}) \,|\, \Q(\zeta_8)$, and} 
$$ 
\item  $L(\Q(\zeta_8)\,|\,\Q,s) = \zeta_{\Q(\zeta_8)}(s)/\zeta(s)$, where $\zeta_{\Q(\zeta_8)}(s)$ is the Dedekind zeta function of $\Q(\zeta_8)$ and $\zeta(s)=\zeta_\Q(s)$ the Riemann zeta function.\end{itemize}
\end{prop}

\begin{proof}
We now appeal to Proposition~\ref{prop:F2L2}, which we summarize as:
\begin{equation} \label{eqn:Fsf2sum}
\begin{aligned}
    \#X_{\Fsf_2\Lsf_2,\psi}(\F_q)&= q^2 + q + 1 + 2q\cdot \begin{cases} 3 & \text{ if $q \equiv 1 \pmod 8$} \\ -1 & \text{ if $q \not\equiv 1 \pmod 8$}\end{cases} \\
    &\quad + H_q(\tfrac{1}{4},\tfrac{1}{2},\tfrac{3}{4};0,0,0\,|\,t) + \phi_{-1}(q)qH_q(\tfrac{1}{4},\tfrac{3}{4};0,\tfrac{1}{2} \,|\,t) \\ 
    &\quad + 2\phi_{\sqrt{-1}}(q) q\delta[q\equiv 1 \psmod 4]H_q(\tfrac12; 0 \,|\, t) \\
    &\quad+ 2 \phi_{\sqrt{2}}(q)q\delta[q\equiv 1 \psmod 8]\left(H_q(\tfrac18, \tfrac58; 0, \tfrac14 \,|\, t^{-1}) +  H_q(\tfrac38, \tfrac78; 0, \tfrac34 \,|\, t^{-1})\right).
\end{aligned}
\end{equation}

For the new sum with parameters $\pmb{\alpha},\pmb{\beta}=\{\tfrac{1}{8}, \tfrac{5}{8}\},\{0, \tfrac{1}{4}\}$, we have field of definition $K_{\pmb{\alpha},\pmb{\beta}}=\Q(\sqrt{-1})$ because the subgroup of $(\Z/8\Z)^\times$ preserving these subsets is generated by $5$.

The term $q^2+q+1$ in \eqref{eqn:Fsf2sum} is handled as before.  For the next term, by splitting behavior in the biquadratic field $\Q(\zeta_8)=\Q(\sqrt{-1},\sqrt{2})$ we obtain
\begin{equation} \label{F2L2 Explicit L} L_p(\Q(\zeta_8)\,|\,\Q,pT) = 
\begin{cases}
(1-pT)^3, & \text{ if $p \equiv 1 \psmod{8}$;} \\
\displaystyle{\frac{(1-(pT)^2)^2}{1-pT}=(1-pT)(1+pT)^2}, & \text{ if $p \not\equiv 1 \psmod{8}$}.
\end{cases} 
\end{equation}
For $q \equiv 1 \pmod{8}$ the contribution to the exponential generating series is $3q$, otherwise the contribution is $q-2q=-q$.

All remaining terms except for the term
$$
 2\phi_{\sqrt{2}}(q) q\delta[q\equiv 1 \psmod 8](H_q(\tfrac18, \tfrac58; 0, \tfrac14 \,|\, t^{-1}) +  H_q(\tfrac38, \tfrac78; 0, \tfrac34 \,|\, t^{-1}))
$$
are handled in the proof of Proposition~\ref{prop:mainthmF4}.  We choose $M = \Q(\zeta_8)$ which has conductor $m=8$.  Then $H_K = \langle 5\rangle \leq (\ZZ/8\ZZ)^\times$.  Let $\epsilon=\legen{\sqrt{2}}{\frakp}$ for a prime $\frakp$ above $p$ in $\Q(\zeta_8)$ (and independent of this choice).  

Suppose that $p \equiv 1 \pmod 8$. We compute that $f=1$ and $r(M|K,p) = 2$. By applying Lemma~\ref{lem:Hqabdef}(b), we have:
\begin{align}\label{F2L2 L p1}
\exp&\left(-\sum_{r=1}^\infty 2 \epsilon^r p^{r}H_{p^{r}}\left(\tfrac{1}{8},\tfrac{5}{8};0,\tfrac{1}{4}\,|\,t^{-1}\right)\frac{(p^{-s})^{r}}{r} - \sum_{r=1}^\infty 2 \epsilon^r p^r H_{p^{r}}\left(\tfrac{3}{8},\tfrac{7}{8};0,\tfrac{3}{4}\,|\,t^{-1}\right)\frac{(p^{-s})^{r}}{r} \right) \nonumber\\
    &= L_p(H\left(\tfrac{1}{8},\tfrac{5}{8};0,\tfrac{1}{4}\,|\,t^{-1}\right),p^{1-s}, \phi_{\sqrt{2}})^2L_p(H_{p^{r}}\left(\tfrac{3}{8},\tfrac{7}{8};0,\tfrac{3}{4}\,|\,t^{-1}\right), p^{1-s}, \phi_{\sqrt{2}})^2  \\
    &= \prod_{k \in (\ZZ / 8\ZZ)^\times / \langle H_K, p\rangle} L_p(H\left(\tfrac{1}{8}k,\tfrac{5}{8}k;0,\tfrac{1}{4}k\,|\,t^{-1}\right),p^{1-s}, \phi_{\sqrt{2}})^2 \nonumber\\
    &= L_p(H\left(\tfrac{1}{8},\tfrac{5}{8};0,\tfrac{1}{4}\,|\,t^{-1}\right), \Q(\zeta_8),p^{1-s}, \phi_{\sqrt{2}}). \nonumber
\end{align}

Suppose now that $p\equiv 3,7 \pmod 8$. Then $f = 2$ and $r(M|K,p) = 2$. By Lemma~\ref{lem:fieldofdef}(c), for all $q$ a power of $p$ so that $q \equiv 1 \pmod 8$, we have that
\begin{equation}
    H_{q}\left(\tfrac{1}{8},\tfrac{5}{8};0,\tfrac{1}{4}\,|\,t^{-1}\right)= H_{q}\left(\tfrac{1}{8}p,\tfrac{5}{8}p;0,\tfrac{1}{4}p\,|\,t^{-p}\right)= H_{q}\left(\tfrac{3}{8},\tfrac{7}{8};0,\tfrac{3}{4}\,|\,t^{-1}\right).
\end{equation}
 Again applying Lemma~\ref{lem:Hqabdef}(b), we have:
\begin{align}\label{F2L2 L p37}
\exp&\left(-\sum_{r=1}^\infty 2\eps^r p^{2r}H_{p^{2r}}\left(\tfrac{1}{8},\tfrac{5}{8};0,\tfrac{1}{4}\,|\,t^{-1}\right)\frac{(p^{-s})^{2r}}{2r} - \sum_{r=1}^\infty 2\eps^r p^{2r}H_{p^{2r}}\left(\tfrac{3}{8},\tfrac{7}{8};0,\tfrac{3}{4}\,|\,t^{-1}\right)\frac{(p^{-s})^{2r}}{2r} \right) \nonumber\\
    &=\exp\left(-\sum_{r=1}^\infty 2\eps^rp^{2r}H_{p^{2r}}\left(\tfrac{1}{8},\tfrac{5}{8};0,\tfrac{1}{4}\,|\,t^{-1}\right)\frac{(p^{-s})^{2r}}{r}\right) \nonumber\\
    &= L_{p^2}(H\left(\tfrac{1}{8},\tfrac{5}{8};0,\tfrac{1}{4}\,|\,t^{-1}\right),p^{1-s}, \phi_{\sqrt{2}})^2 \\
    &=  L_{p^2}(H\left(\tfrac{1}{8},\tfrac{5}{8};0,\tfrac{1}{4}\,|\,t^{-1}\right),p^{1-s}, \phi_{\sqrt{2}})^{r(M|K,p)} \nonumber\\
    &= L_{p^2}(H\left(\tfrac{1}{8},\tfrac{5}{8};0,\tfrac{1}{4}\,|\,t^{-1}\right), \Q(\zeta_8),p^{1-s}, \phi_{\sqrt{2}}). \nonumber
\end{align}

Finally, suppose that $p \equiv 5 \pmod 8$. Then $f=2$ and $r(M|K, p) = 1$, and now
\begin{align}\label{F2L2 L p5}
    \exp&\left(-\sum_{r=1}^\infty 2\eps^rp^{2r}H_{p^{2r}}\left(\tfrac{1}{8},\tfrac{5}{8};0,\tfrac{1}{4}\,|\,t^{-1}\right)\frac{(p^{-s})^{2r}}{2r} - \sum_{r=1}^\infty 2\eps^rp^{2r}H_{p^{2r}}\left(\tfrac{3}{8},\tfrac{7}{8};0,\tfrac{3}{4}\,|\,t^{-1}\right)\frac{(p^{-s})^{2r}}{2r} \right) \nonumber \\
    &= L_{p^2}(H\left(\tfrac{1}{8},\tfrac{5}{8};0,\tfrac{1}{4}\,|\,t^{-1}\right),p^{1-s}, \phi_{\sqrt{2}}) L_{p^2}(H\left(\tfrac{3}{8},\tfrac{7}{8};0,\tfrac{3}{4}\,|\,t^{-1}\right),p^{1-s}, \phi_{\sqrt{2}}) \\
    &= \prod_{k\in (\ZZ / 8\ZZ)^\times / \langle H_K, p\rangle} L_{p^2}(H\left(\tfrac{1}{8}k,\tfrac{5}{8}k;0,\tfrac{1}{4}k\,|\,t^{-1}\right),p^{1-s}, \phi_{\sqrt{2}}) \nonumber \\
    &= L_p(H\left(\tfrac{1}{8},\tfrac{5}{8};0,\tfrac{1}{4}\,|\,t^{-1}\right), \Q(\zeta_8),p^{1-s}, \phi_{\sqrt{2}}).\qedhere
\end{align}
\end{proof}

\subsection{The pencil \texorpdfstring{$\Lsf_2\Lsf_2$}{Lsf2Lsf2}}

We now prove Theorem~\ref{mainthm}(d).
\begin{prop}
For the pencil $\Lsf_2 \Lsf_2$, we have
\begin{align*}
L_S(X_{\Lsf_2\Lsf_2,\psi}, s) &= L_S( H(\tfrac{1}{4}, \tfrac{1}{2}, \tfrac{3}{4}; 0, 0, 0\,|\, t), s) \\
&\qquad \cdot \zeta_{\Q(\sqrt{-1})}(s-1)^4 \\
&\qquad \cdot L_S( H(\tfrac{1}{4}, \tfrac{3}{4}; 0, \tfrac{1}{2} \,|\, t), s-1, \phi_{-1}) \\
&\qquad \cdot     L_S( H(\tfrac{1}{8}, \tfrac{3}{8}, \tfrac{5}{8}, \tfrac{7}{8}; 0, \tfrac{1}{4}, \tfrac{1}{2}, \tfrac{3}{4}\,|\, t), \Q(i), s-1, \phi_{\ii}\phi_{\psi}) 
\end{align*}
where the characters $\phi_{-1},\phi_{\ii}$ are defined in \eqref{eqn:fermatchars} and
\begin{equation}
\begin{aligned}
\phi_{\psi}(p) &=\legen{\psi}{p} & & \text{ is associated to $\Q(\sqrt{\psi})\,|\,\Q$} 
\end{aligned}
\end{equation}
and $\zeta_{\Q(\sqrt{-1})}(s)$ is the Dedekind zeta function of $\Q(\sqrt{-1})$.
\end{prop}

\begin{proof}
By Proposition~\ref{prop:L2L2}, we have the point counts
\begin{equation}
    \begin{aligned}
 \#X_{\Lsf_2\Lsf_2,\psi}(\F_q) &= q^2 + q+1 + 
4q\cdot \begin{cases} 2 & \text{ if $q \equiv 1 \pmod 4$} \\ 0 & \text{ if $q \equiv 3 \pmod 4$}\end{cases} \\
&\quad + H_q(\tfrac14, \tfrac12, \tfrac34; 0,0,0 \,|\, t) +(-1)^{\qq\qm/2}qH_q(\tfrac14, \tfrac{3}{4}; 0, \tfrac12 \,|\, t) \\
&\quad + 2 (-1)^{\qq\qm/4} \omega(\psi)^{\qq\qm/2} q\delta[q\equiv 1\psmod{4}]  H_q(\tfrac{1}{8},\tfrac{3}{8},\tfrac{5}{8},\tfrac{7}{8};0,\tfrac{1}{4},\tfrac{1}{2},\tfrac{3}{4}\,|\,t).
	\end{aligned}
	\end{equation}

Again by splitting behavior, we have
\begin{equation}\label{Explicit L L2L2} \zeta_{\Q(i),p}(pT) = \begin{cases}
(1-pT)^2, & \text{ if $p \equiv 1 \psmod{4}$};  \\
1-p^2T^2, & \text{ if $p \equiv 3 \psmod{4}$}.
\end{cases} \end{equation}
For $q \equiv 1 \pmod{4}$, the contribution to the exponential generating series is $2q$, otherwise the contribution is $0$.

All but the last summand have been identified in the previous propositions, and this one follows in a similar but easier manner (because it has $\Q$ as field of definition) applying Lemma~\ref{lem:Hqabdef}(b), with $M=\Q(\sqrt{-1})$ and $fr(M|\Q,p)=2$:
\begin{align}\label{L2L2 L new series}
\exp&\left(-\sum_{r=1}^\infty
2(-1)^{(p^{fr})^\times\qm/4} \omega(\psi)^{(p^{fr})^\times\qm/2} p^{fr}  H_q(\tfrac{1}{8},\tfrac{3}{8},\tfrac{5}{8},\tfrac{7}{8};0,\tfrac{1}{4},\tfrac{1}{2},\tfrac{3}{4}\,|\,t)\frac{(p^{-s})^{fr}}{fr} \right) \nonumber \\
&= L_{p^f}(H(\tfrac{1}{8},\tfrac{3}{8},\tfrac{5}{8},\tfrac{7}{8};0,\tfrac{1}{4},\tfrac{1}{2},\tfrac{3}{4}\,|\,t),p^{1-s},\phi_{\psi}\phi_{-1})^{2/f}\\
&= L_{p^f}(H(\tfrac{1}{8},\tfrac{3}{8},\tfrac{5}{8},\tfrac{7}{8};0,\tfrac{1}{4},\tfrac{1}{2},\tfrac{3}{4}\,|\,t),p^{1-s},\phi_{\psi}\phi_{-1})^{r(M|\Q,p)} \nonumber\\
    &= L_p(H(\tfrac{1}{8},\tfrac{3}{8},\tfrac{5}{8},\tfrac{7}{8};0,\tfrac{1}{4},\tfrac{1}{2},\tfrac{3}{4}\,|\,t), \Q(\sqrt{-1}),p^{1-s},\phi_{\psi}\phi_{-1}). \qedhere
\end{align}
\end{proof}

\subsection{The pencil \texorpdfstring{$\Lsf_4$}{Lsf4}}

Here we prove Theorem~\ref{mainthm}(e). 
\begin{prop}
For the pencil $\Lsf_4$, 
\begin{align*}
L_S(X_{\Lsf_4,\psi}, s) &= L_S( H(\tfrac{1}{4}, \tfrac{1}{2}, \tfrac{3}{4}; 0, 0, 0\,|\, t), s) \zeta(s-1)^2 \\
&\qquad \cdot L_S( H( \tfrac{1}{5}, \tfrac{2}{5}, \tfrac{3}{5}, \tfrac{4}{5}; 0, \tfrac{1}{4}, \tfrac{1}{2}, \tfrac{3}{4} \,|\, t^{-1}) , \Q(\zeta_5), s-1),
\end{align*}
\end{prop}

\begin{proof} By Proposition~\ref{prop:L4}, we have that 
\begin{align*}\#X_{\Lsf_4}(\psi)&=q^2+3q+1+H_q(\tfrac{1}{4},\tfrac{1}{2},\tfrac{3}{4};0,0,0\,|\,t)\\
    &\quad + 4q\delta[q\equiv1 \psmod5]H_q(\tfrac{1}{5},\tfrac{2}{5}, \tfrac{3}{5},\tfrac{4}{5} ;0, \tfrac{1}{4},\tfrac{1}{2}, \tfrac{3}{4}\,|\,t^{-1}).
\end{align*}

The extra two summands of $q$ in the point count correspond to the $L$-series factor $\zeta(s-1)^2$.  We now  focus on the remaining new summand $4qH_q(\tfrac{1}{5},\tfrac{2}{5}, \tfrac{3}{5},\tfrac{4}{5} ;0, \tfrac{1}{4},\tfrac{1}{2}, \tfrac{3}{4}\,|\,\psi^{4})$ that occurs exactly when $q \equiv 1 \pmod 5$.  Let $f$ be the order of $p$ in $(\ZZ/5\ZZ)^\times$, which divides 4. Take $M = \Q(\zeta_5)$. We know that $K = \Q$ for all possible $p$, hence $r(M|K, p) = 4f^{-1}$. By Lemma~\ref{lem:Hqabdef}, we have
\begin{align}\label{L4 new L series}
    \exp&\left(-\sum_{r=1}^\infty 4p^{rf}H_{p^{rf}}(\tfrac{1}{5},\tfrac{2}{5}, \tfrac{3}{5},\tfrac{4}{5} ;0, \tfrac{1}{4},\tfrac{1}{2}, \tfrac{3}{4}\,|\,t^{-1}) \frac{(p^{-s})^{rf}}{fr}\right) \nonumber \\
    &= L_{p^f}(H(\tfrac{1}{5},\tfrac{2}{5}, \tfrac{3}{5},\tfrac{4}{5} ;0, \tfrac{1}{4},\tfrac{1}{2}, \tfrac{3}{4}\,|\,t^{-1}), p^{1-s})^{4/f} \\ 
    &= L_{p^f}(H(\tfrac{1}{5},\tfrac{2}{5}, \tfrac{3}{5},\tfrac{4}{5} ;0, \tfrac{1}{4},\tfrac{1}{2}, \tfrac{3}{4}\,|\,t^{-1}), p^{1-s})^{r(M|\Q,p)} \nonumber \\
    &= L_p(H(\tfrac{1}{5},\tfrac{2}{5}, \tfrac{3}{5},\tfrac{4}{5} ;0, \tfrac{1}{4},\tfrac{1}{2}, \tfrac{3}{4}\,|\,t^{-1}),\Q(\zeta_5), p^{1-s}).\qedhere
\end{align}
\end{proof}

\subsection{Algebraic hypergeometric functions}  \label{sec:alghyp}

We now turn to some applications of our main theorem.  We begin in this section by setting up a discussion of explicit identification of the algebraic hypergometric functions that arise in our decomposition, following foundational work of Beukers--Heckman \cite{BH}.

Recall the hypergeometric function $F(z)=F(\pmb{\alpha};\pmb{\beta} \,|\, z)$ (Definition \ref{defn:hypergeometricFunction}) for parameters $\pmb{\alpha},\pmb{\beta}$.  For certain special parameters, this function may be \emph{algebraic} over $\C(z)$, i.e., the field $\C(z,F(z))$ is a finite extension of $\C(z)$.  By a criterion of Beukers--Heckman, $F(z)$ is algebraic if and only if the parameters interlace (ordering the parameters, they alternate between elements of $\pmb{\alpha}$ and $\pmb{\beta}$) \cite[Theorem 4.8]{BH}; moreover, all sets of interlacing parameters are classified \cite[Theorem 7.1]{BH}.  
We see in Main Theorem \ref{mainthm} that for all but the common factor $L_S(H(\tfrac14,\tfrac12,\tfrac34;0,0,0\,|\,t),s)$, the parameters interlace, so this theory applies.  

\begin{conj} \label{conj:degreeofArtin}
Let $t \in \Q$, let $\pmb{\alpha},\pmb{\beta}$ with $\#\pmb{\alpha}=\pmb{\beta}=d$ be such that the hypergeometric function $F(\pmb{\alpha};\pmb{\beta} \,|\, z)$ is algebraic.  Let $M$ satisfy \eqref{eqn:Mgood}.  Then $L_S(H(\pmb{\alpha},\pmb{\beta}\,|\,t),M,s)$ is an Artin $L$-series of degree $d[M:K_{\pmb{\alpha},\pmb{\beta}}]$; in particular, for all good primes $p$, we have $L_p(H(\pmb{\alpha},\pmb{\beta}\,|\,t),M,T) \in 1+T\Q[T]$ a polynomial of degree $d[M:K_{\pmb{\alpha},\pmb{\beta}}]$.
\end{conj}

Conjecture \ref{conj:degreeofArtin} is implicit in work of Katz \cite[Chapter 8]{Katz}, and there is current, ongoing work on the theory of hypergeometric motives that is expected to prove this conjecture, at least for certain choices of $M$.  An explicit version of Conjecture \ref{conj:degreeofArtin} could be established in each case for the short list of parameters that arise in our Main Theorem.  For example, we can use the following proposition about $L$-series and apply it for the family $\Fsf_4$, proving a conjecture of Duan \cite{LianDuan}.

\begin{prop}[Cohen]\label{Lseries gen cor F4}
We have the following $L$-series relations:
\begin{equation}
\begin{aligned}
    L_S(H(\tfrac14, \tfrac34; 0, \tfrac12\,|\, \psi^{-4}), s, \phi_{-1}) &= L_S(s, \phi_{1-\psi^2})L_S(s,\phi_{
    -1-\psi^2}) \\    
    L_S(\tfrac12;0\,|\,\psi^{-4}, \Q(\sqrt{-1}),s, \phi_{\sqrt{-1}}) &= L_S(s, \phi_{2(1-\psi^4)})L_S(s,\phi_{-2(1-\psi^4)}),
    \end{aligned}
\end{equation}
where $\phi_a = \legen{a}{p}$ is the Legendre symbol.
\end{prop}
\begin{proof}
The hypergeometric $L$-series were computed explicitly by Cohen \cite[Propositions 6.4 and 7.32]{CohenNotes}, and the above formulation follows directly from this computation.
\end{proof}

More generally, Naskr\k{e}cki \cite{Naskrecki} has given an explicit description for algebraic hypergeometric $L$-series of low degree defined over $\Q$ using the variety defined by Beukers--Cohen--Mellit \cite{BCM}.

Proposition \ref{Lseries gen cor F4}, plugged into our hypergeometric decomposition, gives an explicit decomposition of the polynomial $Q_{\Fsf_4, \psi,q}$ as follows.

\begin{cor} \label{Q for F4} 
We have:
\begin{equation*}
\begin{aligned}
    &Q_{\Fsf_4, \psi,q}(T) \\
    &\quad = \begin{cases} \left(1-\textstyle{\legen{1-\psi^2}{q}}qT\right)^3\left(1-\legen{-1-\psi^2}{q}qT\right)^3 \left(1-\legen{1-\psi^4}{q}qT\right)^{12}, & \text{ if $q \equiv 1 \psmod 4$;} \\
    \left(1-\legen{1-\psi^2}{q}qT\right)^3(1-qT)^6(1+qT)^6, & \text{ if $q \equiv 3 \psmod 4$};
    \end{cases}
    \end{aligned}
    \end{equation*}
    where $\legen{a}{q}$ denotes the Jacobi symbol.
\end{cor}

Corollary \ref{Q for F4} explains why the field of definition of the Picard group involves square roots of $1 + \psi^2$ and $1 - \psi^2$.

\subsection{Applications to zeta functions} \label{sec:zetapp}

To conclude, we give an application to zeta functions.  In Sections \ref{S:PFEqns} and \ref{S:PointCounts}, we established a relationship between the periods and the point counts for our collection of invertible K3 polynomial families and hypergeometric functions. In particular, both the periods and the point counts decompose naturally in terms of the group action into hypergeometric components. 

It is easy to see that the zeta function is the characteristic polynomial of Frobenius acting on our cohomology (i.e., the collection of periods). In this sense, both sections \ref{S:PFEqns} and \ref{S:PointCounts} suggest that, as long as the group action and the action of Frobenius commute, the splitting of Frobenius by the group action translates into factors, each corresponding to the Frobenius acting only on a given isotypical component of the action. However, a priori we only know that this factorization over $\overline{\Q}$ (see e.g.\ work of Miyatani \cite{Miyatani}).

Thus, we have the following corollary of Main Theorem \ref{mainthm}. 
 
\begin{cor} \label{cor:factorzeta}
Assuming Conjecture \textup{\ref{conj:degreeofArtin}}, for smooth $X_{\diamond,\psi,q}$, the polynomials $Q_{\diamond,\psi,q}(T)$ factor over $\Q[T]$ under the given hypothesis as follows:
\begin{equation} \label{table:factdegrees}
\centering
\begin{tabular}[c]{c|c|c}
\textup{Family} & \textup{Factorization} & \textup{Hypothesis}  \\
\hline\hline
\rule{0pt}{2.5ex} \multirow{2}{*}{$\Fsf_4$} & $(\deg\, 2)^{3}(\deg\, 1)^{12}$ & $q\equiv 1 \psmod{4}$  \\
& $(\deg\, 2)^{3}(\deg\, 2)^6$ & $q \equiv 3 \psmod 4$ \\  \hline
\rule{0pt}{2.5ex} \multirow{4}{*}{$\Fsf_1\Lsf_3$} & $(\deg\,3)^3(\deg\,3)^3$ & $q\equiv  1 \psmod{7}$  \\
    & $(\deg\,6)^3$ & $q \equiv 6 \psmod 7$ \\ 
    & $(\deg\,9)(\deg\,9)$ & $q\equiv 2,4 \psmod 7$ \\
    & $(\deg\,18)$ & $q\equiv 3,5 \psmod 7$ \\\hline
\rule{0pt}{2.5ex} \multirow{3}{*}{$\Fsf_2\Lsf_2$} & $(1-qT)^6(\deg\,2)(\deg\,1)^2(\deg\,2)^2(\deg\,2)^2$ & $q\equiv 1 \psmod{8}$  \\
    & $(1-qT)^2(1+qT)^4(\deg\,2)(\deg\,1)^2(\deg\,4)^2$ & $q\equiv 5 \psmod 5$ \\
    & $(1-qT)^2(1+qT)^4(\deg\,2)(\deg\,2)(\deg\,4)(\deg\,4)$ & $q\equiv 3,7 \psmod 8$ \\\hline 
\rule{0pt}{2.5ex} \multirow{2}{*}{$\Lsf_2\Lsf_2$} & $(1-qT)^8(\deg\, 2)(\deg\, 4)^2 $ & $q\equiv 1 \psmod{4}$ \\ 
 & $(1-q^2T^2)^4(\deg\, 2)(\deg\, 8) $ & $q\equiv 3 \psmod{4}$ \\ 
\hline
\rule{0pt}{2.5ex} \multirow{3}{*}{$\Lsf_4$} & $(1-qT)^2(\deg\,4)^4$ & $q\equiv 1 \psmod{5}$  \\
 & $(1-qT)^2(\deg\,8)^2$ & $q\equiv 4 \psmod{5}$  \\
 & $(1-qT)^2(\deg\,16)$ & $q\equiv 2,3 \psmod{5}$  \\
\end{tabular}
\end{equation}
\end{cor}

The factorization in Corollary \ref{cor:factorzeta} is to be read as follows: for the family $\Lsf_2\Lsf_2$ when $q \equiv 1\pmod{4}$, we have $Q_{\diamond,\psi,q}(T)=(1-qT)^8 Q_1(T)Q_2(T)^2$ where $\deg Q_1(T)=2$ and $\deg Q_2(T)=4$, but we do not claim that $Q_1,Q_2$ are irreducible.  A complete factorization into irreducibles depends on $\psi \in \F_q^\times$ and can instead be computed from the explicit Artin $L$-series.

\begin{proof}
For each case, we need to identify the field of definition for the terms associated to hypergeometric functions other than $H(\tfrac{1}{4}, \tfrac{1}{2}, \tfrac{3}{4}; 0, 0, 0\,|\, t)$ and check the degrees of the resulting zeta function factors using Lemma~\ref{lem:Hqabdef} and Conjecture \ref{conj:degreeofArtin}.
 
\begin{enumerate}
\item[$\Fsf_4$.]  The case where $q\equiv 3 \pmod 4$ is straightforward from the statement of Proposition~\ref{prop:mainthmF4}. In the case where $q \equiv 1 \pmod 4$, we see in the proof that the $L$-series $$L_S(H_q(\tfrac{1}{2};0\,|\,t), \Q(\sqrt{-1}), s-1, \phi_{\sqrt{-1}})$$ factors into a square (see Equations~\eqref{expF4pmod1} and~\eqref{expF4pmod3}).

\item[$\Fsf_1\Lsf_3$.] In the case where $q\equiv 1,2,4 \pmod 7$, we see in Equation~\eqref{F1L3 mod124 L} that the $L$-series associated to $H(\tfrac{1}{14}, \tfrac{9}{14}, \tfrac{11}{14}; 0, \tfrac{1}{4}, \tfrac{3}{4} \,|\, t)$ factorizes into two terms with multiplicity $3/f$ where $f$ is the order of $q$ in $(\ZZ/7\ZZ)^\times$. The analogous argument holds for when $q\equiv 3,5,6 \pmod 7$ using Equation~\eqref{F1L3 mod356 L} to see that the $L$-series factors into one term with multiplicity $6/f$.

\item[$\Fsf_2\Lsf_2$.] The explicit factors follow directly from Equation~\eqref{F2L2 Explicit L}. The next two factors come from the $L$-series $L_S( H(\tfrac{1}{4}, \tfrac{3}{4}; 0, \tfrac{1}{2} \,|\, t), s-1, \phi_{-1})$ and  $L_S( H( \tfrac{1}{2};0 \,|\, t) , \Q(\sqrt{-1}), s-1, \phi_{\ii})$ were dealt with in the $\Fsf_4$ case. The zeta function factorization implied by the $L$-series $L_S( H(\tfrac{1}{8}, \tfrac{5}{8}; 0, \tfrac{1}{4} \,|\, t^{-1}), \Q(\zeta_8), s-1)$ follows by using Equations~\eqref{F2L2 L p1},~\eqref{F2L2 L p37}, and~\eqref{F2L2 L p5}. 

\item[$\Lsf_2\Lsf_2$.]  The explicit factors follow directly from Equation~\eqref{Explicit L L2L2}. The next factor has been dealt with above. The final factor is implied by Equation~\eqref{L2L2 L new series}.

\item[$\Lsf_4$.] The term $\zeta(s-1)^2$ gives the $(1-qT)^2$ factor. The last factor is direct from Equation~\eqref{L4 new L series}.
\end{enumerate}
\end{proof}

\begin{exm} \label{exm:computedtable}
Because the reciprocal roots of $Q_{\diamond,\psi,q}(T)$ are of the form $q$ times a root of 1, the factors of $Q_{\diamond,\psi,q}(T)$ over $\Z$ are of the form $\Phi(qT)$, where $\Phi$ is a cyclotomic polynomial.  We now give the explicit zeta functions for the case where $q=281$ and $\psi=18$ in the table below. We use a \textsf{SageMath} interface to C code written by Costa, which is described in a paper of Costa--Tschinkel \cite{CT}. Note that the factorizations in Corollary~\ref{cor:factorzeta} are sharp  for the families $\Fsf_1\Lsf_3$ and $\Lsf_4$. 

\begin{equation} \label{table:factexample}
\begin{tabular}{c|c}
\textup{Family} &  $Q_{\diamond, \psi, q}(T)$ \\
\hline\hline
\rule{0pt}{2.5ex} $\Fsf_4$ &  $ ( 1 - q T ) ^{ 12 } ( 1 + q T ) ^{ 6 }$ \\
$\Fsf_2\Lsf_2$ & $ ( 1 - q T ) ^{ 8 } ( 1 + q T ) ^{ 2 } ( 1 + q ^ { 2 } T ^ { 2 } ) ^{ 4 } $ \\
$\Fsf_1\Lsf_3$ & $ ( 1 + q T + q ^ { 2 } T ^ { 2 } + q ^ { 3 } T ^ { 3 } + q ^ { 4 } T ^ { 4 } + q ^ { 5 } T ^ { 5 } + q ^ { 6 } T ^ { 6 } ) ^{ 3 } $ \\
$\Lsf_2\Lsf_2$ &  $ ( 1 - q T ) ^{ 12 } ( 1 + q T ) ^{ 6 } $ \\
$\Lsf_4$ &  $ ( 1 - q T ) ^{ 2 } ( 1 + q T + q ^ { 2 } T ^ { 2 } + q ^ { 3 } T ^ { 3 } + q ^ { 4 } T ^ { 4 } ) ^{ 4 } $ \\
\end{tabular}
\end{equation}
\end{exm}

\appendix 

\section{Remaining Picard--Fuchs equations} \label{appendix:pfs}

In this appendix, we provide the details in the computation of the remaining three pencils $\Fsf_2\Lsf_2$, $\Lsf_2\Lsf_2$, and $\Lsf_4$.  We follow the same strategy as in sections \ref{PF:Dwork}--\ref{PF:KM}.

\subsection{The \texorpdfstring{$\Fsf_2\Lsf_2$}{Fsf2Lsf2} pencil}

Take the pencil 
\[
F_{\psi}  \colonequals  x_0^4+x_1^4 + x_2^3 x_3 + x_3^3x_2 - 4\psi x_0x_1x_2x_3
\]
that defines the pencil of projective hypersurfaces $X_{\psi} = Z(F_\psi) \subset \mathbb{P}^3$. There is a $\Z/8\Z$ scaling symmetry of this family generated by the element
$$
g (x_0: x_1: x_2: x_3) = (\xi^2 x_0: x_1: \xi x_2 : \xi^5 x_3)
$$ 
where $\xi$ is a primitive eighth root of unity. There are eight characters $\chi_k: H \rightarrow \C^\times$, where $\chi_k(g) = \xi^k$.  We can again decompose $V$ into subspaces $W_{\chi_k}$ and write their monomial bases. Note that the monomial bases for $W_{\chi_1}, W_{\chi_3}, W_{\chi_5},$ and $W_{\chi_7}$ are the same up to transpositions of $x_0$ and $x_1$ or $x_2$ and $x_3$ which leave the polynomial invariant; thus, they have the same Picard--Fuchs equations. The monomial bases for $W_{\chi_2}$ and $W_{\chi_6}$ are related by transposing $x_0$ and $x_1$, so they also have the same Picard--Fuchs equations. So we are left with four types of monomial bases:

\begin{enumerate}[(i)]
\item $W_{\chi_0}$ has monomial basis $\{x_0x_1x_2x_3\}$;
\item $W_{\chi_1}$ has monomial basis $\{x_2^2x_0x_3, x_0^2x_1x_3\}$;
\item $W_{\chi_2}$ has monomial basis $\{x_0^2x_2x_3, x_1^2x_2^2, x_1^2x_3^2\}$; and 
\item $W_{\chi_4}$ has monomial basis $\{x_0^2x_1^2,x_2^2x_3^2,x_2^2x_0x_1, x_3^2x_0x_1\}$.
\end{enumerate}

Using \eqref{DiagramPartials}, we compute the following period relations:
\begin{equation}\begin{aligned}
\mathbf{v} + (4,0,0,0)  &= \frac{1+v_0}{4(\omega + 1)} \mathbf{v} + \psi (\mathbf{v}+(1,1,1,1)) \\
\mathbf{v} + (0,4,0,0)  &=  \frac{1+v_1}{4(\omega + 1)} \mathbf{v} + \psi (\mathbf{v}+(1,1,1,1)) \\
\mathbf{v} + (0,0,3,1)  &= \frac{3(v_2+1)-(v_3+1)}{8(\omega + 1)} \mathbf{v} + \psi (\mathbf{v}+(1,1,1,1))  \\
\mathbf{v} + (0,0,1,3) &= \frac{-(v_2+1) + 3(v_3+1)}{8(\omega + 1)} \mathbf{v} + \psi (\mathbf{v}+(1,1,1,1)) 
\end{aligned}\end{equation}

We can now use the diagram method to prove the following proposition.

\begin{prop} \label{prop:F2L2diff-appendix}
For the $\Fsf_2\Lsf_2$ family, the primitive cohomology group $H^2_{\textup{prim}}(X_{\Fsf_2\Lsf_2, \psi}, \C)$ has $15$ periods whose Picard--Fuchs equations are hypergeometric differential equations as follows:
\[ \begin{gathered}
\text{$3$ periods are annihilated by $D(\tfrac{1}{4}, \tfrac{1}{2}, \tfrac{3}{4} ; 1, 1, 1 \,|\,\psi^{-4})$,} \\ 
\text{ $2$ periods are annihilated by $D(\tfrac{1}{4}, \tfrac{3}{4} ; 1, \tfrac{1}{2} \,|\,\psi^{-4})$,} \\ 
\text{$2$ periods are annihilated by $D(\tfrac{1}{2} ; 1 \,|\,\psi^{4})$,} \\ 
\text{$4$ periods are annihilated by $D(\tfrac{1}{8}, \tfrac{5}{8} ; 1, \tfrac{1}{4} \,|\,\psi^{4})$, and } \\ 
\text{$4$ periods are annihilated by $D(\tfrac{-3}{8}, \tfrac{1}{8} ; 0, \tfrac{1}{4} \,|\,\psi^{4})$.} 
\end{gathered} \]
\end{prop}

We consider each of these in turn.  

\begin{lem}\label{F2L2 quarter quadratic periods}
The Picard--Fuchs equation associated to the periods  $\psi(2,2,0,0)$ and $\psi(0,0,2,2)$ is the hypergeometric differential equation $D(\tfrac{1}{4}, \tfrac{3}{4} ; 1, \tfrac{1}{2} \,|\,\psi^{-4})$.
\end{lem}

\begin{proof}
For the periods $(2,2,0,0)$ and $(0,0,2,2)$, corresponding to the quartic monomials $x_0^2x_1^2$ and $x_2^2x_3^2$, we use the diagram

$$
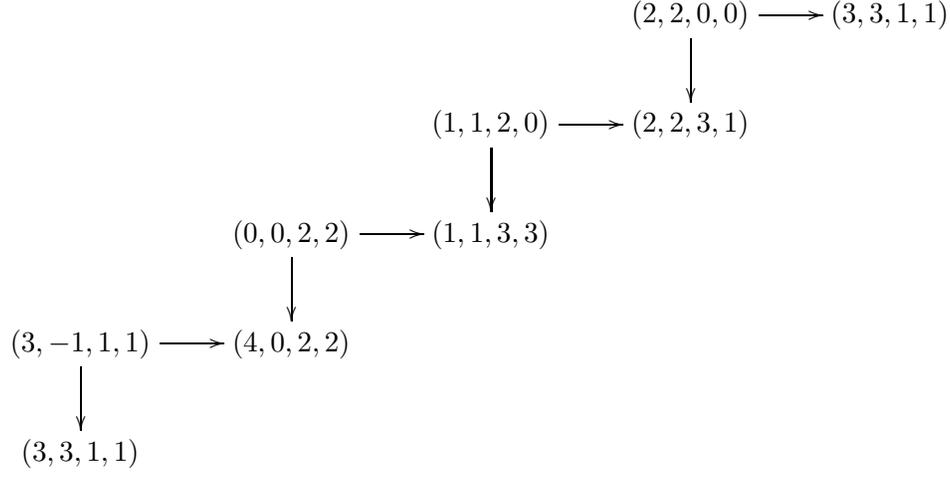
\xymatrix{ &&& (2,2,0,0) \ar[r] \ar[d] & (3,3,1,1) \\
	&& (1,1,2,0) \ar[r] \ar[d] & (2,2,3,1) & \\
	& (0,0,2,2) \ar[r] \ar[d] & (1,1,3,3)&& \\
	(3,-1,1,1) \ar[r] \ar[d] & (4,0,2,2) &&& \\
	(3,3,1,1) &&&&}
$$
and get the relations
\begin{equation}\begin{aligned}
\eta(2,2,0,0) = \psi^2(\eta+1)(0,0,2,2);\\
\eta(0,0,2,2) = \psi^2(\eta+1)(2,2,0,0).
\end{aligned}\end{equation}

For the periods $(2,2,0,0)$ and $(0,0,2,2)$, corresponding to the quartic monomials $x_0^2x_1^2$ and $x_2^2x_3^2$, we get the same Picard--Fuchs equation
\begin{equation}\begin{aligned}
\left[(\eta-2)\eta - \psi^4(\eta+3)(\eta+1) \right] 
\end{aligned}\end{equation}
By multiplying by $\psi$ and substituting $t = \psi^{-4}$ and $\theta = t \displaystyle{\frac{\dD}{\dD t}} = -\eta/4$, we get:
\begin{align*}
\psi\left[(\eta-2)\eta - \psi^4(\eta+3)(\eta+1) \right] (2,2,0,0) &= 0 \\
\left[(\eta-3)(\eta-1) - \psi^4(\eta+2)\eta \right] \psi(2,2,0,0) &= 0 \\
\left[(\theta+\tfrac{3}{4})(\theta+\tfrac{1}{4}) - t^{-1}(\theta-\tfrac{1}{2})\theta \right] \psi(2,2,0,0) &= 0 \\ 
\left[(\theta-\tfrac{1}{2})\theta - t(\theta+\tfrac{3}{4})(\theta+\tfrac{1}{4}) \right] \psi(2,2,0,0) &= 0
\end{align*}
which is the hypergeometric differential equation $D(\tfrac{1}{4}, \tfrac{3}{4} ; 1, \tfrac{1}{2} \,|\,\psi^{-4})$.
\end{proof}

\begin{lem}\label{F2L2 linears}
The Picard--Fuchs equation associated to the periods $(2,0,1,1)$ is the hypergeometric differential equation $D(\tfrac{1}{2} ; 1 \,|\,\psi^{4})$.
\end{lem}

\begin{proof}
For the period $(2,0,1,1)$, corresponding to the quartic monomial $x_0^2x_2x_3$, we use the diagram

$$
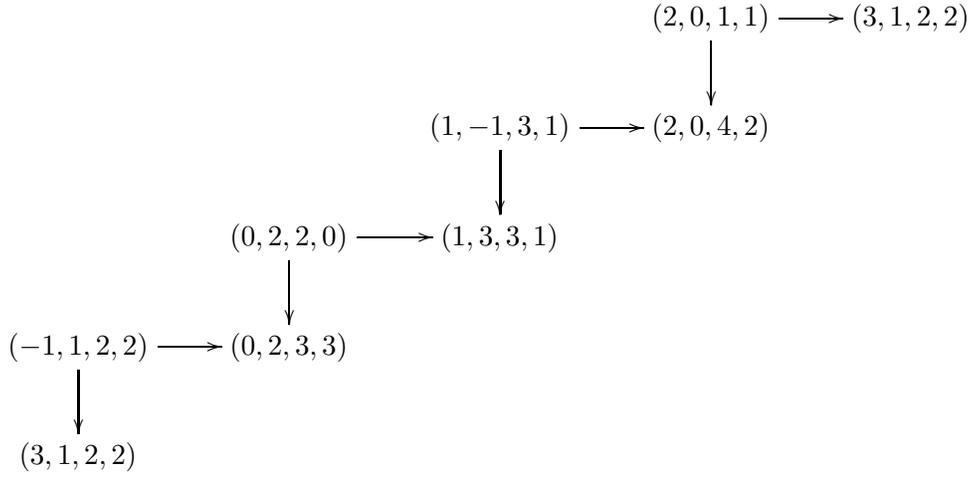
\xymatrix{ &&& (2,0,1,1) \ar[r] \ar[d] & (3,1,2,2) \\
	&& (1,-1,3,1) \ar[r] \ar[d] & (2,0,4,2) & \\
	& (0,2,2,0) \ar[r] \ar[d]  & (1,3,3,1)&& \\
	(-1,1,2,2) \ar[r] \ar[d]  & (0,2,3,3) &&& \\
	(3,1,2,2) &&&&}
$$

One can see that $(2,0,4,2) = \frac{1}{8}(2 + \eta) (2,0,1,1)$, which one can then use to show that:
\begin{equation}\begin{aligned}
\eta (2,0,1,1) &= 8 \psi^2(0,2,3,3) \\
	&= 8 \psi^3 (1,3,3,1) \\
	&= 8 \psi^4 (2,0,4,2) \\
	&= \psi^4 (\eta+2) (2,0,1,1).
\end{aligned}\end{equation}

Thus the period $(2,0,1,1)$ corresponding to the quartic monomial $x_0^2x_2x_3$ satisfies the differential equation:
\[
\left[\eta - \psi^4(\eta+2)\right](2,0,1,1)=0
\]
By substituting $u = \psi^{4}$ and $\sigma = u \displaystyle{\frac{\dD}{\dD u}} = 4\eta$, we get:
\begin{align*}
\left[\eta - \psi^4(\eta+2)\right](2,0,1,1) &=0 \\
\left[\sigma - u(\sigma+ \tfrac{1}{2})\right](2,0,1,1) &=0. \qedhere
\end{align*}
\end{proof}

\begin{lem}\label{F2L2 char1}
The Picard--Fuchs equations associated to the periods $(2,1,0,1),\psi^3(1,0,2,1)$ are the hypergeometric differential equations $D(\tfrac{1}{8}, \tfrac{5}{8} ; 1, \tfrac{1}{4} \,|\,\psi^{4}),D(\tfrac{1}{8}, \tfrac{-3}{8} ; 0, \tfrac{1}{4} \,|\,\psi^{4})$, respectively.
\end{lem}

\begin{proof}
For the period $(2,1,0,1)$, corresponding to the quartic monomial $x_0^2x_1x_3$, we use the diagram

$$
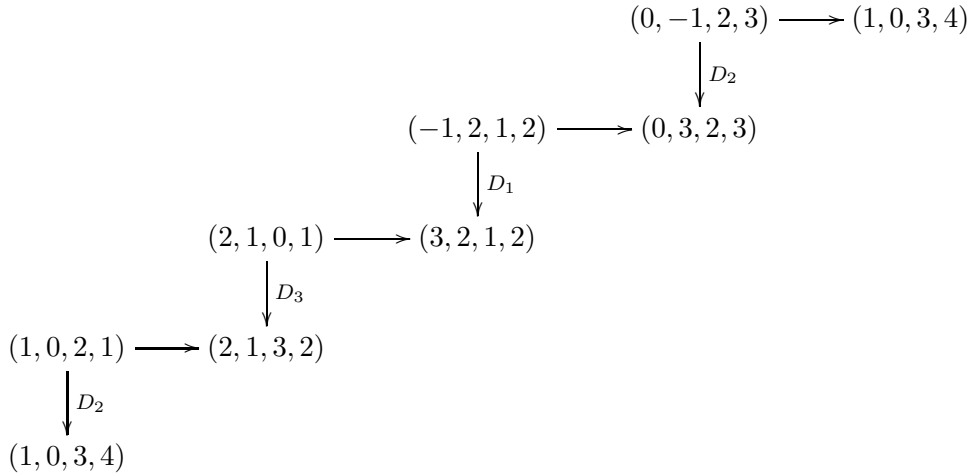
\xymatrix{ &&& (0,-1,2,3) \ar[r] \ar[d]^{D_2} & (1,0,3,4) \\
	&& (-1,2,1,2) \ar[r] \ar[d]^{D_1}& (0,3,2,3) & \\
	& (2,1,0,1) \ar[r] \ar[d]^{D_3} & (3,2,1,2)&& \\
	(1,0,2,1) \ar[r] \ar[d]^{D_2} & (2,1,3,2) &&& \\
	(1,0,3,4) &&&&}
$$

Note that:
\begin{equation}\begin{aligned}
\eta(2,1,0,1) &= 8\psi^3(1,0,3,4)\\
(1,0,3,4) &= \tfrac{1}{8}(\eta+\tfrac{3}{2}) (1,0,2,1) \\
\eta (1,0,2,1) &= \psi(\eta+\tfrac12)(2,1,0,1).
\end{aligned}\end{equation}

We can compute the two periods that satisfy each of the following Picard--Fuchs equations for the four sets of pairs:
\begin{equation}\begin{aligned}
\left[ (\eta-3) \eta - \psi^4(\eta+ \tfrac52)(\eta+ \tfrac{1}{2})\right](2,1,0,1) &= 0;\\
\left[ (\eta-1) \eta - \psi^4(\eta + \tfrac{7}{2})(\eta+\tfrac32)\right](1,0,2,1) &= 0.
\end{aligned}\end{equation}
With the first Picard--Fuchs equation, we can substitute $u = \psi^4$, $\sigma = u\textstyle{\frac{\dD}{\dD u}} = 4 \eta$, and yield the equation:
\begin{align*}
\left[ (\eta-3) \eta - \psi^4(\eta+ \tfrac52)(\eta+ \tfrac{1}{2})\right](2,1,0,1) &= 0; \\
\left[ (\sigma-\tfrac{3}{4}) \sigma - u(\sigma+ \tfrac58)(\sigma+ \tfrac{1}{8})\right](2,1,0,1) &= 0.
\end{align*}
which is the hypergeometric differential equation $D(\tfrac{1}{8}, \tfrac{5}{8} ; 1, \tfrac{1}{4} \,|\,u)$.  For the second Picard--Fuchs equation, we can multiply by $\psi^3$ and then substitute to find:
\begin{align*}
\left[ (\eta-1) \eta - \psi^4(\eta + \tfrac{7}{2})(\eta+\tfrac32)\right](1,0,2,1) &= 0\\
\psi^3 \left[ (\eta-1) \eta - \psi^4(\eta + \tfrac{7}{2})(\eta+\tfrac32)\right](1,0,2,1) &= 0\\
\left[ (\eta-4)(\eta-3) - \psi^4(\eta + \tfrac{1}{2})(\eta-\tfrac32)\right]\psi^3(1,0,2,1) &= 0\\
\left[ (\sigma-1)(\sigma-\tfrac{3}{4}) - u(\sigma + \tfrac{1}{8})(\sigma-\tfrac38)\right]\psi^3(1,0,2,1) &= 0,
\end{align*}
which is the hypergeometric function $D(\tfrac{1}{8}, \tfrac{-3}{8} ; 0, \tfrac{1}{4} \,|\,\psi^{4})$.
\end{proof}

\begin{proof}[{Proof of Proposition \textup{\ref{prop:F2L2diff}}}]
The periods annihilated by $D(\tfrac{1}{4}, \tfrac{1}{2}, \tfrac{3}{4} ; 1, 1, 1 \,|\,\psi^{-4})$ are those corresponding to the holomorphic form.  The $2$ periods are annihilated by $D(\tfrac{1}{4}, \tfrac{3}{4} ; 1, \tfrac{1}{2} \,|\,\psi^{-4})$ are provided by Lemma~ \ref{F2L2 quarter quadratic periods}.  The period annihilated by $D(\tfrac{1}{2} ; 1 \,|\,\psi^{4})$ corresponds to a monomial in the basis for $W_{\chi_2}$, which we compute in Lemma \ref{F2L2 linears}. Since $W_{\chi_2}$ and $W_{\chi_6}$ are related by a transposition of $x_0$ and $x_1$, there are two periods annihilated by the hypergeometric differential equation computed here.  The $4$ periods annihilated by $D(\tfrac{1}{8}, \tfrac{5}{8} ; 1, \tfrac{1}{4} \,|\,\psi^{4})$ and the $4$ periods annihilated by $D(\tfrac{-3}{8}, \tfrac{1}{8} ; 0, \tfrac{1}{4} \,|\,\psi^{4})$ correspond to the monomial bases for $W_{\chi_1}, W_{\chi_3}, W_{\chi_5},$ and $W_{\chi_7}$, which are the same up to transpositions. We compute in Lemma~\ref{F2L2 char1} the Picard--Fuchs equations for $W_{\chi_1}$ which then give us that each of those hypergeometric differential equations annihilate 4 periods. 
\end{proof}

\subsection{The \texorpdfstring{$\Lsf_2\Lsf_2$}{Lsf2Lsf2} pencil}

Now consider
\[
F_{\psi}  \colonequals  x_0^3x_1+x_1^3x_0 + x_2^3 x_3 + x_3^3x_2 - 4\psi x_0x_1x_2x_3
\]
that defines the pencil of projective hypersurfaces $X_{\psi} = Z(F_\psi) \subset \mathbb{P}^3$. There is a $\Z/4\Z$ symmetry with generator
\begin{align*}
g(x_0:x_1:x_2:x_3) &= (\xi x_0:\xi^5 x_1:\xi^3 x_2:\xi^7 x_3)
\end{align*}
 where $\xi$ is a primitive eighth root of unity. There are four characters $\chi_{a}: H \rightarrow \mathbb{G}_m$, where $\chi_{a}(g_1) = \ii^{a}$. We can again decompose $V$ into subspaces $W_{\chi_{a}}$.  Out of the eight, the subspaces $W_{\chi_{0}}, W_{\chi_{1}}, W_{\chi_{2}},$ and $W_{\chi_{3}}$ are empty. The monomial bases for $W_{\chi_{1}}$ and $W_{\chi_{3}}$ are related by a transposition of the variables $x_0$ and $x_1$, so their Picard--Fuchs equations are the same.  We have three types of monomial bases:
 
 \begin{enumerate}[(i)]
\item $W_{\chi_{0}}$ has monomial basis $\{x_0x_1x_2x_3, x_0^2x_2^2, x_0^2x_3^2, x_1^2x_2^2, x_1^2x_3^2\}$;
\item $W_{\chi_{1}}$ has monomial basis $\{x_0^2x_1x_2, x_1^2x_0x_3, x_2^2x_1x_3, x_3^2x_0x_2\}$; and 
\item $W_{\chi_{2}}$ has monomial basis $\{x_0^2x_1^2, x_2^2x_3^2, x_0^2x_2x_3, x_1^2x_2x_3, x_2^2x_0x_1, x_3^2x_0x_1 \}$.
\end{enumerate}

Using \eqref{DiagramPartials}, we compute the following period relations:
\begin{equation}\begin{aligned}
\mathbf{v}+(3,1,0,0) &= \frac{3(v_0+1) -(v_1+1)}{8(\omega+1)}\mathbf{v} + \psi(\mathbf{v}+(1,1,1,1)) \\
\mathbf{v}+(1,3,0,0) &= \frac{-(v_0+1) +3(v_1+1)}{8(\omega+1)}\mathbf{v} + \psi(\mathbf{v}+(1,1,1,1))
\end{aligned}\end{equation}
and the two symmetric relations replacing $0,1$ with $2,3$.

\begin{prop}\label{PFL2L2}
For the $\Lsf_2\Lsf_2$ family, the primitive cohomology group $H^2_{\textup{prim}}(X_{\Lsf_2\Lsf_2, \psi}, \C)$ has $13$ periods whose Picard--Fuchs equations are hypergeometric differential equations as follows:
\[ \begin{gathered}
\text{$3$ periods are annihilated by $D(\tfrac{1}{4}, \tfrac{1}{2}, \tfrac{3}{4} ; 1, 1, 1 \,|\,\psi^{-4})$,} \\ 
\text{$8$ periods are annihilated by $D(\tfrac{1}{8}, \tfrac{3}{8}, \tfrac58, \tfrac78 ; 0, \tfrac{1}{4}, \tfrac12, \tfrac34 \,|\,\psi^{4})$, and } \\
\text{$2$ periods are annihilated by $D(\tfrac{1}{4}, \tfrac{3}{4} ; 1,  \tfrac12 \,|\,\psi^{4})$.}
\end{gathered} \]
\end{prop}

To prove Proposition~\ref{PFL2L2}, we use the diagram method above in a few cases and then use symmetry. We first do two calculations.

\begin{lem}\label{L2L2 chi11}
The Picard--Fuchs equation associated to the periods $(2,1,1,0)$, $(1,0,1,2)$, $(1,2,0,1),$ and $(0,1,2,1)$ is the hypergeometric differential equation $D(\tfrac{1}{8}, \tfrac{3}{8}, \tfrac58, \tfrac78 ; 0, \tfrac{1}{4}, \tfrac12, \tfrac34 \,|\,\psi^{4})$.
\end{lem}

\begin{proof}
To find the Picard--Fuchs equations corresponding to all these cohomology pieces, we use the diagram
$$
\xymatrix{ &&& (2,1,1,0) \ar[r] \ar[d] & (3,2,2,1) \\
	&& (1,0,1,2) \ar[r] \ar[d] & (2,1,2,3) & \\
	& (1,2,0,1) \ar[r] \ar[d] & (2,3,1,2)&& \\
	(0,1,2,1) \ar[r] \ar[d] & (1,2,3,2) &&& \\
	(3,2,2,1) &&&&}
$$

We obtain the following relations:
\begin{equation}\begin{aligned}\label{perrelations char 1 L2L2}
\eta(1,0,1,2) &= \psi(\eta+\tfrac{1}{2}) (2,1,1,0) \\
\eta(1,2,0,1) &= \psi(\eta+\tfrac{1}{2}) (1,0,1,2) \\
\eta(0,1,2,1) &= \psi(\eta+\tfrac{1}{2}) (1,2,0,1) \\
\eta(2,1,1,0) &= \psi(\eta+\tfrac{1}{2}) (0,1,2,1) 
\end{aligned}\end{equation}

Using these relations, we can get the Picard--Fuchs equation:
\begin{equation}
\left[(\eta-3) (\eta-2)(\eta-1) \eta- \psi^4\left(\eta + \tfrac{7}{2}\right)\left(\eta + \tfrac{5}{2}\right)\left(\eta + \tfrac{3}{2} \right) \left(\eta + \tfrac{1}{2}\right) \right](1,0,1,2) = 0
\end{equation}

By substituting $u = \psi^4$ and $\sigma = u \displaystyle{\frac{\dD}{\dD u}} = \tfrac{1}{4} \eta$, we obtain:
\begin{align*}
\left[(4\sigma-3) (4\sigma-2)(4\sigma-1) 4\sigma- u\left(4\sigma + \tfrac{7}{2}\right)\left(\eta + \tfrac{5}{2}\right)\left(4\sigma + \tfrac{3}{2} \right) \left(4\sigma + \tfrac{1}{2}\right) \right](1,0,1,2)  &= 0; \\ 
\left[(\sigma-\tfrac34) (\sigma-\tfrac12)(\sigma-\tfrac14) \sigma- u\left(\sigma + \tfrac{7}{8}\right)\left(\sigma + \tfrac{5}{8}\right)\left(\sigma + \tfrac{3}{8} \right) \left(\sigma + \tfrac{1}{8}\right) \right](1,0,1,2) &= 0,
\end{align*}
which is the hypergeometric differential equation $D(\tfrac{1}{8}, \tfrac{3}{8}, \tfrac58, \tfrac78 ; 0, \tfrac{1}{4}, \tfrac12, \tfrac34 \,|\,\psi^{4})$. The other three Picard--Fuchs equations are the same due to the symmetry in \eqref{perrelations char 1 L2L2}.
\end{proof}

\begin{lem}\label{L2L2 chi20}
The Picard--Fuchs equations associated to the periods $(2,2,0,0)$ and $(0,0,2,2)$ is the hypergeometric differential equation $D(\tfrac{1}{4}, \tfrac{3}{4} ; 1,  \tfrac12 \,|\,\psi^{4})$.
\end{lem}

\begin{proof}
We use the diagram
$$
\xymatrix{ &&& (2,2,0,0) \ar[r] \ar[d] & (3,3,1,1) \\
	&& (1,1,2,0) \ar[r] \ar[d]& (2,2,3,1) & \\
	& (0,0,2,2) \ar[r] \ar[d] & (1,1,3,3)&& \\
	(2,0,1,1) \ar[r] \ar[d] & (3,1,2,2) &&& \\
	(3,3,1,1) &&&&}
$$

We obtain the following relations:
\begin{equation}\begin{aligned}
\eta(0,0,2,2) &= \psi^2(\eta + 1) (2,2,0,0);\\
\eta(2,2,0,0) &= \psi^2(\eta + 1) (0,0,2,2),
\end{aligned}\end{equation}
giving the following Picard--Fuchs equations:
\begin{equation}\begin{aligned}
\left[(\eta-2)\eta - \psi^4 \left( \eta + 3\right)\left(\eta+ 1\right)\right] (2,2,0,0) &= 0;\\
\left[(\eta-2)\eta - \psi^4 \left( \eta + 3\right)\left(\eta+ 1\right)\right] (0,0,2,2) &= 0;
\end{aligned}\end{equation}
By substituting $u = \psi^4$ and $\sigma = u\tfrac{\dD}{\dD u} = \tfrac14 \eta$, we obtain the hypergeometric form:
\begin{align*}
\left[(\sigma-\tfrac12)\sigma -u \left( \sigma + \tfrac34\right)\left(\sigma+ \tfrac14\right)\right] (2,2,0,0) = 0;\\
\left[(\sigma-\tfrac12)\sigma -u \left( \sigma + \tfrac34\right)\left(\sigma+ \tfrac14\right)\right] (0,0,2,2) = 0,
\end{align*}
which is the hypergeometric differential equation $D(\tfrac{1}{4}, \tfrac{3}{4} ; 1,  \tfrac12 \,|\,\psi^{4})$.
\end{proof}

\begin{proof}[Proof of \textup{Proposition~\ref{PFL2L2}}]
The first three periods are the same for each family.  Next, by Lemma~\ref{L2L2 chi11}, all the monomial basis elements in $W_{\chi_{1}}$ are annihilated by the hypergeometric differential equation $D(\tfrac{1}{8}, \tfrac{3}{8}, \tfrac58, \tfrac78 ; 0, \tfrac{1}{4}, \tfrac12, \tfrac34 \,|\,\psi^{4})$. Since $W_{\chi_{1}}$ and $W_{\chi_{3}}$ are related by a transposition, we get 8 periods annihilated by it. The Picard--Fuchs equations for the last two periods are given by Lemma \ref{L2L2 chi20}.
\end{proof}

\subsection{The \texorpdfstring{$\Lsf_4$}{Lsf4} pencil}

Finally we consider
\[
F_{\psi}  \colonequals  x_0^3x_1+x_1^3x_2 + x_2^3 x_3 + x_3^3x_0 - 4\psi x_0x_1x_2x_3
\]
that defines the pencil of projective hypersurfaces $X_{\psi} = Z(F_\psi) \subset \mathbb{P}^3$. There is a $H=\Z/5\Z$ scaling symmetry on $X_{\psi}$ generated by the element 
$$
g(x_0:x_1:x_2:x_3) = (\xi x_0:\xi^2x_2:\xi^4x_2: \xi^3 x_3)
$$
where $\xi$ is a fifth root of unity. There are five characters $\chi_k: H\rightarrow \C^\times$ given by $\chi_k(g) = \xi^k$. We decompose $V$ into five subspaces $W_{\chi_k}$. The monomial bases for $W_{\chi_1}, W_{\chi_2}, W_{\chi_3},$ and $W_{\chi_4}$ are related by a rotation of the variables $x_1, x_2, x_3,$ and $x_4$, so their corresponding Picard--Fuchs equations are the same. We are then left with two types of monomial bases:
\begin{enumerate}[(i)]
\item $W_{\chi_0}$ has monomial basis $\{x_0x_1x_2x_3, x_0^2x_2^2, x_1^2x_3^2\}$; and 
\item $W_{\chi_1}$ has monomial basis $\{x_0^2x_1^2, x_1^2x_2x_3, x_3^2x_0x_2, x_2^2x_0x_1\}$.
\end{enumerate}

For this family, we can compute the period relations:
\begin{equation}\begin{aligned}\label{L4PeriodRelations}
\mathbf{v} + (3,1,0,0) = \frac{27(1+v_0)- (1+v_1) +3(1+v_2) - 9(1+v_3)}{80(\omega+1)} \mathbf{v} + \psi (\mathbf{v}+(1,1,1,1))
\end{aligned}\end{equation}
and its $4$ cyclic permutations.

\begin{prop} \label{prop:L4H2result}
For the family $\Lsf_4$, the primitive cohomology group $H^2_{\textup{prim}}(X_{\Lsf_4, \psi},\C)$ has $19$ periods whose Picard--Fuchs equations are hypergeometric differential equations as follows:
\[ \begin{gathered}
\text{$3$ periods are annihilated by $D(\tfrac{1}{4}, \tfrac{1}{2}, \tfrac{3}{4} ; 1, 1, 1 \,|\,\psi^{-4})$,} \\ 
\text{$4$ periods are annihilated by $D(\tfrac{1}{5}, \tfrac{2}{5}, \tfrac{3}{5}, \tfrac{4}{5}; 1, \tfrac{1}{4}, \tfrac{1}{2}, \tfrac 34 \,|\, \psi^4)$,} \\ 
\text{$4$ periods are annihilated by  $D(\tfrac{-1}{5}, \tfrac{1}{5}, \tfrac{2}{5}, \tfrac{3}{5}; 0, \tfrac{1}{4}, \tfrac{1}{2}, \tfrac 34 \,|\,\psi^4)$,} \\ 
\text{$4$ periods are annihilated by $D( \tfrac{-2}{5}, \tfrac{-1}{5}, \tfrac{1}{5}, \tfrac{2}{5}; \tfrac{-1}{4}, 0, \tfrac{1}{4}, \tfrac{1}{2} \,|\,\psi^4)$, and} \\ 
\text{$4$ periods are annihilated by  $D(\tfrac{-3}{5}, \tfrac{-2}{5}, \tfrac{-1}{5}, \tfrac{1}{5}; \tfrac{-1}{2}, \tfrac{-1}{4}, 0, \tfrac{1}{4} \,|\,\psi^4)$.} 
\end{gathered} \]
\end{prop}

\begin{proof}
The period associated to the holomorphic form is found by the same strategy as before.

Lastly, we use \eqref{L4PeriodRelations} to construct the diagram: 
$$
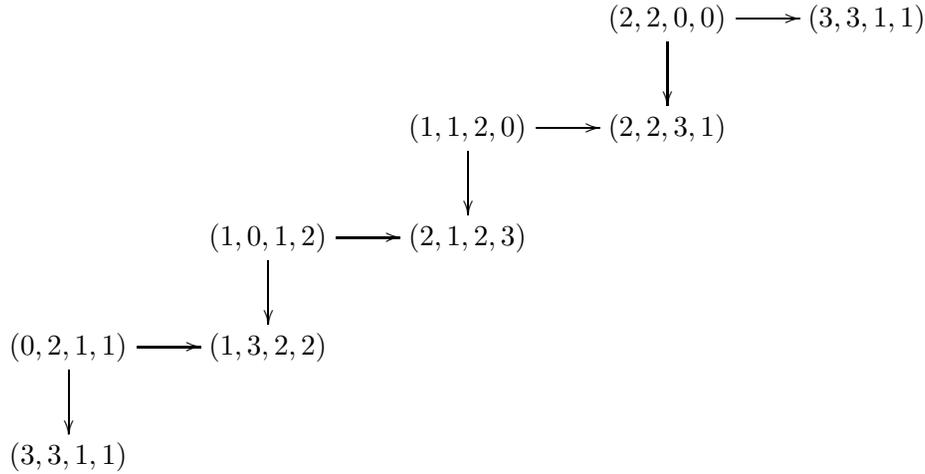
\xymatrix{ &&& (2,2,0,0) \ar[r] \ar[d] & (3,3,1,1) \\
	&& (1,1,2,0) \ar[r] \ar[d] & (2,2,3,1) & \\
	& (1,0,1,2) \ar[r] \ar[d] & (2,1,2,3)&& \\
	(0,2,1,1) \ar[r] \ar[d] & (1,3,2,2) &&& \\
	(3,3,1,1) &&&&}
$$
and consequently obtain the following relations:
\begin{equation}\begin{aligned}
\eta(1,1,2,0) &= \psi\left(\eta + \tfrac{2}{5}\right) (2,2,0,0); \\
\eta(1,0,1,2) &= \psi\left(\eta + \tfrac{1}{5}\right) (1,1,2,0); \\
\eta(0,2,1,1) &= \psi\left(\eta + \tfrac{4}{5}\right) (1,0,1,2); \\
\eta(2,2,0,0) &= \psi\left(\eta + \tfrac{3}{5}\right) (0,2,1,1).
\end{aligned}\end{equation}
We then cyclically use these relations to find a recursion which yields the following Picard--Fuchs equations:
\begin{equation}\begin{aligned}
\left[\left(\eta + \tfrac{16}{5}\right)\left(\eta + \tfrac{12}{5}\right)\left(\eta + \tfrac{8}{5}\right)\left(\eta+ \tfrac{4}{5}\right) - \tfrac{1}{\psi^4} (\eta-3)(\eta-2)(\eta-1)\eta\right](1,0,1,2) &=0 \\
\left[\left(\eta + \tfrac{17}{5}\right)\left(\eta + \tfrac{13}{5}\right)\left(\eta + \tfrac{9}{5}\right)\left(\eta+ \tfrac{1}{5}\right) - \tfrac{1}{\psi^4} (\eta-3)(\eta-2)(\eta-1)\eta\right](1,1,2,0) &= 0 \\
\left[\left(\eta + \tfrac{18}{5}\right)\left(\eta + \tfrac{14}{5}\right)\left(\eta + \tfrac{6}{5}\right)\left(\eta+ \tfrac{2}{5}\right) - \tfrac{1}{\psi^4} (\eta-3)(\eta-2)(\eta-1)\eta\right](2,2,0,0) &=0\\
\left[\left(\eta + \tfrac{19}{5}\right)\left(\eta + \tfrac{11}{5}\right)\left(\eta + \tfrac{7}{5}\right)\left(\eta+ \tfrac{3}{5}\right) - \tfrac{1}{\psi^4} (\eta-3)(\eta-2)(\eta-1)\eta\right](0,2,1,1) &=0.
\end{aligned}\end{equation}

By multiplying these lines by 1, $\psi$,  $\psi^2$, and $\psi^3$, respectively, substituting $u = \psi^{4}$, $\sigma = u \displaystyle{\frac{\dD}{\dD u}}$, and then multiplying by $-u$, we obtain the following equations:
\begin{align*}
\left[(\sigma-\tfrac34)(\sigma-\tfrac12)(\sigma-\tfrac14)\sigma-
 u\left(\sigma + \tfrac{4}{5}\right)\left(\sigma + \tfrac{3}{5}\right)\left(\sigma + \tfrac{2}{5}\right)\left(\sigma+ \tfrac{1}{5}\right) \right](1,0,1,2) &=0 \\
\left[(\sigma-1)(\sigma-\tfrac{3}{4})(\sigma-\tfrac{1}{2})(\sigma-\tfrac{1}{4})- 
 u\left(\sigma + \tfrac{3}{5}\right)\left(\sigma + \tfrac{2}{5}\right)\left(\sigma + \tfrac{1}{5}\right)\left(\sigma- \tfrac{1}{5}\right)\right]\psi(1,1,2,0) &= 0 \\
\left[ (\sigma-\tfrac{5}{4})(\sigma-1)(\sigma-\tfrac{3}{4})(\sigma-\tfrac{1}{2})
- u\left(\sigma + \tfrac{2}{5}\right)\left(\sigma + \tfrac{1}{5}\right)\left(\sigma - \tfrac{1}{5}\right)\left(\sigma- \tfrac{2}{5}\right) \right]\psi^2(2,2,0,0) &=0\\
\left[(\sigma-\tfrac{3}{2})(\sigma-\tfrac{5}{4})(\sigma-1)(\sigma-\tfrac{3}{4})-
 u\left(\sigma + \tfrac{1}{5}\right)\left(\sigma - \tfrac{1}{5}\right)\left(\sigma - \tfrac{2}{5}\right)\left(\sigma- \tfrac{3}{5}\right) \right]\psi^3(0,2,1,1) &=0.
\end{align*}
These are the claimed hypergeometric differential equations.
\end{proof}

\section{Finite field hypergeometric sums} \label{appendix:pts}

In this part of the appendix, we write down the details of manipulations of hypergeometric sums.

\subsection{Hybrid definition}

In this section, we apply the argument of Beukers--Cohen--Mellit to show that the hybrid definition of the finite field hypergeometric sum reduces to the classical one.  We retain the notation from sections \ref{defn:hgmdef}--\ref{defn:hgmdef-hybrid}.
 
\begin{lem} \label{lem:appendixhybridagree}
Suppose that $q$ is good and splittable for $\pmb{\alpha},\pmb{\beta}$.  If $\alpha_i \qq, \beta_i \qq \in \Z$ for all $i=1,\dots,d$, then Definitions \textup{\ref{classic HGF over FF}} and \textup{\ref{Hybrid Definition}} agree.
\end{lem}

\begin{proof}
Our proof follows Beukers--Cohen--Mellit \cite[Theorem 1.3]{BCM}.  We consider 
\[ G(m+\pmb{\alpha}'\qq,-m-\pmb{\beta}'\qq) = \prod_{\alpha_i' \in \pmb{\alpha}'}\frac{g(m+ \alpha_i'\qq)}{g(\alpha_i \qq)} 
\prod_{\beta_i' \in \pmb{\beta}'}\frac{g(-m-\beta_i'\qq)}{g(-\beta_i \qq)}. \]
We massage this expression, and for simplicity drop the subscripts $0$.
First, 
\[ D(x) \prod_{\alpha_j \in \hat{\pmb{\alpha}}} (x- e^{2\pi\sqrt{-1} \alpha_j}) = \prod_{j=1}^r (x^{p_j} - 1) \quad \text{and} \quad
D(x) \prod_{\beta_j \in \tilde{\pmb{\beta}}}(x- e^{2\pi\sqrt{-1} \beta_j}) = \prod_{j=1}^s (x^{q_j} - 1). \]
Write $D(x) = \prod_{j=1}^\delta (x- e^{2\pi\sqrt{-1} c_j/\qq})$.  Then
\begin{equation} 
\begin{aligned}
&G(m+\pmb{\alpha}'\qq,-m-\pmb{\beta}'\qq) \\
&\quad= \left(\prod_{i=1}^r \prod_{j=0}^{p_i-1} \frac{g(m+ j\qq\qm/p_i)}{g(j\qq\qm/p_i)}\right)\left(\prod_{i=1}^s \prod_{j=0}^{q_i-1} \frac{g(-m-j\qq\qm/q_i)}{g(-j\qq\qm/q_i)}\right) \prod_{j=1}^\delta \frac{g(c_j)g(-c_j)}{g(m+c_j) g(-m-c_j)}.
\end{aligned}
\end{equation}
Since $p_i$ divides $\qq$, by the Hasse--Davenport relation (Lemma \ref{gausssumidentities}(c)) we have that
\begin{equation}\label{HD for pi}
 \prod_{j=0}^{p_i-1} \frac{g(m+ j\qq\qm/p_i)}{g(j\qq\qm/p_i)} = \frac{-g(p_i m)}{\omega(p_i)^{p_im}}.
\end{equation}
Analogously, since $q_i$ divides $\qq$, we use Hasse--Davenport to find that 
\begin{equation}\label{HD for qi}
\prod_{j=0}^{q_i-1} \frac{g(-m-j\qq\qm/q_i)}{g(-j\qq\qm/q_i)} = -g(-q_i m) \omega(q_i)^{q_i m}
\end{equation}

Note that if $c_j \neq 0$ then $g(c_j)g(-c_j) = (-1)^{c_j}q$ and $1$ otherwise, hence $$\prod_{i=1}^\delta g(c_j)g(-c_j) = (-1)^{\sum c_j} q^{\delta - s(0)},$$ where $s(0)$ is the multiplicity of $1$ in $D(x)$, or, equivalently, the number of times $c_j$ is $0$. Now note the number of times that $m+c_j = 0$ is the multiplicity of the root $e^{-2\pi\sqrt{-1} m / \qq}$ in $D(x)$, which, equivalently, is the multiplicity of $e^{2\pi\sqrt{-1} m / \qq}$ in $D(x)$ as $D(x)$ is a product of cyclotomic polynomials. This implies that
$$
\prod_{j=1}^\delta g(m+c_j)g(-m-c_j) = (-1)^{m+c_j} q^{\delta-\lambda(m)}.
$$
We then have that 
\begin{equation}\label{cjs}
\prod_{j=1}^\delta \frac{g(c_j)g(-c_j)}{g(m+c_j) g(-m-c_j)} = \frac{ (-1)^{\sum_j c_j} q^{\delta - s(0)}}{ (-1)^{\sum_j(m+c_j)} q^{\delta-s(m)}} = (-1)^{-\delta m} q^{s(m) - s(0)}.
\end{equation}
Combining Equations~\eqref{HD for pi}, \eqref{HD for qi}, and \eqref{cjs}, we then have
\begin{align*}
&G(m+\pmb{\alpha}'\qq,-m-\pmb{\beta}'\qq) \\
& \quad = \left(\prod_{i=1}^r \frac{-g(p_i m)}{\omega(p_i)^{p_im}}\right) \left(\prod_{i=1}^s -g(-q_i m) \omega(q_i)^{q_i m}\right) \left((-1)^{-\delta m} q^{s(m) - s(0)}\right)\\
	&\quad = (-1)^{r+s}q^{s(m) - s(0)} g(p_1m)\cdots g(p_rm) g(-q_1m) \cdots g(-q_sm) \omega((-1)^{\delta}p_1^{-p_1} \cdots p_r^{-p_r} q_1^{q_1} \cdots q_s^{q_s})^m\\
	&\quad =  (-1)^{r+s}q^{s(m) - s(0)} g(p_1m)\cdots g(p_rm) g(-q_1m) \cdots g(-q_sm) \omega((-1)^{\delta}M)^m.
\end{align*}
By plugging this equation into Definition~\ref{classic HGF over FF} for the appropriate factors we obtain the quantity given in Definition~\ref{Hybrid Definition}.
\end{proof}

\subsection{The pencil \texorpdfstring{$\Fsf_2\Lsf_2$}{Fsf2Lsf2}} 

\begin{prop}\label{prop:F2L2-appendix}
The number of $\F_q$-points on $\Fsf_2\Lsf_2$ can be written in terms of hypergeometric functions, as follows:
\begin{enumalph}
\item If $q \equiv 3\pmod4$, then
\begin{align*} 
X_{\Fsf_2\Lsf_2,\psi}(\F_q)&=q^2-q+1+H_q(\tfrac{1}{4},\tfrac{1}{2},\tfrac{3}{4};0,0,0\,|\,\psi^{-4}) - q H_q(\tfrac{1}{4},\tfrac{3}{4};0,\tfrac{1}{2} \,|\,\psi^{-4}).
\end{align*}

\item If $q\equiv5\pmod8$, then
\begin{align*}
X_{\Fsf_2\Lsf_2,\psi}(\F_q)&=q^2-q+1+H_q(\tfrac{1}{4},\tfrac{1}{2},\tfrac{3}{4};0,0,0\,|\,\psi^{-4})\\
	&\qquad+qH_q(\tfrac{1}{4},\tfrac{3}{4};0,\tfrac{1}{2} \,|\,\psi^{-4})-2qH_q(\tfrac{1}{2};0\,|\,\psi^{-4}).
\end{align*}

\item If $q\equiv1\pmod8$, then
\begin{align*}X_{\Fsf_2\Lsf_2,\psi}(\F_q)&= q^2+7q +1 +H_q(\tfrac{1}{4},\tfrac{1}{2},\tfrac{3}{4};0,0,0\,|\,\psi^{-4})  +qH_q(\tfrac14, \tfrac{3}{4}; 0, \tfrac12 \,|\, \psi^{-4})  \\
	&\qquad + 2q H_q(\tfrac12; 0 \,|\, \psi^{-4}) + 2\omega(2)^{\qq\qm/4}qH_q(\tfrac18, \tfrac58; 0, \tfrac14 \,|\, \psi^4) + 2\omega(2)^{\qq\qm/4}qH_q(\tfrac38, \tfrac78; 0, \tfrac34 \,|\, \psi^4).
\end{align*}
\end{enumalph}
\end{prop}

\begin{rmk} Notice that the hypergeometric functions appearing in the point count correspond to the Picard--Fuchs equations in Proposition~\ref{prop:F2L2diff}. We also see the appearance of six additional trivial factors. 
\end{rmk}

\subsubsection*{Step 1: Computing and clustering the characters.}

To use Theorem \ref{Dels} we compute the subset $\tyS\subset (\Z/\qq\Z)^r$ given by the constraints in \eqref{WCong}. 
\begin{enumalph}
\item If $q\equiv 3 \pmod 4$ then $\tyS$ can be clustered in the following way:
\begin{enumerate}[(i)]
\item the set $\tyS_1 = \{ k(1,1,1,1,-4) : k \in \Z/\qq\Z\}$ and
\item the set $\tyS_4 = \{ k(1,1,1,1,-4) + \tfrac{\qq}{2}(0,0,1,1,0): k \in \Z/\qq\Z\}$.
\setcounter{L2L2Clusters}{\value{enumi}}
\end{enumerate}
\item If $q\equiv 5 \pmod 8$ then $\tyS$ contains the two sets above and:
\begin{enumerate}[(i)]
\setcounter{enumi}{\value{L2L2Clusters}}
\item the set $\tyS_5 = \{ k(1,1,1,1,-4) + \tfrac{\qq}{4}(0,2,1,1,0): k \in \Z/\qq\Z\}$ and
\item the set $\tyS_6 = \{ k(1,1,1,1,-4) + 3\tfrac{\qq}{4}(0,2,1,1,0): k \in \Z/\qq\Z\}$.
\setcounter{L2L2Clusters}{\value{enumi}}
\end{enumerate}
\item If $q\equiv 1 \pmod 8$ then $\tyS$ contains the four sets above and 
\begin{enumerate}[(i)]
\setcounter{enumi}{\value{L2L2Clusters}}
\item two sets of the form $S_{10} =\{ k(1,1,1,1,-4) + \tfrac{\qq}{8}(0,2,1,5,0): k \in \Z/\qq\Z\}$ and
\item two sets of the form $S_{11} =\{ k(1,1,1,1,-4) + \tfrac{\qq}{8}(0,6,7,3,0): k \in \Z/\qq\Z\}$.
\end{enumerate}
\end{enumalph}
\subsubsection*{Step 2: Counting points on the open subset with nonzero coordinates.}

\begin{lem}\label{OpenF2L2}
Suppose $\psi \in \F_q^{\times}$.

\begin{enumalph}
\item If $q \equiv 3 \pmod 4$ then
\begin{equation}
\#U_{\Fsf_2\Lsf_2, \psi}(\F_q) = q^2-3q+1+H_q(\tfrac{1}{4},\tfrac{1}{2},\tfrac{3}{4};0,0,0\,|\,\psi^{-4})- qH_q(\tfrac14, \tfrac{3}{4}; 0, \tfrac12 \,|\, \psi^{-4}).
\end{equation}
\item If $q \equiv 5 \pmod 8$ then
\begin{equation}\begin{aligned}
\#U_{\Fsf_2\Lsf_2, \psi}(\F_q) &= q^2-3q+3+H_q(\tfrac{1}{4},\tfrac{1}{2},\tfrac{3}{4};0,0,0\,|\,\psi^{-4}) + qH_q(\tfrac14, \tfrac{3}{4}; 0, \tfrac12 \,|\, \psi^{-4})\\
	&\qquad -2q H_q(\tfrac12; 0 \,|\, \psi^{-4}) - 2 \left(\frac{g(\tfrac{\qq}{4})^2 + g(\tfrac{3\qq}{4})^2}{g(\tfrac{\qq}{2})}\right).
\end{aligned}\end{equation}
\item If $q \equiv 1 \pmod 8$ then
\begin{equation}\begin{aligned}
\#U_{\Fsf_2\Lsf_2, \psi}(\F_q) &= q^2-3q+7+H_q(\tfrac{1}{4},\tfrac{1}{2},\tfrac{3}{4};0,0,0\,|\,\psi^{-4})  +qH_q(\tfrac14, \tfrac{3}{4}; 0, \tfrac12 \,|\, \psi^{-4})  \\
	&\qquad + 2q H_q(\tfrac12; 0 \,|\, \psi^{-4}) + 2\omega(2)^{\qq\qm/4}qH_q(\tfrac18, \tfrac58; 0, \tfrac14 \,|\, \psi^4) + 2\omega(2)^{\qq\qm/4}qH_q(\tfrac38, \tfrac78; 0, \tfrac34 \,|\, \psi^4) \\
	&\qquad - 2\left(\frac{g(\tfrac{\qq}{4})^2 + g(\tfrac{3\qq}{4})^2}{g(\tfrac{\qq}{2})}\right)  - \tfrac{4}{q}g(\tfrac{\qq}{8})g(\tfrac{5\qq}{8}) g(\tfrac{\qq}{4}) - \tfrac{4}{q}g(\tfrac{3\qq}{8})g(\tfrac{7\qq}{8})g(\tfrac{3\qq}{4}).
\end{aligned}\end{equation}
\end{enumalph}
\end{lem}

\begin{proof}
If $q \equiv 3\pmod 4$ then by Lemmas~\ref{Cluster no shifts} and~\ref{Cluster two half shifts}
\begin{equation}\begin{aligned}
\#U_{\Fsf_2\Lsf_2, \psi}(\F_q) &= \sum_{s \in \tyS_1} \omega(a)^{-s} c_s +  \sum_{s \in \tyS_4} \omega(a)^{-s} c_s \\
	&= (q^2-3q+3+H_q(\tfrac{1}{4},\tfrac{1}{2},\tfrac{3}{4};0,0,0\,|\,\psi^{-4}))  + \\
	&\qquad  -2 - qH_q(\tfrac14, \tfrac{3}{4}; 0, \tfrac12 \,|\, \psi^{-4}) \\
	&= q^2-3q+1+H_q(\tfrac{1}{4},\tfrac{1}{2},\tfrac{3}{4};0,0,0\,|\,\psi^{-4})- qH_q(\tfrac14, \tfrac{3}{4}; 0, \tfrac12 \,|\, \psi^{-4}).
\end{aligned}\end{equation}

If $q \equiv 5 \pmod 8$ then $q\equiv 1 \pmod4 $ but $q\not\equiv 1 \pmod 8$, so by Lemmas~\ref{Cluster no shifts},~\ref{Cluster two half shifts},~\ref{Cluster 2 quarters 1 half shifts}, and~\ref{Cluster 2 3quarter 1 half shifts}

\begin{equation}\begin{aligned}
\#U_{\Fsf_2\Lsf_2, \psi}(\F_q) &= \sum_{s \in \tyS_1} \omega(a)^{-s} c_s +  \sum_{s \in \tyS_4} \omega(a)^{-s} c_s  + \sum_{s \in \tyS_5} \omega(a)^{-s} c_s   + \sum_{s \in \tyS_6} \omega(a)^{-s} c_s  \\
	&= (q^2-3q+3+H_q(\tfrac{1}{4},\tfrac{1}{2},\tfrac{3}{4};0,0,0\,|\,\psi^{-4}))  + (2 + qH_q(\tfrac14, \tfrac{3}{4}; 0, \tfrac12 \,|\, \psi^{-4}))  \\
	&\qquad + (-1)^{\qq\qm/4} q H_q(\tfrac12; 0 \,|\, \psi^{-4}) + (-1)^{\qq\qm/4} - \left(\frac{g(\tfrac{\qq}{4})^2 + g(\tfrac{3\qq}{4})^2}{g(\tfrac{\qq}{2})}\right) \\
	&\qquad + (-1)^{\qq\qm/4} q H_q(\tfrac12; 0 \,|\, \psi^{-4}) + (-1)^{\qq\qm/4} - \left(\frac{g(\tfrac{\qq}{4})^2 + g(\tfrac{3\qq}{4})^2}{g(\tfrac{\qq}{2})}\right) \\
	&= q^2-3q+3+H_q(\tfrac{1}{4},\tfrac{1}{2},\tfrac{3}{4};0,0,0\,|\,\psi^{-4}) + qH_q(\tfrac14, \tfrac{3}{4}; 0, \tfrac12 \,|\, \psi^{-4}))\\
	&\qquad -2q H_q(\tfrac12; 0 \,|\, \psi^{-4}) - 2 \left(\frac{g(\tfrac{\qq}{4})^2 + g(\tfrac{3\qq}{4})^2}{g(\tfrac{\qq}{2})}\right).
\end{aligned}\end{equation}

If $q\equiv 1 \pmod 8$, then by Lemmas~\ref{Cluster no shifts},~\ref{Cluster two half shifts},~\ref{Cluster 2 quarters 1 half shifts},~\ref{Cluster 2 3quarter 1 half shifts},~\ref{Cluster quarter eighth 5eighths shifts}, and~\ref{Cluster 3quarter 3eighths 7eighths shifts}, we have:
\begin{equation}\begin{aligned}
\#U_{\Fsf_2\Lsf_2, \psi}(\F_q) &= \sum_{s \in \tyS_1} \omega(a)^{-s} c_s +  \sum_{s \in \tyS_4} \omega(a)^{-s} c_s  + \sum_{s \in \tyS_5} \omega(a)^{-s} c_s   + \sum_{s \in \tyS_6} \omega(a)^{-s} c_s  \\
	&\qquad + \sum_{s \in \tyS_{10}} \omega(a)^{-s} c_s + \sum_{s \in \tyS_{11}} \omega(a)^{-s} c_s \\ 
	&= (q^2-3q+3+H_q(\tfrac{1}{4},\tfrac{1}{2},\tfrac{3}{4};0,0,0\,|\,\psi^{-4}))  + (2 + qH_q(\tfrac14, \tfrac{3}{4}; 0, \tfrac12 \,|\, \psi^{-4}))  \\
	&\qquad + 2q H_q(\tfrac12; 0 \,|\, \psi^{-4}) + 2 - 2\left(\frac{g(\tfrac{\qq}{4})^2 + g(\tfrac{3\qq}{4})^2}{g(\tfrac{\qq}{2})}\right) \\
	&\qquad + 2\omega(2)^{\qq\qm/4} qH_q(\tfrac18, \tfrac58; 0, \tfrac14 \,|\, \psi^4) + 2\omega(2)^{\qq\qm/4}qH_q(\tfrac38, \tfrac78; 0, \tfrac34 \,|\, \psi^4) \\
	&\qquad - \tfrac{4}{q}g(\tfrac{\qq}{8})g(\tfrac{5\qq}{8}) g(\tfrac{\qq}{4}) - \tfrac{4}{q}g(\tfrac{3\qq}{8})g(\tfrac{7\qq}{8})g(\tfrac{3\qq}{4})\\
	&= q^2-3q+7+H_q(\tfrac{1}{4},\tfrac{1}{2},\tfrac{3}{4};0,0,0\,|\,\psi^{-4})  +qH_q(\tfrac14, \tfrac{3}{4}; 0, \tfrac12 \,|\, \psi^{-4})  \\
	&\qquad + 2q H_q(\tfrac12; 0 \,|\, \psi^{-4}) + 2\omega(2)^{\qq\qm/4} qH_q(\tfrac18, \tfrac58; 0, \tfrac14 \,|\, \psi^4) \\
	& \qquad + 2\omega(2)^{\qq\qm/4}qH_q(\tfrac38, \tfrac78; 0, \tfrac34 \,|\, \psi^4)- 2\left(\frac{g(\tfrac{\qq}{4})^2 + g(\tfrac{3\qq}{4})^2}{g(\tfrac{\qq}{2})}\right)  \\
	&\qquad  - \tfrac{4}{q}g(\tfrac{\qq}{8})g(\tfrac{5\qq}{8}) g(\tfrac{\qq}{4}) - \tfrac{4}{q}g(\tfrac{3\qq}{8})g(\tfrac{7\qq}{8})g(\tfrac{3\qq}{4})
\end{aligned}\end{equation}
as claimed.
\end{proof}

Before proving the lemmas that associate the quantities $\sum_{s \in \tyS_{10}}\omega(a)^{-s} c_s$ and $\sum_{s \in \tyS_{11}}\omega(a)^{-s} c_s$ to hypergeometric sums, we need the following lemma.

\begin{lem}\label{F2L2 gauss sum identity}
Suppose $q \equiv 1 \pmod 8$ and $q = p^r$ for some natural number $r$ and prime $p$. Then 
$$
\frac{g(\tfrac{3\qq}{4}) g(\tfrac{\qq}{8}) g(\tfrac{5\qq}{8}) } {g(\tfrac{\qq}{2})} =   \omega(2)^{\qq\qm/4} q.
$$
\end{lem}

\begin{proof}
Since $q\equiv1 \pmod 8$, we can use Hasse--Davenport with $N=2$ and $m = \tfrac{\qq}{8}$ to get that 
\begin{equation}
g(\tfrac{\qq}{4}) = \omega(2)^{\qq/4} \frac{g(\tfrac{\qq}{8})g(\tfrac{5\qq}{8})}{g(\tfrac{\qq}{2})}.
\end{equation}
By multiplying both sides by $g(\tfrac{3\qq}{4})$, and dividing by $\omega(2)^{\qq/4}$, we have
$$
\omega(2)^{-\qq/4}g(\tfrac{\qq}{4})g(\tfrac{3\qq}{4})  =\frac{g(\tfrac{\qq}{8})g(\tfrac{5\qq}{8})g(\tfrac{3\qq}{4}) }{g(\tfrac{\qq}{2})}.
$$
We obtain the identity above after noting that $g(\tfrac{\qq}{4})g(\tfrac{3\qq}{4}) = (-1)^{\qq/4} q = q$, since $q\equiv 1 \pmod 8$ and that $\omega(2)^{\qq\qm/4} = \pm 1$ when $q\equiv 1\pmod8$ hence $\omega(2)^{\qq\qm/4} = \omega(2)^{-\qq\qm/4}$.
\end{proof}

\begin{lem}\label{Cluster quarter eighth 5eighths shifts}
Suppose $q \equiv 1 \pmod 8$. Then for 
\[ \tyS_{10} = \{k(1,1,1,1,-4) + \tfrac{\qq}{8}(0,2,1,5,0) : k \in \Z/\qq\Z\} \]
we have
$$
\sum_{s \in \tyS_{10}} \omega(a)^{-s} c_s = \omega(2)^{\qq\qm/4}qH_q(\tfrac18, \tfrac58; 0, \tfrac14 \,|\, \psi^4) - \frac{1}{q}g(\tfrac{\qq}{8})g(\tfrac{5\qq}{8}) g(\tfrac{\qq}{4}) - \frac{1}{q}g(\tfrac{3\qq}{8})g(\tfrac{7\qq}{8})g(\tfrac{3\qq}{4}).
$$
\end{lem}

\begin{proof}

First, we take the definition of the sum and take out all terms that are of the form $\ell\tfrac{\qq}{4}$ to obtain the equality:

\begin{equation}\begin{aligned}
\sum_{s \in \tyS_{10}} \omega(a)^{-s} c_s &= \frac{1}{q\qq} g(\tfrac{\qq}{4}) g(\tfrac{\qq}{8}) g(\tfrac{5\qq}{8}) - \frac{1}{q\qq} g(\tfrac{\qq}{4})g(\tfrac{\qq}{2}) g(\tfrac{3\qq}{8})g(\tfrac{7\qq}{8}) \\
	&\quad - \frac{1}{q\qq} g(\tfrac{\qq}{2})g(\tfrac{3\qq}{4})g(\tfrac{\qq}{8})g(\tfrac{5\qq}{8}) + \frac{1}{q\qq} g(\tfrac{3\qq}{4}) g(\tfrac{7\qq}{8})g(\tfrac{\qq}{8}) \\
	&\quad + \frac{1}{q\qq} \sum_{\stackrel{k=0}{4k\not\equiv 0\psmod{\qq}}}^{q-2}\omega(4\psi)^{4k}g(k) g(k+\tfrac{\qq}{4}) g(k + \tfrac{\qq}{8}) g(k+\tfrac{5\qq}{8})g(-4k).
\end{aligned}\end{equation}
Next, we use the Hasse-Davenport relationship to expand $g(-4k)$ and then use relation from~\ref{gausssumidentities}(b) to cancel out the $g(k+\tfrac{\qq}{4})$ factor in the summation. Through this, we obtain:
\begin{equation}\begin{aligned}
\sum_{s \in \tyS_{10}} \omega(a)^{-s} c_s &= \frac{1}{q\qq} g(\tfrac{\qq}{4}) g(\tfrac{\qq}{8}) g(\tfrac{5\qq}{8}) - \frac{1}{q\qq} g(\tfrac{\qq}{4})g(\tfrac{\qq}{2}) g(\tfrac{3\qq}{8})g(\tfrac{7\qq}{8}) \\
	&\quad - \frac{1}{q\qq} g(\tfrac{\qq}{2})g(\tfrac{3\qq}{4})g(\tfrac{\qq}{8})g(\tfrac{5\qq}{8}) + \frac{1}{q\qq} g(\tfrac{3\qq}{4}) g(\tfrac{7\qq}{8})g(\tfrac{\qq}{8}) \\
	&\quad + \frac{1}{\qq} \sum_{\stackrel{k=0}{4k\not\equiv 0\psmod{\qq}}}^{q-2} \omega(\psi)^{4k} \frac{g(k+\tfrac{\qq}{8}) g(k+\tfrac{5\qq}{8})g(-k+\tfrac{\qq}{4})g(-k+\tfrac{\qq}{2})}{g(\tfrac{\qq}{2})}.
\end{aligned}\end{equation}

Here, we re-index the summation by $m = k+ \tfrac{\qq}{2}$ to obtain:

\begin{equation}\begin{aligned}
\sum_{s \in \tyS_{10}} \omega(a)^{-s} c_s &= \frac{1}{q\qq} g(\tfrac{\qq}{4}) g(\tfrac{\qq}{8}) g(\tfrac{5\qq}{8}) - \frac{1}{q\qq} g(\tfrac{\qq}{4})g(\tfrac{\qq}{2}) g(\tfrac{3\qq}{8})g(\tfrac{7\qq}{8}) \\
	&\quad - \frac{1}{q\qq} g(\tfrac{\qq}{2})g(\tfrac{3\qq}{4})g(\tfrac{\qq}{8})g(\tfrac{5\qq}{8}) + \frac{1}{q\qq} g(\tfrac{3\qq}{4}) g(\tfrac{7\qq}{8})g(\tfrac{\qq}{8}) \\
	&\quad + \frac{1}{q-1} \sum_{\stackrel{m=0}{4m\not\equiv 0\psmod{\qq}}}^{q-2} \omega(\psi)^{4m} \frac{g(m+\tfrac{\qq}{8}) g(m+\tfrac{5\qq}{8})g(-m-\tfrac{\qq}{4})g(-m)}{g(\tfrac{\qq}{2})}.
\end{aligned}\end{equation}

We now multiply the final form by the expression
\[ \frac{g(\tfrac{3\qq}{4}) g(\tfrac{\qq}{8}) g(\tfrac{5\qq}{8}) } {g(\tfrac{3\qq}{4}) g(\tfrac{\qq}{8}) g(\tfrac{5\qq}{8})}=1. \] 
We put the denominator of this factor into the summation to relate the summation to a hypergeometric function but factor out the numerator along with a factor of $g(\tfrac{\qq}{2})$. We then apply Lemma~\ref{F2L2 gauss sum identity} to this factor ahead of the summation. We thus obtain:
\begin{equation}\begin{aligned}
\sum_{s \in \tyS_{10}} \omega(a)^{-s} c_s &= \frac{1}{q\qq} g(\tfrac{\qq}{4}) g(\tfrac{\qq}{8}) g(\tfrac{5\qq}{8}) - \frac{1}{q\qq} g(\tfrac{\qq}{4})g(\tfrac{\qq}{2}) g(\tfrac{3\qq}{8})g(\tfrac{7\qq}{8}) \\
	&\quad - \frac{1}{q\qq} g(\tfrac{\qq}{2})g(\tfrac{3\qq}{4})g(\tfrac{\qq}{8})g(\tfrac{5\qq}{8}) + \frac{1}{q\qq} g(\tfrac{3\qq}{4}) g(\tfrac{7\qq}{8})g(\tfrac{\qq}{8}) \\
	&\quad - \frac{ \omega(2)^{\qq\qm/4} q}{\qq} \sum_{\stackrel{m=0}{4m\not\equiv 0\psmod{\qq}}}^{q-2} \omega(\psi)^{4m} \frac{g(m+\tfrac{\qq}{8}) g(m+\tfrac{5\qq}{8})g(-m-\tfrac{\qq}{4})g(-m)}{g(\tfrac{\qq}{8}) g(\tfrac{5\qq}{8})g(-\tfrac{\qq}{4})g(0)}.
\end{aligned}\end{equation}
By comparing terms of the summations above and the hypergeometric function itself, we obtain the desired result.
\end{proof}

\begin{lem}\label{Cluster 3quarter 3eighths 7eighths shifts}
Suppose $q \equiv 1 \pmod 8$. Then for 
\[ \tyS_{11} = \{k(1,1,1,1,-4) +\tfrac{\qq}{8}(0,6,7,3,0) : k \in \Z/\qq\Z\} \]
we have
$$
\sum_{s \in \tyS_{11}} \omega(a)^{-s} c_s = \omega(2)^{\qq\qm/4} qH_q(\tfrac38, \tfrac78; 0, \tfrac34 \,|\, \psi^4) - \frac{1}{q}g(\tfrac{\qq}{8})g(\tfrac{5\qq}{8}) g(\tfrac{\qq}{4}) - \frac{1}{q}g(\tfrac{3\qq}{8})g(\tfrac{7\qq}{8})g(\tfrac{3\qq}{4}).
$$
\end{lem}

\begin{proof}
The proof is analogous to the proof in Lemma~\ref{Cluster quarter eighth 5eighths shifts} except we substitute $m = k + \tfrac{\qq}{2}$.  Alternatively, apply complex conjugation to the conclusion of Lemma  \ref{Cluster quarter eighth 5eighths shifts}, negating indices as in the proof of Lemma \ref{shift by 3 fourteenths}.
\end{proof}

\subsubsection*{Step 3: Count points when at least one coordinate is zero.}

\begin{lem}\label{oneCoordZero F2L2}
The following statements hold.
\begin{enumalph}
\item If $q\equiv 3 \pmod 4$, then
$$
\#X_{\Fsf_2\Lsf_2, \psi}(\F_q) - \#U_{\Fsf_2\Lsf_2, \psi}(\F_q) = 2q.
$$
\item If $q\equiv 5 \pmod 8$, then 
$$
\#X_{\Fsf_2\Lsf_2, \psi}(\F_q) - \#U_{\Fsf_2\Lsf_2, \psi}(\F_q) = 2q-2 + 2\left(\frac{g(\tfrac{\qq}{4})^2 + g(\tfrac{3\qq}{4})^2}{g(\tfrac{\qq}{2})}\right).
$$
\item If $q \equiv 1 \pmod 8$, then 
\begin{align*}
\#X_{\Fsf_2\Lsf_2, \psi}(\F_q) - \#U_{\Fsf_2\Lsf_2, \psi}(\F_q) &= 10q - 6 + 2\left(\frac{g(\tfrac{\qq}{4})^2 + g(\tfrac{3\qq}{4})^2}{g(\tfrac{\qq}{2})}\right)+ \tfrac{4}{q}g(\tfrac{\qq}{4})g(\tfrac{\qq}{8})g(\tfrac{5\qq}{8}) \\
	&\qquad  + \tfrac{4}{q}g(\tfrac{3\qq}{4})g(\tfrac{3\qq}{8})g(\tfrac{7\qq}{8}).
\end{align*}
\end{enumalph}
\end{lem}
\begin{proof}
We do this case by case. If $x_1$ is the only variable equalling zero, then we must count the number of solutions in the open torus for the hypersurface $Z(x^4 + y^3z+ z^3y)\subset \PP^2$. We can see by using Theorem~\ref{Dels} that this depends on $q$. Here, in case (a) we get $q-1$ points, in case (b) we get $q-3 + (g(\tfrac{\qq}{4})^2 + g(\tfrac{3\qq}{4})^2)g(\tfrac{\qq}{2})^{-1}$, and in case (c) we get $q-3 + (g(\tfrac{\qq}{4})^2 + g(\tfrac{3\qq}{4})^2)g(\tfrac{\qq}{2})^{-1} + 2q^{-1}(g(\tfrac{\qq}{4})g(\tfrac{\qq}{8})g(\tfrac{5\qq}{8}) + g(\tfrac{3\qq}{4})g(\tfrac{3\qq}{8})g(\tfrac{7\qq}{8}))$. There are two such cases, when either $x_1$ or $x_2$ is the only variable equalling 0.

Next is when both $x_1$ and $x_2$ are zero and the other two variables are nonzero. Here the number of solutions is $1 + (-1)^{\qq\qm/2}$. 

Next is when $x_3$ is zero but the rest are nonzero. Here this is $(q-1)$ times the number of solutions of $Z(x^4+y^4)$ in the open torus of $\PP^1$. We then get that the number of solutions is $0$ if $q \not \equiv 1 \pmod 8$ and $4$ if $q\equiv 1 \pmod 8$, hence $4q-4$ points. There are two such cases, when $x_3$ or $x_4$ are uniquely zero. 

The next case is when $x_3$ and $x_4$ are both zero. Then the number of nonzero solutions is exactly the number of solutions of $Z(x^4+y^4)$ in the open torus of $\PP^1$, i.e., $0$ if $q \not \equiv 1 \pmod 8$ and $4$ if $q\equiv 1 \pmod 8$.

There are no rational points where $x_1$, and $x_3$ are both zero and $x_2$ nonzero and the same when you swap $x_1$ with $x_2$ or $x_3$ with $x_4$. Finally, there are two more solutions: $(0:0:1:0)$ and $(0:0:0:1)$. We now count.

\begin{enumalph}
\item If $q\not\equiv 1 \pmod 4$ and $q$ is odd, then
$$
\#X_{\Fsf_2\Lsf_2, \psi}(\F_q) - \#U_{\Fsf_2\Lsf_2, \psi}(\F_q) = 2(q-1) + 0 + 0 + 0 + 2 = 2q.
$$
\item If $q\equiv 5 \pmod 8$, then 
\begin{align*}
\#X_{\Fsf_2\Lsf_2, \psi}(\F_q) - \#U_{\Fsf_2\Lsf_2, \psi}(\F_q) &= 2\left(q-3 +\frac{g(\tfrac{\qq}{4})^2 + g(\tfrac{3\qq}{4})^2}{g(\tfrac{\qq}{2})}\right) + 2 + 0 + 0  + 2 \\
	&= 2q-2 + 2\left(\frac{g(\tfrac{\qq}{4})^2 + g(\tfrac{3\qq}{4})^2}{g(\tfrac{\qq}{2})}\right).
\end{align*}
\item If $q\equiv 1 \pmod 8$, then
\begin{align*}
\#X_{\Fsf_2\Lsf_2, \psi}(\F_q) - \#U_{\Fsf_2\Lsf_2, \psi}(\F_q) &= 2\left(q-3 +\frac{g(\tfrac{\qq}{4})^2 + g(\tfrac{3\qq}{4})^2}{g(\tfrac{\qq}{2})} + \frac{2}{q}g(\tfrac{\qq}{4})g(\tfrac{\qq}{8})g(\tfrac{5\qq}{8}) \right.\\
	&\qquad \left.+ \frac{2}{q}g(\tfrac{3\qq}{4})g(\tfrac{3\qq}{8})g(\tfrac{7\qq}{8})\right) +  2 + 2(4(q-1)) + 4 + 2 \\
	&= 10q - 6 + 2\frac{g(\tfrac{\qq}{4})^2 + g(\tfrac{3\qq}{4})^2}{g(\tfrac{\qq}{2})}+ \tfrac{4}{q}g(\tfrac{\qq}{4})g(\tfrac{\qq}{8})g(\tfrac{5\qq}{8}) \\
	&\qquad  + \tfrac{4}{q}g(\tfrac{3\qq}{4})g(\tfrac{3\qq}{8})g(\tfrac{7\qq}{8}). \qedhere
\end{align*}
\end{enumalph}
\end{proof}

\subsubsection*{Step 4: Combine Steps 2 and 3 to reach the conclusion.}

\begin{proof}[Proof of Proposition~\ref{prop:F2L2}]

We now combine Lemmas~\ref{OpenF2L2} and~\ref{oneCoordZero F2L2} as follows.  For (a), for $q \equiv 3 \pmod{4}$,
\begin{align*}
\#X_{\Fsf_2\Lsf_2, \psi}(\F_q) &= q^2-3q+1+H_q(\tfrac{1}{4},\tfrac{1}{2},\tfrac{3}{4};0,0,0\,|\,\psi^{-4})- qH_q(\tfrac14, \tfrac{3}{4}; 0, \tfrac12 \,|\, \psi^{-4}) + 2q  \\
	&= q^2-q+1+H_q(\tfrac{1}{4},\tfrac{1}{2},\tfrac{3}{4};0,0,0\,|\,\psi^{-4})- qH_q(\tfrac14, \tfrac{3}{4}; 0, \tfrac12 \,|\, \psi^{-4}).
\end{align*}
For (b), for $q\equiv 5 \pmod 8$,
\begin{align*}
\#X_{\Fsf_2\Lsf_2, \psi}(\F_q) &=q^2-3q+3+H_q(\tfrac{1}{4},\tfrac{1}{2},\tfrac{3}{4};0,0,0\,|\,\psi^{-4}) + qH_q(\tfrac14, \tfrac{3}{4}; 0, \tfrac12 \,|\, \psi^{-4})\\
	&\quad -2q H_q(\tfrac12; 0 \,|\, \psi^{-4}) - 2 \left(\frac{g(\tfrac{\qq}{4})^2 + g(\tfrac{3\qq}{4})^2}{g(\tfrac{\qq}{2})}\right) + 2q-2 + 2\left(\frac{g(\tfrac{\qq}{4})^2 + g(\tfrac{3\qq}{4})^2}{g(\tfrac{\qq}{2})}\right) \\ 
	&= q^2-q+1+H_q(\tfrac{1}{4},\tfrac{1}{2},\tfrac{3}{4};0,0,0\,|\,\psi^{-4}) + qH_q(\tfrac14, \tfrac{3}{4}; 0, \tfrac12 \,|\, \psi^{-4})-2q H_q(\tfrac12; 0 \,|\, \psi^{-4})
\end{align*}
Finally, for (c) with $q\equiv 1 \pmod 8$,
\begin{align*}
\#X_{\Fsf_2\Lsf_2, \psi}(\F_q) &= q^2-3q+5+H_q(\tfrac{1}{4},\tfrac{1}{2},\tfrac{3}{4};0,0,0\,|\,\psi^{-4})  +qH_q(\tfrac14, \tfrac{3}{4}; 0, \tfrac12 \,|\, \psi^{-4})  \\
	&\qquad + 2q H_q(\tfrac12; 0 \,|\, \psi^{-4}) + 2\omega(2)^{\qq\qm/4}qH_q(\tfrac18, \tfrac58; 0, \tfrac14 \,|\, \psi^4) + 2\omega(2)^{\qq\qm/4}qH_q(\tfrac38, \tfrac78; 0, \tfrac34 \,|\, \psi^4) \\
	&\qquad - 2\left(\frac{g(\tfrac{\qq}{4})^2 + g(\tfrac{3\qq}{4})^2}{g(\tfrac{\qq}{2})}\right)  - \tfrac{4}{q}g(\tfrac{\qq}{8})g(\tfrac{5\qq}{8}) g(\tfrac{\qq}{4}) - \tfrac{4}{q}g(\tfrac{3\qq}{8})g(\tfrac{7\qq}{8})g(\tfrac{3\qq}{4}) \\
	&\qquad + 10q - 6 + 2\left(\frac{g(\tfrac{\qq}{4})^2 + g(\tfrac{3\qq}{4})^2}{g(\tfrac{\qq}{2})}\right)+ \tfrac{4}{q}g(\tfrac{\qq}{4})g(\tfrac{\qq}{8})g(\tfrac{5\qq}{8}) \\
	&\qquad  + \tfrac{4}{q}g(\tfrac{3\qq}{4})g(\tfrac{3\qq}{8})g(\tfrac{7\qq}{8}) \\
	&= q^2+7q +1 +H_q(\tfrac{1}{4},\tfrac{1}{2},\tfrac{3}{4};0,0,0\,|\,\psi^{-4})  +qH_q(\tfrac14, \tfrac{3}{4}; 0, \tfrac12 \,|\, \psi^{-4})  \\
	&\qquad + 2q H_q(\tfrac12; 0 \,|\, \psi^{-4}) + 2\omega(2)^{\qq\qm/4}qH_q(\tfrac18, \tfrac58; 0, \tfrac14 \,|\, \psi^4)+2\omega(2)^{\qq\qm/4}qH_q(\tfrac38, \tfrac78; 0, \tfrac34 \,|\, \psi^4). \qedhere
\end{align*}
\end{proof}

\subsection{The pencil \texorpdfstring{$\Lsf_2\Lsf_2$}{Lsf2Lsf2}}

\begin{prop}\label{prop:L2L2-appendix}
The number of $\F_q$-points on $\Lsf_2\Lsf_2$ can be written in terms of hypergeometric functions, as follows:
\begin{enumalph}
\item If $q\equiv3\pmod4$, then
\begin{equation}
\#X_{\Lsf_2\Lsf_2,\psi}(\F_q) = q^2 + q+1 + H_q(\tfrac14, \tfrac12, \tfrac34; 0,0,0 \,|\, \psi^{-4}) - qH_q(\tfrac14, \tfrac{3}{4}; 0, \tfrac12 \,|\, \psi^{-4}).
\end{equation}

\item If $q\equiv1\pmod4$, then
\begin{equation}\begin{aligned}
\#X_{\Lsf_2\Lsf_2,\psi}(\F_q) &= q^2+9q+1  + H_q(\tfrac14, \tfrac12, \tfrac34; 0,0,0 \,|\, \psi^{-4}) + qH_q(\tfrac14, \tfrac{3}{4}; 0, \tfrac12 \,|\, \psi^{-4}) \\
	&\qquad + 2 (-1)^{\qq\qm/4} \omega(\psi)^{\qq\qm/2} q H_q(\tfrac{1}{8},\tfrac{3}{8},\tfrac{5}{8},\tfrac{7}{8};0,\tfrac{1}{4},\tfrac{1}{2},\tfrac{3}{4}\,|\,\psi^{-4}).
\end{aligned}\end{equation}
\end{enumalph}
\end{prop}

\begin{rmk} Again, notice that the hypergeometric functions appearing in the point count correspond to exactly one of the Picard--Fuchs equations in Proposition~\ref{PFL2L2}. We also see the appearance of eight additional trivial factors. 
\end{rmk}

 \subsubsection*{Step 1: Computing and clustering the characters}
 
 As with all the previous families, we first compute the set $\tyS$ of solutions to the system of congruences given by Theorem \ref{Dels}:
 
 \begin{enumalph} 
\item If $q\not\equiv 1\pmod4$, and $q$ is odd, then $\tyS$ consists of
\begin{enumerate}[(i)]
\item the set $\tyS_1 = \{k(1,1,1,1,-4) : k \in\Z/\qq\Z\}$ and 
\item the set $\tyS_4 = \{k(1,1,1,1,-4) + \tfrac{\qq}{2}(0,0,1,1,0) : k \in \Z/\qq\Z\}. $
\end{enumerate}

\item If $q\equiv 1 \pmod 4$, then
\begin{enumerate}[(i)]
\item the set $\tyS_1 = \{k(1,1,1,1,-4) : k \in\Z/\qq\Z\},$
\item the set $\tyS_4 = \{k(1,1,1,1,-4) + \tfrac{\qq}{2}(0,0,1,1,0) : k \in \Z/\qq\Z\}, $ and
\item two sets of the form $\tyS_{12} = \{k(1,1,1,1,-4) + \tfrac{\qq}{4}(0,2,3,1,2) : k \in \Z/\qq\Z\}$.
\end{enumerate}
\end{enumalph}

\subsubsection*{Step 2: Counting points on the open subset with nonzero coordinates}

\begin{lem}\label{open L2L2} Suppose $\psi \in \F_q^\times$. For $q$, we have:
\begin{enumalph}
\item  If $q\equiv 3\pmod4$, then
\begin{equation}
\#U_{\Lsf_2\Lsf_2, \psi}(\F_q) = q^2 -3q+1 + H_q(\tfrac14, \tfrac12, \tfrac34; 0,0,0 \,|\, \psi^{-4}) - qH_q(\tfrac14, \tfrac{3}{4}; 0, \tfrac12 \,|\, \psi^{-4}).
\end{equation}
\item If $q\equiv 1\pmod4$, then 
\begin{equation}\begin{aligned}
\#U_{\Lsf_2\Lsf_2, \psi}(\F_q) &= q^2-3q+5  + H_q(\tfrac14, \tfrac12, \tfrac34; 0,0,0 \,|\, \psi^{-4}) + qH_q(\tfrac14, \tfrac{3}{4}; 0, \tfrac12 \,|\, \psi^{-4}) \\
	&\qquad + 2 (-1)^{\qq\qm/4} \omega(\psi)^{\qq\qm/2} q H_q(\tfrac{1}{8},\tfrac{3}{8},\tfrac{5}{8},\tfrac{7}{8};0,\tfrac{1}{4},\tfrac{1}{2},\tfrac{3}{4}\,|\,\psi^{-4}).
\end{aligned}\end{equation}
\end{enumalph}
\end{lem}

\begin{proof}We do this by cases.  For (a), where $q\equiv 3 \pmod 4$, then by using Theorem~\ref{Dels} with Lemmas~\ref{Cluster no shifts} and~\ref{Cluster two half shifts}  we have that
\begin{equation}\begin{aligned}
\#U_{\Lsf_2\Lsf_2, \psi}(\F_q) &= \sum_{s \in \tyS_1} \omega(a)^{-s}c_s +  \sum_{s \in \tyS_4} \omega(a)^{-s}c_s \\
	&= q^2-3q+3  + H_q(\tfrac14, \tfrac12, \tfrac34; 0,0,0 \,|\, \psi^{-4}) -2 - qH_q(\tfrac14, \tfrac{3}{4}; 0, \tfrac12 \,|\, \psi^{-4}) \\
	&= q^2 -3q+1 + H_q(\tfrac14, \tfrac12, \tfrac34; 0,0,0 \,|\, \psi^{-4}) - qH_q(\tfrac14, \tfrac{3}{4}; 0, \tfrac12 \,|\, \psi^{-4}).
\end{aligned}\end{equation}

For (b) with $q\equiv 1 \pmod 4$, by using Theorem~\ref{Dels} with Lemmas~\ref{Cluster no shifts},~\ref{Cluster two half shifts}, and~\ref{L2L2 nonstandard shift} below, we have that
\begin{align}
\#U_{\Lsf_2\Lsf_2, \psi}(\F_q) &= \sum_{s \in \tyS_1} \omega(a)^{-s}c_s +  \sum_{s \in \tyS_4} \omega(a)^{-s}c_s +2\sum_{s \in \tyS_{12}} \omega(a)^{-s}c_s  \nonumber \\
	&= q^2-3q+3  + H_q(\tfrac14, \tfrac12, \tfrac34; 0,0,0 \,|\, \psi^{-4}) +  2 + qH_q(\tfrac14, \tfrac{3}{4}; 0, \tfrac12 \,|\, \psi^{-4}) \nonumber \\
	&\qquad + 2 (-1)^{\qq\qm/4} \omega(\psi)^{\qq\qm/2} q H_q(\frac{1}{8},\frac{3}{8},\frac{5}{8},\frac{7}{8};0,\frac{1}{4},\frac{1}{2},\frac{3}{4}\,|\,\psi^{-4}) \\
	&= q^2-3q+5  + H_q(\tfrac14, \tfrac12, \tfrac34; 0,0,0 \,|\, \psi^{-4}) + qH_q(\tfrac14, \tfrac{3}{4}; 0, \tfrac12 \,|\, \psi^{-4}) \nonumber \\
	&\qquad + 2 (-1)^{\qq\qm/4} \omega(\psi)^{\qq\qm/2} q H_q(\tfrac{1}{8},\tfrac{3}{8},\tfrac{5}{8},\tfrac{7}{8};0,\tfrac{1}{4},\tfrac{1}{2},\tfrac{3}{4}\,|\,\psi^{-4}). \nonumber \qedhere
\end{align}
\end{proof}

\begin{lem}\label{L2L2 nonstandard shift} Suppose that $q \equiv 1 \pmod 4$. Then
$$
\sum_{s \in \tyS_{12}} \omega(a)^{-s} c_s = (-1)^{\qq\qm/4} \omega(\psi)^{\qq\qm/2} q H_q(\tfrac{1}{8},\tfrac{3}{8},\tfrac{5}{8},\tfrac{7}{8};0,\tfrac{1}{4},\tfrac{1}{2},\tfrac{3}{4}\,|\,\psi^{-4}).
$$
\end{lem}
\begin{proof}

We start with the definition, factor out $\omega(\psi)^{\qq\qm/2}$, and then use the Hasse--Davenport relation~\eqref{HasseDavenport} with $N=4$ with respect to $m=k$ to obtain:
\begin{equation}\begin{aligned}
\sum_{s \in \tyS_{12}} \omega(a)^{-s} c_s &= \frac{1}{q\qq} \sum_{k=0}^{q-2} \omega(-4\psi)^{4k + \tfrac{\qq}{2}} g(k)g(k+ \tfrac{\qq}{4})g(k+\tfrac{\qq}{2})g(k+\tfrac{3\qq}{4}) g(-4k+\tfrac{\qq}{2}) \\ 
	&= \frac{\omega(\psi)^{\qq\qm/2}}{q\qq} \sum_{k=0}^{q-2} \omega(-4\psi)^{4k} g(4k)\omega(4)^{-4k} g(\tfrac{\qq}{2})g(\tfrac{\qq}{4})g(\tfrac{3\qq}{4})g(-4k+\tfrac{\qq}{2}).
\end{aligned}\end{equation}
Simplify with Lemma~\ref{gausssumidentities}(b) to get 
\begin{equation}
\sum_{s \in \tyS_{12}} \omega(a)^{-s} c_s = \frac{(-1)^{\qq\qm/4}\omega(\psi)^{\qq\qm/2}}{q-1} \sum_{k=0}^{q-2} \omega(-\psi)^{4k} g(4k)g(\tfrac{\qq}{2})g(-4k+\tfrac{\qq}{2}).
\end{equation}
Now we use the Hasse--Davenport relation again with $N=2$ and $m=-4k$ to find:
\begin{equation}
\sum_{s \in \tyS_{12}} \omega(a)^{-s} c_s = \frac{(-1)^{\qq\qm/4}\omega(\psi)^{\qq\qm/2}}{q-1} \sum_{k=0}^{q-2} \omega(-\psi)^{4k} g(4k)g(\tfrac{\qq}{2})\left(\frac{g(-8k) g(\tfrac{\qq}{2})\omega(2)^{8k}}{g(-4k)}\right).
\end{equation}

We now simplify using Lemma~\ref{gausssumidentities}(b) again and then expand the summation to get:
\begin{equation}\begin{aligned}
\sum_{s \in \tyS_{12}} \omega(a)^{-s} c_s &= (-1)^{\qq\qm/4}\omega(\psi)^{\qq\qm/2}q\left(-\frac{4}{\qq}  + \frac{1}{\qq} \sum_{\stackrel{k=0}{4k \not\equiv 0 \psmod{\qq}}}^{q-2} \omega(4\psi)^{4k} q^{-1}g(4k)^2g(-8k) \right).
	\end{aligned}\end{equation}
We finally reindex the sum with $m=-k$, yielding
\begin{equation}\begin{aligned}
\sum_{s \in \tyS_{12}} \omega(a)^{-s} c_s&= (-1)^{\qq\qm/4}\omega(\psi)^{\qq\qm/2}q\left(-\frac{4}{\qq} + \frac{1}{\qq} \sum_{\stackrel{m=0}{4m \not\equiv 0 \psmod{\qq}}}^{q-2} \omega(4\psi)^{-4m} q^{-1}g(-4m)^2g(8m) \right) \\
	&=  (-1)^{\qq\qm/4}\omega(\psi)^{\qq\qm/2}q H_q(\tfrac{1}{8},\tfrac{3}{8},\tfrac{5}{8},\tfrac{7}{8};0,\tfrac{1}{4},\tfrac{1}{2},\tfrac{3}{4}\,|\,\psi^{-4})     
\end{aligned} \end{equation} 
relating back to the finite field hypergeometric sum.
\end{proof}

\subsubsection*{Step 3: Count points when at least one coordinate is zero.}

\begin{lem}\label{oneCoordZero L2L2}
Let $q$ be an odd prime that is not $7$. Then
\begin{enumalph}
\item If $q\equiv 3 \pmod 4$, then
$$
\#X_{\Lsf_2\Lsf_2, \psi}(\F_q) - \#U_{\Lsf_2\Lsf_2, \psi}(\F_q) = 4q.
$$
\item If $q\equiv 1 \pmod 4$, then 
$$
\#X_{\Lsf_2\Lsf_2, \psi}(\F_q) - \#U_{\Lsf_2\Lsf_2, \psi}(\F_q) = 12q - 4.
$$
\end{enumalph}
\end{lem}
\begin{proof}
Suppose that $x_1=0$ and the rest are nonzero. Then, by using Theorem~\ref{Dels}, we can see that there are $(q-1)((-1)^{\qq\qm/2} +1)$ such points. Since there are four choices of one coordinate being zero, this counts $4(q-1)((-1)^{\qq\qm/2} +1)$ points.

Suppose now that $x_1=x_2=0$ and the rest nonzero, then by Theorem~\ref{Dels} again, we have $((-1)^{\qq\qm/2} +1)$ points. By symmetry, this is the same as the case where $x_3=x_4=0$ and the rest nonzero, so we now count $2((-1)^{\qq\qm/2} +1)$. 

Next, suppose $x_1=x_3=0$ and the rest nonzero. Automatically, the polynomial vanishes, hence there are $q-1$ such points. There are 4 such cases from choosing one of $x_1$ and $x_2$ and another from $x_3$ and $x_4$ to equal zero, hence we count $4(q-1)$ points. 
Finally, the four points where three coordinates are zero are all solutions, hence we count 4 more points. Thus
$$
\#X_{\Lsf_2\Lsf_2, \psi}(\F_q) - \#U_{\Lsf_2\Lsf_2, \psi}(\F_q) = 4(q-1)((-1)^{\qq\qm/2} +1) + 2((-1)^{\qq\qm/2} +1) + 4(q-1) + 4.
$$
If $q\equiv 3 \pmod 4$, then $(-1)^{\qq\qm/2} = -1$, so
$
\#X_{\Lsf_2\Lsf_2, \psi}(\F_q) - \#U_{\Lsf_2\Lsf_2, \psi}(\F_q) = 4q$.  
If $q\equiv 1 \pmod 4$, then $(-1)^{\qq\qm/2} = 1$, so 
$
\#X_{\Lsf_2\Lsf_2, \psi}(\F_q) - \#U_{\Lsf_2\Lsf_2, \psi}(\F_q) = 12q-4$.  
\end{proof}

\subsubsection*{Step 4: Combine Steps 2 and 3 to reach the conclusion}
\begin{proof}[Proof of Proposition~\ref{prop:L2L2}]
If $q \not\equiv 1 \pmod 4$ then, by Lemmas~\ref{open L2L2} and~\ref{oneCoordZero L2L2}, we have that 
\begin{equation}\begin{aligned}
\#X_{\Lsf_2\Lsf_2,\psi}(\F_q) &= (q^2 -3q+1 + H_q(\tfrac14, \tfrac12, \tfrac34; 0,0,0 \,|\, \psi^{-4}) - qH_q(\tfrac14, \tfrac{3}{4}; 0, \tfrac12 \,|\, \psi^{-4})) + 4q \\
	&= q^2 + q+1 + H_q(\tfrac14, \tfrac12, \tfrac34; 0,0,0 \,|\, \psi^{-4}) - qH_q(\tfrac14, \tfrac{3}{4}; 0, \tfrac12 \,|\, \psi^{-4}).
\end{aligned}\end{equation}
If $q \equiv 1 \pmod 4$ then, by Lemmas~\ref{open L2L2} and~\ref{oneCoordZero L2L2}, we have that 
\begin{align}
\#X_{\Lsf_2\Lsf_2,\psi}(\F_q) &= q^2-3q+5  + H_q(\tfrac14, \tfrac12, \tfrac34; 0,0,0 \,|\, \psi^{-4}) + qH_q(\tfrac14, \tfrac{3}{4}; 0, \tfrac12 \,|\, \psi^{-4}) \nonumber \\
	&\qquad + 2 (-1)^{\qq\qm/4} \omega(\psi)^{\qq\qm/2} q H_q(\tfrac{1}{8},\tfrac{3}{8},\tfrac{5}{8},\tfrac{7}{8};0,\tfrac{1}{4},\tfrac{1}{2},\tfrac{3}{4}\,|\,\psi^{-4}) + 12q-4 \nonumber \\
	&= q^2+9q+1  + H_q(\tfrac14, \tfrac12, \tfrac34; 0,0,0 \,|\, \psi^{-4}) + qH_q(\tfrac14, \tfrac{3}{4}; 0, \tfrac12 \,|\, \psi^{-4}) \\
	&\qquad + 2 (-1)^{\qq\qm/4} \omega(\psi)^{\qq\qm/2} q H_q(\tfrac{1}{8},\tfrac{3}{8},\tfrac{5}{8},\tfrac{7}{8};0,\tfrac{1}{4},\tfrac{1}{2},\tfrac{3}{4}\,|\,\psi^{-4}). \nonumber \qedhere
\end{align}
\end{proof}

\subsection{The pencil \texorpdfstring{$\Lsf_4$}{Lsf4}}

\begin{prop}\label{prop:L4-appendix}
The number of $\F_q$ points on $\Lsf_4$ for $q$ odd is given in terms of hypergeometric functions as follows.
\begin{enumalph}
\item If $q\not\equiv1\pmod5$, then
\[\#X_{\Lsf_4,\psi}(\F_q)=q^2+3q+1+H_q(\tfrac{1}{4},\tfrac{1}{2},\tfrac{3}{4};0,0,0\,|\,\psi^{-4}).\]
\item If $q\equiv 1 \pmod 5$, then
\[\#X_{\Lsf_4,\psi}(\F_q)=q^2+3q+1+H_q(\tfrac{1}{4},\tfrac{1}{2},\tfrac{3}{4};0,0,0\,|\,\psi^{-4})+4qH_q(\tfrac{1}{5},\tfrac{2}{5}, \tfrac{3}{5},\tfrac{4}{5} ;0, \tfrac{1}{4},\tfrac{1}{2}, \tfrac{3}{4}\,|\,\psi^{4}).\]
\end{enumalph}
\end{prop}

\begin{rmk} As before, we can identify the parameters of the hypergeometric function $$H_q(\tfrac{1}{5},\tfrac{2}{5}, \tfrac{3}{5},\tfrac{4}{5} ;0, \tfrac{1}{4},\tfrac{1}{2}, \tfrac{3}{4}\,|\,\psi^{4})$$ with the parameters of the second Picard--Fuchs equation in Proposition~\ref{prop:L4H2result}. If we use Theorem 3.4 of \cite{BCM} again to shift parameters,  then we see that in fact all of the Picard--Fuchs equations satisfied by the non-holomorphic periods correspond to this same hypergeometric motive over $\Q$.

Also notice that in the discussion following Proposition~\ref{prop:L4H2result}, we see two periods that are ``missed" by the Griffiths--Dwork method, and here they clearly correspond to the two additional trivial factors coming from the $3q$ term in the point count.  
\end{rmk}

\subsubsection*{Step 1: Computing and clustering the characters}

Again, we compute the solutions to the system of congruences given by Theorem \ref{Dels}. We obtain

\begin{enumalph}
\item  If $q\not\equiv 1\pmod5$, the solution set is 
\begin{enumerate}[(i)]
\item the set $\tyS_1 = \{ k(1,1,1,1,-4) : k \in \Z / \qq\Z\}$.

\end{enumerate}
\item If $q\equiv 1\pmod5$, then clusters of solutions are
\begin{enumerate}[(i)]
\item the set $\tyS_1 = \{ k(1,1,1,1,-4) : k \in \Z / \qq\Z\}$ and
\item four sets of the form $\tyS_{13} = \{ k(1,1,1,1,-4) + \frac{\qq}{5}(1,2,4,3,0) : k \in \Z / \qq\Z\}$.
\end{enumerate}
\end{enumalph}

\subsubsection*{Step 2: Counting points on the open subset with nonzero coordinates}

\begin{lem}\label{openL4} Suppose $\psi \in \F_q^\times$. For $q$ odd, we have:
\begin{enumalph}
\item  If $q\not\equiv 1\pmod5$, then
\begin{equation}
\#U_{\Lsf_4, \psi}(\F_q) = q^2-3q+3  + H_q(\tfrac14, \tfrac12, \tfrac34; 0,0,0 \,|\, \psi^{-4}).
\end{equation}
\item If $q\equiv 1\pmod5$, then 
\begin{equation}
\#U_{\Lsf_4, \psi}(\F_q) = q^2-3q+3  + H_q(\tfrac14, \tfrac12, \tfrac34; 0,0,0 \,|\, \psi^{-4}) + 4q H_q(\tfrac{1}{5},\tfrac{2}{5}, \tfrac{3}{5},\tfrac{4}{5} ;0, \tfrac{1}{4},\tfrac{1}{2}, \tfrac{3}{4}\,|\,\psi^{4}).
\end{equation}
\end{enumalph}
\end{lem}

\begin{proof}
When $q \not\equiv 1 \pmod 5$, we know that there is only one cluster of characters, $\tyS_1$. By Lemma~\ref{Cluster no shifts}, we know that 
\begin{equation}
\#U_{\Lsf_4, \psi}(\F_q) =  \sum_{s \in \tyS_1} \omega(a)^{-s} c_s=  q^2-3q+3  + H_q(\tfrac14, \tfrac12, \tfrac34; 0,0,0 \,|\, \psi^{-4}).
\end{equation}
When $q\equiv 1 \pmod 5$ we have two types of clusters of characters. By Lemmas~\ref{Cluster no shifts} and~\ref{Shifts by fifths},
\begin{align}
\#U_{\Lsf_4, \psi}(\F_q) &= \sum_{s \in \tyS_1} \omega(a)^{-s} c_s + 4 \sum_{s \in \tyS_{13}} \omega(a)^{-s} c_s \nonumber \\
	&= q^2-3q+3  + H_q(\tfrac14, \tfrac12, \tfrac34; 0,0,0 \,|\, \psi^{-4}) \\
	&\qquad + 4q H_q(\tfrac{1}{5},\tfrac{2}{5}, \tfrac{3}{5},\tfrac{4}{5} ;0, \tfrac{1}{4},\tfrac{1}{2}, \tfrac{3}{4}\,|\,\psi^{4}). \nonumber \qedhere
\end{align}
\end{proof}

We now just need a hypergeometric way to write the point count associated to the cluster $\tyS_{13}$.

\begin{lem}\label{Shifts by fifths}
If $q \equiv 1 \pmod 5$ and $q$ is odd then 
\begin{equation}
\sum_{s\in \tyS_{13}} \omega(a)^{-s} c_s = qH_q(\tfrac{1}{5},\tfrac{2}{5}, \tfrac{3}{5},\tfrac{4}{5} ;0, \tfrac{1}{4},\tfrac{1}{2}, \tfrac{3}{4}\,|\,\psi^{4}) 
\end{equation}
\end{lem}
\begin{proof}
By using the hybrid hypergeometric definition, this equality is found quickly:
\begin{equation}\begin{aligned}
\sum_{s\in \tyS_{13}} \omega(a)^{-s} c_s &=  \frac{1}{q\qq} \sum_{k=0}^{q-2} \omega(-4^4\psi^4)g(k+\tfrac{\qq}{5})g(k+\tfrac{2\qq}{5})g(k+\tfrac{3\qq}{5})g(k + \tfrac{4\qq}{5})g(-4k) \\
	&= \frac{1}{q\qq} \sum_{k=0}^{q-2} \omega(4^4\psi^4)g(k+\tfrac{\qq}{5})g(k+\tfrac{2\qq}{5})g(k+\tfrac{3\qq}{5})g(k + \tfrac{4\qq}{5})g(-4k) \\
	&= \frac{q}{\qq} \sum_{k=0}^{q-2} \omega(4^4\psi^4)\frac{g(k+\tfrac{\qq}{5})g(k+\tfrac{2\qq}{5})g(k+\tfrac{3\qq}{5})g(k + \tfrac{4\qq}{5})}{q^2}g(-4k)  \\
	&= \frac{q}{\qq} \sum_{k=0}^{q-2} \omega(4^4\psi^4) \frac{g(k+\tfrac{\qq}{5})g(k+\tfrac{2\qq}{5})g(k+\tfrac{3\qq}{5})g(k + \tfrac{4\qq}{5})}{g(\tfrac{\qq}{5})g(\tfrac{2\qq}{5})g(\tfrac{3\qq}{5})g( \tfrac{4\qq}{5})} g(-4k) \\
	&= qH_q(\tfrac{1}{5},\tfrac{2}{5}, \tfrac{3}{5},\tfrac{4}{5} ;0, \tfrac{1}{4},\tfrac{1}{2}, \tfrac{3}{4}\,|\,\psi^{4}) 
\end{aligned}\end{equation}
The last line uses the hybrid definition (Definition~\ref{Hybrid Definition}) of the hypergeometric function $$H_q(\tfrac{1}{5},\tfrac{2}{5}, \tfrac{3}{5},\tfrac{4}{5} ;0, \tfrac{1}{4},\tfrac{1}{2}, \tfrac{3}{4}\,|\,\psi^{4}).$$ Even though the hypergeometric function is defined over $\Q$, we can get to the relation much more quickly using the hybrid definition.
\end{proof}

\subsubsection*{Step 3: Count points when at least one coordinate is zero.}

\begin{lem}\label{anyCoordZero L4}
If $q$ is odd and not 7, then 
$$
\#X_{\Lsf_4,\psi}(\F_q) - \#U_{\Lsf_4,\psi}(\F_q) = 6q-2.
$$
\end{lem}
\begin{proof}
First, we count the number of rational points when exactly variable equals zero. Without loss of generality, assume $x_1=0$. Then we want solutions of 
$$
x_2^3 x_3 + x_3^3 x_4 = 0
$$
which we can solve for $x_4$. Since $x_4$ is completely determined by  $x_2$ and $x_3$, we can normalize $x_2=1$ and see there are exactly $q-1$ solutions when only $x_1$ is zero. By symmetry, this shows that there are $4q-4$ solutions when exactly one variable equals zero. If two consecutive variables are zero (say $x_1=x_2=0$) then we then want solutions of the form $x_3^3x_4 = 0$ which implies that a third variable equals zero. Thus there are 4 solutions with 3 variables equaling zero and no solutions when exactly two variables equal zero  and those variables are consecutive. Lastly, if two non-consecutive variables are zero then any other solution works. For any pair of non-consecutive variables (of which there are two), we then have $q-1$ solutions. Therefore 
$$
\#X_{\Lsf_4,\psi}(\F_q) - \#U_{\Lsf_4,\psi}(\F_q) = 4q-4 + 4 + 2(q-1) = 6q-2.
$$
\end{proof}
\subsubsection*{Step 4: Combine Steps 2 and 3 to find conclusion}

We now prove Proposition~\ref{prop:L4}.

\begin{proof}[Proof of Proposition~\ref{prop:L4}]
Combining Lemmas~\ref{openL4} and~\ref{anyCoordZero L4}, we have
\begin{enumalph}
\item  If $q\not\equiv 1\pmod5$, then
\begin{equation}\begin{aligned}
\#X_{\Lsf_4, \psi}(\F_q) &= (q^2-3q+3  + H_q(\tfrac14, \tfrac12, \tfrac34; 0,0,0 \,|\, \psi^{-4})) + (6q-2) \\
	&= q^2+3q+1  + H_q(\tfrac14, \tfrac12, \tfrac34; 0,0,0 \,|\, \psi^{-4}).
\end{aligned}\end{equation}
\item If $q\equiv 1\pmod5$, then 
\begin{equation}\begin{aligned}
\#X_{\Lsf_4, \psi}(\F_q) &= (q^2-3q+3  + H_q(\tfrac14, \tfrac12, \tfrac34; 0,0,0 \,|\, \psi^{-4}) + 4q H_q(\tfrac{1}{5},\tfrac{2}{5}, \tfrac{3}{5},\tfrac{4}{5} ;0, \tfrac{1}{4},\tfrac{1}{2}, \tfrac{3}{4}\,|\,\psi^{4})) + (6q-2) \\
	&= q^2+3q+1  + H_q(\tfrac14, \tfrac12, \tfrac34; 0,0,0 \,|\, \psi^{-4}) + 4q H_q(\tfrac{1}{5},\tfrac{2}{5}, \tfrac{3}{5},\tfrac{4}{5} ;0, \tfrac{1}{4},\tfrac{1}{2}, \tfrac{3}{4}\,|\,\psi^{4}).
\end{aligned}\end{equation}
\end{enumalph}
This completes the proof.
\end{proof}

\end{document}